\documentclass[a4paper]{amsart}

\numberwithin{equation}{section}
\usepackage[english]{babel} 
\usepackage[T1]{fontenc}
\usepackage[latin1]{inputenc}
\usepackage{enumitem}
\usepackage{amsmath,amssymb, amsbsy}
\usepackage{amsfonts}
\usepackage{hyperref}
\usepackage{latexsym}
\usepackage{amsthm}

\usepackage[usenames,dvipsnames]{xcolor}

\usepackage{pgf}
\usepackage{tikz}
\usepackage{pgfplots}
\usepackage{mathrsfs}
\usetikzlibrary{arrows}
\usetikzlibrary{decorations.pathreplacing}

\usepackage{verbatim}
 \usepackage{mathrsfs}

\usepackage[latin1]{inputenc}
\usepackage[active]{srcltx}

\usepackage[a4paper,width={16cm},left=2.5cm,bottom=3cm, top=3cm]{geometry}

\newcommand{\N}{\mathcal{N}}
\newcommand{\R}{\mathbb{R}}
\newtheorem{teo}{Theorem}[section]
\newtheorem{Corollary}[teo]{Corollary}
\newtheorem{Lemma}[teo]{Lemma}
\newtheorem{Theorem}[teo]{Theorem}
\newtheorem{Proposition}[teo]{Proposition}
\theoremstyle{definition}

\newtheorem{remark}[teo]{Remark}

\begin{document}

\title[Fractional unique continuation principle and asymptotics at the boundary]
{Strong unique continuation and local asymptotics at the boundary for fractional elliptic equations}
\date{January 13, 2022}

\author{Alessandra De Luca, Veronica Felli and Stefano Vita}

\address[A. De Luca, V. Felli]{Dipartimento di Matematica
  e Applicazioni
\newline\indent
Universit\`a degli Studi di Milano - Bicocca
\newline\indent
Via Cozzi 55, 20125, Milano, Italy}
\email{a.deluca41@campus.unimib.it, veronica.felli@unimib.it}

\address[S. Vita]{Dipartimento di Matematica ``Giuseppe Peano''
\newline\indent
Universit\`a degli Studi di Torino
\newline\indent
Via Carlo Alberto 10, 10123, Torino, Italy}
\email{stefano.vita@unito.it}

\thanks{{\it 2020 Mathematics Subject Classification:}
  31B25, 
35R11, 
35C20. 
\\
  \indent {\it Keywords:} Fractional elliptic equations; unique continuation;
monotonicity formula; boundary behavior of solutions.
}

\maketitle

\begin{abstract}
  We study local asymptotics of solutions to fractional elliptic
  equations at boundary points, under some outer homogeneous Dirichlet
  boundary condition. Our analysis is based on a blow-up procedure
  which involves some Almgren type monotonicity formul\ae \ and
  provides a classification of all possible homogeneity degrees of
  limiting entire profiles. As a consequence, we establish a strong
  unique continuation principle from boundary points.
\end{abstract}

\section{Introduction and main results}\label{sec:intr-main-results}
Let $N\geq2$ and $s\in(0,1)$. Dealing with nontrivial solutions to the following fractional equation
\begin{equation}\label{prob1}
(-\Delta)^su=hu \qquad\mathrm{in \ }\Omega
\end{equation}
where $\Omega\subset\R^N$ is a bounded domain, we are interested in a
strong unique continuation property and local asymptotics of solutions
at boundary points where the domain is locally $C^{1,1}$ and some
outer homogeneous Dirichlet boundary condition is prescribed. 

A family
of solutions to some elliptic equation is said to satisfy the
\textit{strong unique continuation property} if no element of the
family has a zero of infinite order, except for the null function.

Asymptotic expansions of solutions to fractional elliptic equations at
interior points of the domain were derived in \cite{FalFel}, even in
the presence of singular homogeneous potentials, by combining Almgren
type monotonicity formulas with blow-up arguments; as a relevant
byproduct of such sharp asymptotic analysis, in \cite{FalFel} unique
continuation principles were established.  The difficulty of defining
a suitable Almgren's type frequency function in a non-local setting
was overcome in \cite{FalFel} by considering the Caffarelli-Silvestre
extension \cite{CafSil1}, which provides an equivalent formulation of
the fractional equation as a local problem in one dimension more.  For
local problems such as second order elliptic equations, the classical
approach developed by Garofalo and Lin \cite{Garlin} allows deriving
unique continuation directly from doubling conditions obtained as a
consequence of the boundedness of an Almgren type frequency
function. In the fractional case instead, the monotonicity formula and
the doubling type conditions obtained in \cite{FalFel} imply unique
continuation properties only for the extended local problem and not
for the fractional one; then in \cite{FalFel} the further step of
classification of blow-up limits is performed in order to derive first
asymptotic estimates, and then unique continuation principles, in the
spirit of \cite{FelFerTer1,FelFerTer2}, see also \cite{FalFel2}.

Since  \cite{FalFel}, the literature devoted to unique continuation
for fractional problems has flourished producing many important results
in several directions; we mention, among others,  \cite{Ruland} for unique
continuation in presence of rough
potentials by  Carleman estimates, 
\cite{Yu} for fractional operators with variable coefficients, and
\cite{FelFer3,FelFer2,ruland-garcia,seo3,seo1,seo2,Yang} for higher
order fractional problems.

The aim of the present paper is to extend the results of
\cite{FalFel} to boundary
points of the domain, i.e. to establish sharp asymptotics and unique
continuation from boundary points for fractional equations of type
\eqref{prob1}.
Possible loss of regularity and  unavoidable interference with the geometry of the
domain make the derivation of monotonicity formulas around  boundary points,   and consequently
the proof of unique continuation, 
more difficult and, at the same time, produce  new interesting
phenomena: in particular, in \cite{Almgren-type} it was shown that, under homogeneous Dirichlet boundary
conditions,   the possible
vanishing rates of solutions  at conical 
boundary points depend of the opening of the vertex.
Unique continuation  from the boundary for
elliptic equations was also
investigated in \cite{AdoEsc,AdoEscKen,Kukavica-Nystrom,Tao}  under
homogeneous Dirichlet conditions and in
\cite{DFV,TZ05} under Neumann type conditions.
Furthermore, we refer to \cite{FallFelliFerrero} for  unique continuation
from Dirichlet-Neumann junctions  for planar mixed boundary value problems 
and to \cite{DelFel} for unique continuation from the edge of a crack.

 The related problem of regularity up to the boundary for
 solutions to fractional elliptic problems was studied  in
 \cite{RosSer1,RosSer2}. We also mention the paper \cite{BFV}, where  
quantitative upper and lower estimates at the boundary were  discussed
for nonnegative solutions to semilinear nonlocal elliptic equations,
giving us a motivation to search for sharp asymptotics at boundary
points; see also \cite{FernandezReal-RosOton} for boundary asymptotics of $s$-harmonic functions with applications to the thin one-phase problem.

In order to give a suitable weak formulation of \eqref{prob1}, we
introduce the functional space 
$\mathcal D^{s,2}(\R^N)$, defined as the completion of 
$C^\infty_c(\R^N)$ 
with respect to the scalar product
\begin{equation}\label{eq:scalar}
(u,v)_{\mathcal D^{s,2}(\R^N)}:=\int_{\R^N} |\xi|^{2s}\, \widehat
u(\xi)\overline{\widehat v(\xi)} \, d\xi
\end{equation}
and the associated norm $\|u\|_{\mathcal D^{s,2}(\R^N)}=\big((u,u)_{\mathcal D^{s,2}(\R^N)}\big)^{1/2}$,
where $\hat u$ denotes the unitary Fourier transform of $u$ in $\R^N$, i.e.
\begin{equation*}
\widehat u(\xi)=\mathcal F u(\xi):=\frac 1{(2\pi)^{N/2}}
\int_{\R^N} e^{-ix \cdot\xi}u(x)\, dx .
\end{equation*}
The fractional Laplacian $(-\Delta)^s$ can be defined as the Riesz isomorphism of
$\mathcal{D}^{s,2}(\R^N)$ with respect to the scalar product
\eqref{eq:scalar}, i.e. 
\begin{equation*}
  \phantom{a}_{( \mathcal{D}^{s,2}(\R^N) )^*}\langle(-\Delta)^s u,
  v\rangle_{ \mathcal{D}^{s,2}(\R^N) } = (u,v)_{\mathcal{D}^{s,2}(\R^N)} 
\end{equation*}
for all $u,v\in \mathcal{D}^{s,2}(\R^N)$.  Then we can define a weak
solution to \eqref{prob1} as a function $u\in \mathcal D^{s,2}(\R^N)$
satisfying
\begin{equation}\label{eq:weak}
  (u,\varphi)_{\mathcal D^{s,2}(\R^N)}=\int_\Omega h(x)u(x) \varphi(x)\,dx,
  \text{ for all }\varphi\in C^\infty_c(\Omega).
\end{equation}
As far as the potential term is concerned, we assume that 
\begin{equation}\label{eq:ass-h}
  \text{there exists }p>\frac{N}{2s} \quad\text{such that}\quad h\in W^{1,p}(\Omega).
\end{equation}
We observe that the right hand side of \eqref{eq:weak} is well defined in
view of assumption \eqref{eq:ass-h}, H\"older's inequality, and the
following well-known Sobolev-type inequality 
\begin{equation}\label{eq:sobolev}
  S_{N,s} \|u\|_{L^{2^*(s)}(\R^N)}^2\leq\|u\|^2_{\mathcal D^{s,2}(\R^N)},
\end{equation}
where $S_{N,s}$ is a positive constant depending only
on $N$ and $s$ and 
\begin{equation}\label{eq:2star}
  2^*(s)=\frac{2N}{N-2s},
\end{equation}
see \cite{CotTav}. 

Let $u\in \mathcal D^{s,2}(\R^N)$ be a solution to \eqref{eq:weak}. Let
us assume that there exists a boundary point $x_0\in\partial\Omega$
such that the boundary $\partial \Omega$ is of class $C^{1,1}$ in a
neighbourhood of $x_0$, i.e. there exist $R>0$ and
$g\in C^{1,1}(\R^{N-1})$ such that, choosing a proper coordinate
system $(x',x_N)\in\R^{N-1}\times\R$,
\begin{align}\label{eq:C11}
  &B'_{R}(x_0)\cap
    \Omega=\{(x',x_N)\in B'_{R}(x_0): x_N<g (x') \} \quad\text{ and }\\
  &\nonumber B'_{R}(x_0)\cap
\partial \Omega=\{(x',x_N)\in B'_{R}(x_0): x_N=g (x') \}, 
\end{align}
where $B_{R}'(x_0)=\{x\in \R^N:|x-x_0|<R\}$ is the ball in $\R^N$
centered at $x_0$ with radius $R$. We prescribe for the solution $u$ a
local outer
homogeneous Dirichlet boundary condition, i.e 
\begin{equation}\label{eq:DBC}
  u=0\quad\text{a.e. in }\Omega^c\cap B'_{R}(x_0). 
\end{equation}
By the extension technique introduced in \cite{CafSil1}, by adding an
additional space variable $t\in[0,+\infty)$, we can reformulate the
nonlocal problem \eqref{prob1} as a local degenerate or singular
problem on the half space 
\begin{equation*}
\R^{N+1}_+=\R^N\times(0,+\infty).
\end{equation*}
We denote the total variable $z=(x,t)\in\R^N\times(0,+\infty)$, with
$x=(x',x_N)=(x_1,...,x_{N-1},x_N)$, and define
$\mathcal{D}^{1,2}(\R^{N+1}_+,t^{1-2s}\,dz)$ as the completion of
$C^\infty_c(\overline{\R^{N+1}_+})$ with respect to the norm
\begin{equation*}
  \|U\|_{\mathcal{D}^{1,2}(\R^{N+1}_+,t^{1-2s}\,dz)}= \sqrt{\int_{\R^{N+1}_+}t^{1-2s}|\nabla U(x,t)|^2\,dx\,dt}.
\end{equation*}
It is well-known that there exists a continuous trace map 
$\mathop{\rm Tr}: \mathcal{D}^{1,2}(\R^{N+1}_+,t^{1-2s}\,dz) \to
\mathcal{D}^{s,2}(\R^N)$ 
which is onto, see \cite{brandle2013concave}.
By \cite{CafSil1}, for every $u\in \mathcal D^{s,2}(\R^N)$, the minimization problem
\begin{equation*}
  \min\left\{\|W\|_{\mathcal{D}^{1,2}(\R^{N+1}_+,t^{1-2s}\,dz)}^2 :
    \, W\in \mathcal{D}^{1,2}(\R^{N+1}_+,t^{1-2s}\,dz),\, \mathop{\rm Tr}W=u\right\}
\end{equation*}
admits a unique minimizer $U=\mathcal H(u) \in
\mathcal{D}^{1,2}(\R^{N+1}_+,t^{1-2s}\,dz)$, which
can be
obtained by convoluting $u$ with the Poisson kernel of the half-space
$\R^{N+1}_+$ and weakly solves 
\begin{equation*}
\begin{cases}
  -\text{div}(t^{1-2s} \nabla U)=0 \quad&\text{in } \R^{N+1}_+, \\
  -\lim_{t\to0^+}t^{1-2s} \partial_t U = \kappa_s
  (-\Delta)^su &\text{in } \R^N\times\{0\},
\end{cases}
\end{equation*}
where
\begin{equation*}
  \kappa_s=\frac{\Gamma(1-s)}{2^{2s-1}\Gamma(s)}>0,
\end{equation*}
i.e.
\begin{equation*}
  \int_{\R^{N+1}_+}t^{1-2s}\nabla \mathcal H(u) (x,t)\cdot\nabla
  W(x,t) \,dx\,dt
  =\kappa_s(u,\mathop{\rm Tr}W)_{\mathcal
    D^{s,2}(\R^N)}\quad\text{ for all }W \in
  \mathcal{D}^{1,2}(\R^{N+1}_+,t^{1-2s}\,dz).
\end{equation*}
As a consequence, 
$u\in \mathcal D^{s,2}(\R^N)$ is a solution \eqref{eq:weak} 
if and only if its extension $U=\mathcal H(u)$ weakly solves 
\begin{equation}\label{prob2}
\begin{cases}
-\mathrm{div}\left(t^{1-2s}\nabla U\right)=0 &\mathrm{in \ }\R^{N+1}_+,\\
\mathrm{Tr}U=u & \mathrm{in \ }\R^N\times\{0\},\\
-\lim_{t\to0^+}t^{1-2s}\partial_tU=\kappa_shu & \mathrm{in \ }\Omega\times\{0\},
\end{cases}
\end{equation}
in a weak sense, i.e 
\begin{equation}\label{eq:45}
\int_{\R^{N+1}_+}t^{1-2s}\nabla U (x,t)\cdot\nabla \phi(x,t)
\,dx\,dt=\kappa_s\int_\Omega hu\mathop{\rm Tr}\phi\,dx
\end{equation}
for every $\phi\in C^\infty_c(\overline{\R^{N+1}_+})$ with $\mathop{\rm Tr}\phi\in C^\infty_c(\Omega)$.

The asymptotic behaviour at $x_0\in\partial\Omega$ of solutions to
\eqref{prob2}, and consequently to \eqref{prob1}, will turn out to be
related to eigenvalues and eigenfunctions of the following weighted
spherical eigenvalue problem with mixed Dirichlet-Neumann boundary
conditions
\begin{equation}\label{eig}
\begin{cases}
  -\mathrm{div}_{\mathbb S^N}\left(\theta_{N+1}^{1-2s}
    \nabla_{\mathbb S^N}\psi\right)=\theta_{N+1}^{1-2s}\mu\psi &\mathrm{in} \ \mathbb S^N_+,\\
  \psi=0 &\mathrm{on} \ \mathbb S^{N-1}\cap\{\theta_N\geq0\},\\
  \lim_{\theta_{N+1}\to0^+}\theta_{N+1}^{1-2s}\nabla_{\mathbb
    S^N}\psi\cdot \nu=0 &\mathrm{on} \ \mathbb
  S^{N-1}\cap\{\theta_N<0\},
\end{cases}
\end{equation}
on the half-sphere 
\begin{equation*}
{\mathbb S}^N_+=
\{(\theta_1,\dots,\theta_N, \theta_{N+1})\in
{\mathbb S}^{N}:\theta_{N+1}>0\},
\end{equation*}
where $\nu=(0,0,\dots,0,-1)$
and $\partial {\mathbb S}^{N}_+={\mathbb S}^{N-1}\times\{0\}$ is identified with
${\mathbb S}^{N-1}$. In order to write the variational formulation of
  \eqref{eig}, we define 
$H^{1}({\mathbb S}^{N}_+,\theta_{N+1}^{1-2s}dS)$ as the completion of
$C^\infty(\overline{{\mathbb S}^{N}_+})$ with respect to the norm
\begin{equation*}
\|\psi\|_{H^{1}({\mathbb S}^{N}_+,\theta_{N+1}^{1-2s}dS)}=\bigg(
\int_{{\mathbb S}^{N}_+}\theta_{N+1}^{1-2s}\big(|\nabla_{{\mathbb
S}^{N}}\psi(\theta)|^2+\psi^2(\theta)\big)dS
\bigg)^{\!\!1/2},
\end{equation*}
where $dS$ denotes the volume element on $N$-dimensional spheres.  Let
$\mathcal H_0$ be the closure of $C^\infty_{\rm c}(\overline{{\mathbb S}^{N}_+}\setminus~\!\!S_1^+)$ in
$H^{1}({\mathbb S}^{N}_+,\theta_{N+1}^{1-2s}dS)$, where
$S^+_{1}=\{(\theta',\theta_N,0)\in \mathbb S^{N-1}:\theta_N\geq 0\}$.
We say that $\mu\in\R$ is an eigenvalue of \eqref{eig} if there exists
$\psi\in \mathcal H_0 \setminus\{0\}$ such that
\begin{equation}\label{defautoval}
\int_{\mathbb S^N_+}\theta_{N+1}^{1-2s}\nabla_{\mathbb
  S^N}\psi\cdot\nabla_{\mathbb S^N}\phi\,dS=
\mu\int_{\mathbb S^N_+}\theta_{N+1}^{1-2s}\psi\phi\,dS\quad\text{for
  any }\phi\in\mathcal H_0.
\end{equation}
By classical spectral theory,  problem \eqref{eig} admits a diverging
    sequence of real eigenvalues with finite multiplicity
    $\{\mu_k\}_{k\geq0}$. In Appendix \ref{sec:color-eigenv-probl}
we obtain the following explicit formula for such eigenvalues
\begin{equation}\label{eq:28}
  \mu_k=(k+s)(k+N-s), \quad k\in \mathbb{N}.
\end{equation}
For all $k\in\mathbb N$, let $M_k\in\mathbb N\setminus\{0\}$ be the
multiplicity of the eigenvalue $\mu_k$ and
$\{Y_{k,m}\}_{m=1,2,...,M_k}$ be a
$L^2(\mathbb S^N_+,\theta_{N+1}^{1-2s}dS)$-orthonormal basis of the
eigenspace of problem \eqref{eig} associated to $\mu_k$. In
particular,
\begin{equation}\label{eq:orthobasis}
\{Y_{k,m}  : k\in\mathbb N, \ m=1,...,M_k\}
\end{equation}
is an orthonormal basis of $L^2(\mathbb S^N_+,\theta_{N+1}^{1-2s}dS)$.

\begin{remark}\label{rem:auto-non-van} 
It is worth highlighting the fact that
    eigenfunctions of problem \eqref{eig} cannot vanish 
    identically on $\mathbb S^{N-1}\cap\{\theta_N<0\}$, i.e. on the
    boundary portion where a Neumann homogeneous condition is imposed. Indeed, if an eigenfunction $\psi$
    associated to the eigenvalue $\mu_k=(k+s)(k+N-s)$ vanishes on $\mathbb S^{N-1}\cap\{\theta_N<0\}$, then the function
$\Psi(\rho\theta)=\rho^{k+s}\psi(\theta)$ would be
    a weak solution to the equation ${\rm div}(t^{1-2s}\nabla \Psi)=0$ in
    $\R^{N-1}\times(-\infty,0)\times (0,+\infty)$ satisfying both Dirichlet and weighted Neumann
    homogeneous boundary conditions on $\R^{N-1}\times(-\infty,0)\times \{0\}$; then its trivial extension to
    $\R^{N-1}\times(-\infty,0)\times \R$
 would violate the unique continuation
    principle for elliptic equations with Muckenhoupt weights proved
    in \cite{Tao} (see also \cite{Garlin}, \cite[Corollary 3.3]{STT}, and
    \cite[Proposition 2.2]{Ruland}).
  \end{remark}

Our first result is a sharp description of the
asymptotic behaviour  of solutions to \eqref{prob1} at a boundary point.

\begin{Theorem}\label{t:asymp-u}
  Let
 $\Omega$ be a bounded domain in $\R^N$ such that there exist $g\in
 C^{1,1}(\R^{N-1})$, 
 $x_0\in \partial\Omega$ and $R>0$ satisfying 
\eqref{eq:C11}. Let $h$ satisfy \eqref{eq:ass-h} and
$u\in \mathcal D^{s,2}(\R^N)$, $u\not\equiv0$,  be a weak solution to \eqref{prob1} in the
sense of \eqref{eq:weak}, satisfying \eqref{eq:DBC}. Then
there
  exists $k_0\in {\mathbb N}$ and an eigenfunction $Y$ of problem
  \eqref{eig} associated to the eigenvalue
  $\mu_{k_0}=(k_0+s)(k_0+N-s)$ such that
 \begin{equation*}
\frac{u(x_0+\lambda x)}{\lambda^{k_0+s}}\rightarrow
|x|^{k_0+s}Y\left(\tfrac{x}{|x|},0\right)\quad \text{in $H^s(B_1')$ as $\lambda\rightarrow 0^+$},
\end{equation*}
 where
$H^s(B_1')$ is the usual fractional Sobolev space on the
$N$-dimensional unit ball $B_1'=B_1' (0)$.
\end{Theorem}

Theorem \ref{t:asymp-u} will be proved as a consequence of the
following description of the
 asymptotic behaviour
of nontrivial solutions to \eqref{prob2} near $x_0\in\partial\Omega$.

\begin{Theorem}\label{t:asymp-U}
  Let $\Omega$ be a bounded domain in $\R^N$ such that there exist
  $g\in C^{1,1}(\R^{N-1})$, $x_0\in \partial\Omega$ and $R>0$
  satisfying \eqref{eq:C11}. Let $h$ satisfy \eqref{eq:ass-h} and
  $U\in \mathcal{D}^{1,2}(\R^{N+1}_+,t^{1-2s}\,dz)$ be a weak solution
  to \eqref{prob2} in the sense of \eqref{eq:45}, with $U\not\equiv0$
  and $\mathop{\rm Tr} U=u$ satisfying \eqref{eq:DBC}.  Then there
  exists $k_0\in {\mathbb N}$ and an eigenfunction $Y$ of problem
  \eqref{eig} associated to the eigenvalue
  $\mu_{k_0}=(k_0+s)(k_0+N-s)$ such that, letting $z_0=(x_0,0)$,
 \begin{equation*}
\frac{U(z_0+\lambda z)}{\lambda^{k_0+s}}\rightarrow
|z|^{k_0+s}Y(z/|z|)\quad \text{in $H^1(B_1^+,t^{1-2s}dz)$ as $\lambda\rightarrow 0^+$},
\end{equation*}
 where
$H^1(B_1^+,t^{1-2s}dz)$ is the weighted Sobolev space on the half ball
$B_1^+$ defined in Section \ref{sec:funct-sett-constr}.
 \end{Theorem}

Actually the proof of Theorem \ref{t:asymp-U} contains a more precise
characterization of the angular limit profile $Y$, as a linear
combination of the orthonormalized eigenfunctions
$\{Y_{k_0,m}\}_{m=1,2,...,M_{k_0}}$ of \eqref{eig}  associated to the
  eigenvalue $\mu_{k_0}$ with
coefficients explicitly given by formula \eqref{eq:coeff-beta}.

The salient consequence of the precise asymptotic expansions described
above is the
following \emph{strong unique continuation principle} for problems
\eqref{prob1} and \eqref{prob2}, whose proof follows straightforwardly
from Theorems \ref{t:asymp-u} and \ref{t:asymp-U}, taking into account
Remark \ref{rem:auto-non-van}.

\begin{Corollary}\label{c:unique_continuation}\quad \\[-7pt]  
\begin{itemize}
  \item[\rm (i)] Under the same assumptions as in Theorems \ref{t:asymp-u}, let
    $u\in \mathcal D^{s,2}(\R^N)$ be a weak solution to \eqref{prob1}
    (in the sense of \eqref{eq:weak}) satisfying \eqref{eq:DBC} and
    such that $u(x)=O(|x-x_0|^k)$ as $x\to x_0$, for any
    $k\in {\mathbb N}$.  Then $u\equiv 0$ in $\R^N$.
  \item[\rm (ii)] Under the same assumptions as in Theorems \ref{t:asymp-U}, let
    $U\in \mathcal{D}^{1,2}(\R^{N+1}_+,t^{1-2s}\,dz)$ be a weak
    solution to \eqref{prob2} (in the sense of \eqref{eq:45}) with
     $\mathop{\rm Tr} U=u$ satisfying
    \eqref{eq:DBC} and such that $U(z)=O(|z-z_0|^k)$ as
    $z\to z_0$, for any $k\in {\mathbb N}$.  Then $U\equiv 0$ in $\R^{N+1}_+$.
  \end{itemize}
\end{Corollary}

The proof of Theorem \ref{t:asymp-U} makes use of the procedure developed
in \cite{FalFel,FelFerTer1,FelFerTer2} and consisting in a fine
blow-up analysis
for scaled solutions based on sharp energy estimates, obtained as a consequence of the existence
of the limit of an Almgren type frequency function. Such method
is applied to an equivalent auxiliary problem 
obtained by straightening the boundary of the domain
$\Omega$ through a 
diffeomorphic deformation, which is inspired by \cite{AdoEsc}
 and built specifically to ensure that the extended equation is conserved by reflection
through a straightened vertical boundary, see Section \ref{sect2.1}.

Significant additional difficulties arise
with respect to the non-fractional or interior setting.  
 Since
the optimal regularity of solutions  to
\eqref{prob1} is $s$-H\"older continuity, see
\cite{RosSer2},  
solutions to
\eqref{prob2} could have singular gradient at $\partial\Omega$,
which represents for problem \eqref{prob2} the interface between
mixed Dirichlet and Neumann boundary conditions. 
In fact, problems with mixed boundary conditions
raise delicate regularity issues, which turn out to be more difficult
in dimension $N\geq2$ due to the positive dimension of the junction
set and some role played by the geometry of the domain. 

In particular,  we remark that 
the regularity results known for  non-fractional
problems (for which solutions are smooth up to the boundary) or for
interior points (see e.g. \cite{JLX}) are not available here, even
excluding a neighborhood  of the boundary point if $N\geq2$.
Therefore, for problems that, like
\eqref{prob2}, are characterized by mixed
boundary conditions in dimension $N\geq2$, 
the development of a monotonicity argument around points located at 
Dirichlet-Neumann junctions
presents substantial new difficulties with respect to both the case
treated in \cite{FalFel} of boundary points
around which a Neumann condition is given and the 
case treated in \cite{FallFelliFerrero} of mixed conditions in
dimension $N=1$. 

In the present paper, the difficulties related to lack of regularity at Dirichlet-Neumann
junctions are overcome by a double approximation procedure: by
approximating the potential $h$ with potentials vanishing near the
boundary and the Dirichlet $N$-dimensional region with smooth
$(N+1)$-sets with straight vertical boundary, we will be able to
construct a sequence of approximating solutions (see Section \ref{sec:constr-appr-sequ}) which enjoy enough
regularity to derive Pohozaev type identities, needed to obtain 
Almgren type monotonicity formul\ae \ and consequently to perform blow-up analysis.
We mention that a similar approximation procedure was
developed in \cite{DelFel} for a class of elliptic equations in a domain with a crack.

The paper is structured as follows. In Section
\ref{sec:funct-sett-constr}, after introducing a suitable functional
setting for the study of the extended problem \eqref{prob2}, we present an
equivalent auxiliary problem obtained by straightening the boundary;
then, after providing some Hardy-Sobolev and coercivity type inequalities, in
Subsection \ref{sec:constr-appr-sequ} we perform the approximation
procedure which allows us to establish a Pohozaev type
inequality. Section \ref{sec:almgr-type-freq} is devoted to the proof
of an Almgren type monotonicity formula, which is the key tool for the
blow-up analysis carried out in Section
\ref{sec:blow-up-analysis}. Finally, in Appendix
\ref{sec:some-bound-regul} we present some  boundary regularity
results for singular/degenerate equations in cylinders, while in Appendix
  \ref{sec:color-eigenv-probl} we prove \eqref{eq:28}, through a
classification of  possible homogeneity degrees of homogeneous 
solutions to \eqref{eq:27}.

\section{Functional setting and construction of the approximating domains}\label{sec:funct-sett-constr}
Let us call the total variable $z=(x,t)\in\R^N\times\R$, with
$x=(x',x_N)=(x_1,...,x_{N-1},x_N)$. 
We set the following notations for all $r>0$:
\begin{align*}
&B_r=\{z\in\R^{N+1}:|z|<r\},\quad B_r^+=B_r\cap
\R^{N+1}_+,\\
&B'_r=\{(x,t)\in B_r:t=0\},\quad \partial^+B^+_r=\partial
B_r\cap \R^{N+1}_+.
\end{align*}
In the sequel $B'_r$ will be sometimes identified with the ball in $\R^N$
centered at $0$ with radius $r$.
The weighted Sobolev space $H^1(B_r^+,t^{1-2s}\,dz)$ in the extension context is defined 
as the completion of 
$C^\infty(\overline{B_r^+})$ with respect to the norm
\begin{equation*}
\|U\|_{H^1(B_r^+,t^{1-2s}\,dz)}=\sqrt{\int_{B_r^+}t^{1-2s}\left(|U|^2+|\nabla
U|^2\right)\,dz}.
\end{equation*}
It is well known, see e.g. \cite[Proposition 2.1]{JLX}, that there
exists a well-defined continuous trace map 
$\mathop{\rm Tr}: H^1(B_r^+,t^{1-2s}\,dz) \to L^{2^*(s)}(B_r')$; in
particular there exists a positive constant $C_{N,s}$ depending only
on $N$ and $s$ such that, for all $r>0$ and $U\in H^1(B_r^+,t^{1-2s}\,dz)$, 
\begin{equation}\label{eq:traccia-H1BR}
\|\mathop{\rm Tr}(U)\|_{L^{2^*(s)}(B_r')}^2\leq C_{N,s}
\int_{B_r^+}t^{1-2s}\left(r^{-2}|U(z)|^2+|\nabla
U(z)|^2\right)\,dz.
\end{equation}
We are interested in boundary qualitative properties of solutions to
\eqref{prob1} close to a fixed point
$x_0\in\partial\Omega$. Without loss of generality, up
to translation and rotation, we can assume $x_0=0$ and 
consider the extension $U=\mathcal H(u)$ which, under
assumptions \eqref{eq:C11} and \eqref{eq:DBC}, solves
\begin{equation}\label{prob3}
\begin{cases}
-\mathrm{div}\left(t^{1-2s}\nabla U\right)=0 &\mathrm{in \ } B_R^+,\\
-\lim_{t\to0^+}t^{1-2s}\partial_tU=\kappa_shu &\mathrm{in \ } \Gamma_{
g,R}^-:=\{(x',x_N,0)\in B_R':x_N< g(x')\},\\
U=0 & \mathrm{in \ } \Gamma_{ g,R}^+:=
\{(x',x_N,0)\in B_R':x_N\geq g(x')\},
\end{cases}
\end{equation}
for some $R>0$, where $g$ is the function which
locally parametrizes the boundary $\partial\Omega$ around $x_0=0$
according to \eqref{eq:C11}, with 
\begin{equation}\label{hpdellag}
g\in C^{1,1}(\mathbb{R}^{N-1}),\quad g(0)=0\quad \text{and}\quad\nabla g(0)=0.
\end{equation}
The suitable weighted Sobolev space for energy solutions to
\eqref{prob3} is $H^1_{\Gamma_{ g,R}^+}(B_R^+,t^{1-2s}\,dz)$, defined
as the closure of
$C^\infty_c(\overline{B_R^+}\setminus \Gamma_{ g,R}^+)$ in
$H^1(B_R^+,t^{1-2s}\,dz)$.  By energy solution to \eqref{prob3} we
mean a function $U\in H^1_{\Gamma_{ g,R}^+}(B_R^+,t^{1-2s}\,dz)$ such
that
\begin{equation*}
\int_{B_R^+}t^{1-2s}\nabla
U(x,t)\cdot\nabla\phi(x,t)\,dz-\kappa_s\int_{\Gamma_{
g,R}^-}h\mathop{\rm Tr}U\mathop{\rm Tr}\phi\,dx=0\quad
\text{for all }\phi\in C^\infty_c(B_R^+\cup \Gamma_{ g,R}^-).
\end{equation*}

\subsection{A diffeomorphism to straighten the boundary}\label{sect2.1}
We follow the construction in \cite{AdoEsc}. Let us consider the
following set of variables $(y,t)\in\R^N\times[0,+\infty)$, with
$y=(y',y_N)=(y_1,...,y_{N-1},y_N)$. Let $\rho\in C^\infty_c(\R^{N-1})$
be such that $\rho\geq0$, $\mathrm{supp}(\rho)\subset B_1'$ and
$\int_{\R^{N-1}}\rho(y')\,dy'=1$.  For every $\delta>0$ we define
\begin{equation*}
\rho_\delta(y')=\delta^{-N+1}\rho\left(\frac{y'}{\delta}\right).
\end{equation*}
Let us define also, for every $j=1,...,N-1$,
\begin{equation*}
G_j(y',y_N) = \begin{cases}
\left( \rho_{y_N}\ast \partial_{y_j}g \right)(y')&\text{if }y'\in\R^{N-1},\
y_N>0,\\[5pt]
\partial_{y_j}g (y')&\text{if } y'\in\R^{N-1},\
y_N=0,
\end{cases}
\end{equation*}
where $\ast$ denotes the convolution product.

It is easy to verify that, for all $j=1,\dots,N-1$,
$G_j\in C^\infty(\R^N_+)$, $G_j$ is Lipschitz continous in
$\overline{\R^N_+}$, and
$\frac{\partial G_j}{\partial y_i}\in L^\infty(\R^N_+)$ for every
$i\in\{1,\dots,N\}$.  Moreover, for all $j=1,\dots,N-1$ and
$i=1,\dots,N$,
\begin{equation*}
y_N \frac{\partial
G_j}{\partial y_i}\quad\text{is Lipschitz continuous in
$\overline{\R^N_+}$}.
\end{equation*}
As a consequence, we have that, letting 
\begin{equation*}
\widetilde G_j:\R^N\to\R,\quad \widetilde G_j(y',y_N):=G_j(y',|y_N|)
\end{equation*}
and \begin{equation*}
\psi_j:\R^N\to\R,\quad \psi_j(y',y_N)=y_j-y_N\widetilde G_j(y',y_N),
\end{equation*}
$\widetilde G_j$ is Lipschitz continuous in
$\R^N$
and $\psi_j\in C^{1,1}(\R^N)$ (i.e. $\psi_j$ is continuously
differentiable with Lipschitz gradient) for all $j=1,\dots,N-1$. Let 
\begin{equation*}
\widetilde G(y',y_N)=( \widetilde G_1(y',y_N), \widetilde G_2(y',y_N),\dots,
\widetilde G_{N-1}(y',y_N))
\end{equation*}
and denote as $J_{\widetilde G}(y',y_N)$ the Jacobian matrix of
$\widetilde G$ at $(y',y_N)$. Then $J_{\widetilde G}\in L^\infty(\R^N,\R^{N(N-1)})$ 
and 
\begin{equation}\label{eq:3}
|\widetilde G(y',y_N)-\nabla g(y')|\leq C\,
|y_N|\quad\text{for all }(y',y_N)
\in \R^N,
\end{equation}
for some constant $C>0$ independent of $(y',y_N)$.

Let us consider the local diffeomorphism $F\colon \mathbb{R}^{N+1}\to \mathbb{R}^{N+1}$ defined as
\begin{equation}\label{F}
F(y',y_N,t)=(\psi_1(y',y_N),...,\psi_{N-1}(y',y_N),y_N+g(y'),t).
\end{equation}
We observe that $F$ is of class $C^{1,1}$ and $F(y',0,0)=(y',
g(y'),0)$, namely $F^{-1}$ is straightening the boundary of the set
$\{(x',x_N,0):x_N<g(x')\}$. 

Direct computations and \eqref{eq:3} yield that
\begin{align}\label{eq:JacF}
  J_{F}(y',y_N,t)&=J(y',y_N)\\
  \notag&=	
          \begin{pmatrix}
    1-y_N\frac{\partial \widetilde G_1}{\partial y_1} &
    -y_N\frac{\partial \widetilde G_1}{\partial y_2} & 
    \cdots & -y_N\frac{\partial \widetilde G_1}{\partial y_{N-1}} &
    -\widetilde G_1
    -y_N \frac{\partial
      \widetilde G_1}{\partial y_N}&0
    \\[5pt]
    -y_N\frac{\partial \widetilde G_2}{\partial y_1} & 1-y_N\frac{\partial \widetilde G_2}{\partial y_2} & \cdots & -y_N\frac{\partial \widetilde G_2}{\partial y_{N-1}} & - \widetilde G_2 -y_N \frac{\partial
      \widetilde G_2}{\partial y_N}&0 \\
    \vdots & \vdots & \ddots & \vdots & \vdots &\vdots\\[5pt]
    -y_N\frac{\partial \widetilde G_{N-1}}{\partial y_1} & -y_N\frac{\partial \widetilde G_{N-1}}{\partial y_2} & \cdots & 1-y_N\frac{\partial \widetilde G_{N-1}}{\partial y_{N-1}} & - \widetilde G_{N-1} -y_N \frac{\partial
      \widetilde G_{N-1}}{\partial y_N}&0 \\[5pt]
    \frac{\partial g}{\partial y_1}(y') & \frac{\partial
      g}{\partial y_2}(y') & \cdots & \frac{\partial
      g}{\partial y_{N-1}}(y') & 1&0 \\[5pt]
    0&0&\dots&0&0&1
  \end{pmatrix}\\[5pt]
                 &\notag= \left( \renewcommand{\arraystretch}{1.5}
\begin{array}{c|c|c}
\mathrm{Id}_{N-1}-y_NJ_{\widetilde G}&-\nabla g(y')+O(y_N)&{\mathbf 0}\\\hline
(\nabla g(y'))^T&1&0\\\hline
{\mathbf 0}^T&0&1
\end{array}\right),
\end{align}
where $\nabla g(y')$ is meant as a column vector in $\R^{N-1}$,
${\mathbf 0}$ is the null column vector in $\R^{N-1}$ and
$(\nabla g(y'))^T,{\mathbf 0}^T$ are their transpose; from now on, the
notation $O(y_N)$ will be used to denote blocks of matrices with all
entries being $O(y_N)$ as $y_N\to 0$ uniformly with respect to $y'$
and $t$.

From \eqref{hpdellag} and the fact that $g\in C^{1,1}(\R^{N-1})$ it
follows that $\nabla g(y')=O(|y'|)$ as $|y'|\to0$, then 
\begin{equation}\label{eq:det-jac}
\det J(y',y_N)=1+|\nabla g(y')|^2+O(y_N)=
1+O(|y'|^2)+O(y_N)
\end{equation}
as $y_N\to0$ and $|y'|\to0$. 

In particular we have that $\det J_{F}(0)=1\ne0$; therefore, by the Inverse Function Theorem, $F$ is invertible in a neighbourhood of the origin, i.e.
there exists $R_1>0$ such that
\begin{equation}\label{eq:alpha}
\alpha(y',y_N):=\det J(y',y_N)>0\quad\text{in }B_{R_1}'
\end{equation}
and $F$ is a diffeomorphism of
class $C^{1,1}$ from $B_{R_1}$ to $\mathcal U=F(B_{R_1})$ for some
$\mathcal U$ open neighbourhood of $0$ such that $\mathcal U\subset B_R$. Furthermore
\begin{equation*}
  F^{-1}(\mathcal U\cap\Gamma_{
    g,R}^-)=\Gamma_{R_1}^-\quad\text{and}\quad F^{-1}(\mathcal U\cap\Gamma_{
    g,R}^+)=\Gamma_{R_1}^+,
\end{equation*}
where, for all $r>0$, we denote 
\begin{equation*}
  \Gamma_{r}^-:=\{(y',y_N,0)\in
  B'_{r}:y_N< 0\},\quad \Gamma_{r}^+:=\{(y',y_N,0)\in B'_{r}:y_N\geq
  0\}.
\end{equation*}
Since
\begin{equation*}
  F^{-1}\in C^{1,1}(\mathcal U,B_{R_1}), \quad
  F\in C^{1,1}(B_{R_1},\mathcal U), \quad
  F(0)=F^{-1}(0)=0, \quad J_F(0)=J_{F^{-1}}(0)=\mathrm{Id}_{N+1},
\end{equation*}
we have that 
\begin{align}
  \label{eq:11}& J_{F^{-1}}(x)=\mathrm{Id}_{N+1}+O(|x|)\quad\text{and}\quad 
                 F^{-1} (x)=x+O(|x|^2)\quad\text{as }|x|\to0,\\
  \label{eq:10}
             & J_{F}(y)=\mathrm{Id}_{N+1}+O(|y|)\quad\text{and}\quad 
               F(y)=y+O(|y|^2)\quad\text{as }|y|\to0.
\end{align}
If $U$ is a solution to \eqref{prob3}, then $W=U\circ F$ is solution to
\begin{equation}\label{prob4}
  \begin{cases}
  -\mathrm{div}\left(t^{1-2s}A\nabla W\right)=0 &\mathrm{in \ } B_{R_1}^+,\\
  \lim_{t\to0^+}\left(
    t^{1-2s}A\nabla W\cdot\nu\right)=\kappa_s\tilde h \, \mathop{\rm Tr}W &\mathrm{in \ } \Gamma_{R_1}^-,\\
  W=0 & \mathrm{in \ } \Gamma_{R_1}^+,
\end{cases}
\end{equation}
where $\nu=(0,0,\dots,0,-1)$ is the vertical downward
unit vector,
$A$ is the $(N+1)\times(N+1)$ variable coefficient matrix (not depending
on $t$) given by 
\begin{equation}\label{A}
  A(y)=(J(y))^{-1}((J(y))^{-1})^T|\det J(y)|,
\end{equation}
and 
\begin{equation*}
  \tilde h(y)=\alpha(y)h(F(y,0)),\quad y\in \Gamma_{R_1}^-.
\end{equation*}
Equation \eqref{prob4} is meant 
in
a weak sense, i.e. 
$W$ belongs to $H^1_{\Gamma_{R_1}^+}(B_{R_1}^+,t^{1-2s}\,dz)$ (defined
as the closure of $C^\infty_c(\overline{B_{R_1}^+}\setminus \Gamma_{R_1}^+)$ in $H^1(B_{R_1}^+,t^{1-2s}\,dz)$) 
and satisfies
\begin{equation}\label{eq:9}
  \int_{B_{R_1}^+}t^{1-2s}A(y)\nabla
  W(y,t)\cdot\nabla\phi(y,t)\,dz-\kappa_s\int_{\Gamma_{R_1}^-}\tilde
  h\mathop{\rm Tr}W\mathop{\rm Tr}\phi\,dy=0
\end{equation}
for all $\phi\in C^\infty_c(B_{R_1}^+\cup \Gamma_{R_1}^-)$.

We observe that $A$ is symmetric and, in view of
\eqref{eq:11}--\eqref{eq:10}, uniformly elliptic if
$R_1$ is chosen sufficiently small; furthermore $A$ has $C^{0,1}$
coefficients. 
We also remark that, under assumption~\eqref{eq:ass-h}, 
\begin{equation}\label{eq:tildehw1p}
  \tilde h\in W^{1,p}(\Gamma_{R_1}^-).
\end{equation}
From \eqref{eq:JacF} it follows that 
\begin{equation*}
  J^{-1}=
  \left(
    \begin{array}{c|c}
      M^{-1} & {\mathbf 0} \\[2pt]
      \hline \\[-10pt]
      {\mathbf 0}^T & 1
\end{array}
\right),
\end{equation*}
where ${\mathbf 0}$ is the null column vector in $\R^{N}$ and 
\begin{equation}\label{eq:M}
  M=M(y',y_N)=\left( \renewcommand{\arraystretch}{1.5}
    \begin{array}{c|c}
      \mathrm{Id}_{N-1}-y_NJ_{\widetilde G}&-\nabla g(y')+O(y_N)\\\hline
      (\nabla g(y'))^T&1
    \end{array}\right).
\end{equation}
From \eqref{eq:JacF}  and \eqref{eq:alpha} one can deduce that 
\begin{equation}\label{detJac}
  \mathop{\rm det}M(y',y_N)=\alpha(y',y_N) >0\quad\text{in }B_{R_1}'.
\end{equation}
Let us define 
\begin{equation*}
  B(y',y_N):=\mathop{\rm det}M(y',y_N)(M(y',y_N))^{-1}.
\end{equation*}
By \eqref{eq:M} and a direct calculation we have that 
\begin{equation}\label{B}
  B=
  \left( \renewcommand{\arraystretch}{1.8}
    \begin{array}{c|c}
      {\scriptsize\begin{matrix}
          1+\mathop{\sum}\limits_{j\neq1}\big|\frac{\partial g}{\partial
            y_j}\big|^2\!+\!O(y_N) &
          -\frac{\partial g}{\partial y_1}\frac{\partial g}{\partial y_2}\!+\!O(y_N)& 
          \cdots & -\frac{\partial g}{\partial y_1}\frac{\partial g}{\partial y_{N-1}}\!+\!O(y_N)
          \\[-3pt]
          -\frac{\partial g}{\partial y_2}\frac{\partial g}{\partial
            y_1}\!+\!O(y_N)&	
          1+\mathop{\sum}\limits_{j\neq2}\big|\frac{\partial g}{\partial
            y_j}\big|^2\!+\!O(y_N)& 
          \cdots & -\frac{\partial g}{\partial y_2}\frac{\partial g}{\partial y_{N-1}}\!+\!O(y_N)\\[-3pt]
          \vdots & \vdots & \ddots & \vdots \\[-3pt]
          -\frac{\partial g}{\partial y_{N-1}}\frac{\partial g}{\partial
            y_1}\!+\!O(y_N)	
          & -\frac{\partial g}{\partial y_{N-1}}\frac{\partial g}{\partial y_2}\!+\!O(y_N)&
          \cdots & 1+\mathop{\sum}\limits_{j\neq{N-1}}\big|\frac{\partial g}{\partial
            y_j}\big|^2\!+\!O(y_N)
        \end{matrix} } & \nabla g+O(y_N)\!\\\hline -(\nabla
      g)^T+O(y_N)& 1+O(y_N)
    \end{array}\right).
\end{equation}
Then $(J(y))^{-1}$ can be rewritten as follows
\begin{equation*}
  (J(y))^{-1} =
  \left(
    \renewcommand{\arraystretch}{1.1}
    \begin{array}{c|c}
      \frac1{\alpha(y)}\,B(y) & {\mathbf 0} \\
      \hline \\[-10pt]
      {\mathbf 0}^T & 1
    \end{array}
  \right),
\end{equation*}
thus from \eqref{A} it turns out that
\begin{equation}\label{eq:1}
  A(y)=
  \left(
    \renewcommand{\arraystretch}{1.1}
    \begin{array}{c|c}
  D(y) & 0 \\
  \hline
  0 & \alpha(y)
    \end{array}
  \right),
\end{equation}
where $D=\frac1{\alpha}BB^T$. From \eqref{B},
\eqref{eq:det-jac}, and \eqref{eq:alpha} it follows that
\begin{equation}\label{eq:6}
  D(y',y_N)=
  \left(
    \renewcommand{\arraystretch}{1.5}
    \begin{array}{c|c}
      \mathrm{Id}_{N-1}+O(|y'|^2)+O(y_N) & O(y_N) \\
      \hline
      O(y_N) & 1+O(|y'|^2)+O(y_N)
    \end{array}
  \right),
\end{equation}
where here $O(y_N)$, respectively $O(|y'|^2)$, denotes
blocks of matrices with all entries being $O(y_N)$ as $y_N\to 0$, 
respectively $O(|y'|^2)$ as $|y'|\to 0$. In particular we have that
\begin{equation}\label{sviluppodellaA}
  A(y)=\mathrm{Id}_{N+1}+O(|y|)\quad \text{as $|y|\rightarrow 0$}.
\end{equation}
Let us define, for every $z=(y,t)\in B_{R_1}$,
\begin{equation}\label{beta}
  \beta(z)=\frac{A(y)z}{\mu(z)}=(\beta'(z),\beta_{N+1}(z))=
\left(\frac{D(y)y}{\mu(z)},\frac{\alpha(y)t}{\mu(z)}\right),
\end{equation}
where
\begin{equation}\label{mu}
  \mu(z)=\frac{A(y)z\cdot z}{|z|^2}, \quad z\neq0.
\end{equation}
We observe that, possibly choosing
$R_1$ smaller, $\beta$ is
well-defined, since $\mu(z)>0$ in $B_{R_1}$, and 
\begin{equation}\label{eq:8}
\|A(y)\|_{\mathcal L(\R^{N+1},\R^{N+1})}\leq 2\quad\text{for all
}y\in B_{R_1}'.
\end{equation}
Moreover, 
for every $\xi=(\xi_1,\dots,\xi_N,\xi_{N+1} ) \in \R^{N+1}$ and $y\in
B_{R_1}'$, we define $dA(y)\,\xi\,\xi\in \R^{N+1}$ as the
vector in $\R^{N+1}$ with $i$-th component, for $i=1,...,N+1$, given by
\begin{equation}\label{dA}
(dA(y)\,\xi\,\xi)_i=\sum_{j,k=1}^{N+1}\partial_{z_i} a_{jk}(y)\,\xi_{j}\xi_{k}.
\end{equation}
\begin{Lemma}\label{lemmastimamu}
Let $\mu$ be as in \eqref{mu} and $A$ as in \eqref{A}. Then
\begin{equation}\label{sviluppomu}
\mu(z)=1+O(|z|)\quad \text{as $|z|\rightarrow 0^+$}
\end{equation}
and
\begin{equation}\label{sviluppogradmu}
\nabla\mu(z)=O(1)\quad \text{as $|z|\rightarrow 0^+$}.
\end{equation}
\end{Lemma}
\begin{proof}
Estimate \eqref{sviluppomu} follows directly from
\eqref{mu} and
\eqref{sviluppodellaA}. In order to prove
\eqref{sviluppogradmu}, we differentiate \eqref{mu}, obtaining that,
for all $z=(y,t)\in B_{R_1}$,
\begin{equation*}
  \nabla\mu(z)=-2\frac{(A(y)z\cdot z)z}{|z|^4}+
  \frac{dA(y)zz}{|z|^2}+2\frac{A(y)z}{|z|^2}=-2\frac{\mu(z)z}{|z|^2}+\frac{dA(y)zz}{|z|^2}+2\frac{A(y)z}{|z|^2}.
\end{equation*}
Noting that $dA(y)zz=O(|z|^2)$ as $|z|\to0$, since the matrix $A$ has
Lipschitz coefficients, and using \eqref{sviluppomu} and \eqref{sviluppodellaA}, we deduce that 
\begin{equation*}
\begin{split}
\nabla\mu(z)&=-\frac{2z}{|z|^2}[1+O(|z|)]+O(1)+\frac{2}{|z|^2}[z+O(|z|^2)]=O(1)
\end{split}
\end{equation*}
as $|z|\rightarrow 0^+$, thus proving \eqref{sviluppogradmu}.
\end{proof}

\begin{Lemma}\label{lemmabeta}
Let $\beta$ be as in \eqref{beta} and $A$ as in \eqref{A}. Then we
have that, as $|z|\rightarrow 0^+$,
\begin{align}\label{sviluppobeta}
&\beta(z)=z+O(|z|^2)=O(|z|),\\
\label{sviluppojacbeta}&J_\beta(z)=A(y)+O(|z|)= \mathrm{Id}_{N+1}+O(|z|),\\
\label{sviluppodivbeta}&\mathrm{div}\beta(z)=N+1+O(|z|).
\end{align}
\end{Lemma}
\begin{proof}
The result follows by combining \eqref{sviluppomu},
\eqref{sviluppogradmu} and \eqref{sviluppodellaA}. 
\end{proof}

\subsection{Some inequalities}\label{sec:some-inequalities}

We recall from \cite[Lemma 2.4]{FalFel} the following Hardy type
inequality with boundary terms, which will be used throughout the paper.

\begin{Lemma}\label{l:hardy_boundary}
  For all $r>0$ and $w\in H^1(B_r^+,t^{1-2s}\,dz)$
  \begin{align*}
    \bigg(\frac{N-2s}{2}\bigg)^{\!\!2}\int_{B_r^+}t^{1-2s}\frac{w^2(z)}{|z|^2}\,dz
\leq \int_{B_r^+}t^{1-2s}\bigg(\nabla w(z)\cdot\frac{z}{|z|}\bigg)^{\!\!2}dz
+    \bigg(\frac{N-2s}{2r}\bigg)\int_{\partial^+B_r^+}t^{1-2s}w^2dS.
  \end{align*}
\end{Lemma}

In order to prove a coercivity-type inequality, we provide the
following Sobolev type inequality with boundary terms (see Lemma 2.6
in \cite{FalFel}).
\begin{Lemma}\label{lemma2.6FF} There exists $\tilde{S}_{N,s}>0$ such
that, for all $r>0$ and $V\in H^1(B^+_r, t^{1-2s} \,dz)$,
\begin{equation}\label{2.6FF}
\left(\int_{B'_r}|\mathop{\rm Tr} V|^{2^*(s)}dy\right)^{\frac{2}{2^*(s)}}\leq \tilde{S}_{N,s}\biggl[\frac{N-2s}{2r}\int_{\partial^+B_r^+}t^{1-2s}V^2dS+\int_{B_r^+}t^{1-2s}|\nabla V|^2dz\biggr].
\end{equation}
\end{Lemma}
Exploiting \eqref{2.6FF} we now prove the following inequality which
will be useful in the sequel.
\begin{Lemma}\label{lemma3.12FF}
For every $\bar{\alpha}>0$, there exists $r(\bar{\alpha})\in
(0,R_1)$ such that, for any $0<r\leq r(\bar{\alpha})$,
$\zeta\in~\!\!L^p(B'_{R_1})$ such that
$\|\zeta\|_{L^p(B'_{R_1})}\leq\bar{\alpha}$ and $V\in H^1(B_r^+,t^{1-2s}\,dz)$,
\begin{align}\label{coercvity}
\int_{B_r^+}t^{1-2s}A\nabla V\cdot\nabla V\, dz-\kappa_s&\int_{B'_r}\zeta |\mathop{\rm Tr} V|^2dy+\frac{N-2s}{2r}\int_{\partial^+B_r^+}t^{1-2s}\mu V^2\, dS\\
\notag&\geq \tilde C_{N,s}\left(\int_{B_r^+}t^{1-2s}|\nabla V|^2\, dz+\left(\int_{B'_r}|\mathop{\rm Tr} V|^{2^*(s)}dy\right)^{\frac{2}{2^*(s)}}\right),
\end{align}
for some positive constant $\tilde C_{N,s}>0$
depending only on $N$ and $s$.
\end{Lemma}
\begin{proof}
Let us estimate from below each term in the left hand side of
\eqref{coercvity}. To this aim, exploiting \eqref{sviluppodellaA}, we
can choose $r_1\in (0,R_1)$ such that, for all $0<r\leq r_1$
and $V\in H^1(B_r^+,t^{1-2s}\,dz)$,
\begin{equation}\label{eqcoerc1}
\int_{B_r^+}t^{1-2s}A\nabla V\cdot\nabla V\ dz\geq \frac{1}{2}\int_{B_r^+}t^{1-2s}|\nabla V|^2\ dz. 
\end{equation}
Furthermore, thanks to \eqref{sviluppomu}, we can assert that $\mu
\geq 1/4$ in $B_r$ if $0<r\leq r_2$, for some $r_2\in
(0,R_1)$.
Hence, using \eqref{2.6FF}, we deduce that, for all
$0<r\leq r_2$
and $V\in H^1(B_r^+,t^{1-2s}\,dz)$,
\begin{equation}\label{eqcoerc2}
\frac{N-2s}{2r}\int_{\partial^+B_r^+}t^{1-2s}\mu V^2\, dS\geq
\frac{1}{4\tilde{S}_{N,s}} \left(\int_{B'_r}|\mathop{\rm Tr}
V|^{2^*(s)}dy\right)^{\frac{2}{2^*(s)}}-\frac{1}{4}\int_{B^+_r}t^{1-2s}|\nabla
V|^2\, dz.
\end{equation}
By H\"{o}lder's inequality, we infer that,
for all $r\in (0,R_1)$, $V\in
H^1(B_r^+,t^{1-2s}\,dz)$, and $\zeta\in L^p(B'_{R_1})$ such that
$\|\zeta\|_{L^p(B'_{R_1})}\leq\bar{\alpha}$,
\begin{align}\label{useful}
\int_{B'_r}\zeta\,|\mathop{\rm Tr} V|^2\, dy&\leq \tilde
c_{N,s,p}\,r^{\overline\varepsilon}\Vert
\zeta\,\Vert_{L^p(B'_{R_1})}\left(\int_{B'_r}|\mathop{\rm Tr}
V|^{2^*(s)} dy\right)^{\!\!\frac{2}{2^*(s)}}\\
&\notag\leq \tilde c_{N,s,p}\bar{\alpha}\,r^{\overline\varepsilon}\left(\int_{B'_r}|\mathop{\rm Tr} V|^{2^*(s)} dy\right)^{\!\!\frac{2}{2^*(s)}}
\end{align}
for some $\tilde c_{N,s,p}>0$ (depending only on $p,N,s$), where 
\begin{equation}\label{overepsilon}
\overline\varepsilon=\frac{2sp-N}{p}>0.
\end{equation}
Selecting $r_3=r_3(\bar{\alpha})\in (0,R_1)$ such that 
\begin{equation}\label{<1/8}
\kappa_s \tilde c_{N,s,p}\bar{\alpha} r^{\overline\varepsilon}\leq\frac{1}{8\tilde{S}_{N,s}}
\quad\text{for all $0<r\leq r_3$}
\end{equation}
and combining \eqref{eqcoerc1}, \eqref{eqcoerc2},
and \eqref{useful}, we obtain that, for all $0<r\leq r(\bar{\alpha}):=\min\{r_1,r_2,r_3\}$ and 
$V\in H^1(B_r^+,t^{1-2s}\,dz)$,
\begin{equation*}
\begin{split}
\int_{B_r^+}t^{1-2s}A\nabla V\cdot\nabla V\, dz&-\kappa_s\int_{B'_r}\zeta |\mathop{\rm Tr} V|^2dy+\frac{N-2s}{2r}\int_{\partial^+B_r^+}t^{1-2s}\mu V^2\, dS\\
&\geq \frac{1}{4}\int_{B^+_r}t^{1-2s}|\nabla V|^2dz
+\biggl(\frac{1}{4\tilde{S} _{N,s}}-\kappa_s
\tilde c_{N,s,p}\bar{\alpha}\, r^{\overline\varepsilon}\biggr)\left(\int_{B'_r}|\mathop{\rm Tr} V|^{2^*(s)}dy\right)^{\frac{2}{2^*(s)}}\\
&\geq \tilde C_{N,s}\left(\int_{B_r^+}t^{1-2s}|\nabla
V|^2dz+\left(\int_{B'_r}|\mathop{\rm Tr}V|^{2^*(s)}dy\right)^{\frac{2}{2^*(s)}}\right),
\end{split}
\end{equation*}
where $\tilde C_{N,s}:=\min\big\{\frac14,\frac1{8\tilde{S}_{N,s}}\big\}$, thus proving \eqref{coercvity}.
\end{proof}
\begin{remark} 
For $\bar{\alpha}>0$, let $r(\bar{\alpha})$ and $\tilde c_{N,s,p}$ be as 
in Lemma
\ref{lemma3.12FF} and let 
$\zeta\in L^p(B'_{R_1})$ be such that
$\|\zeta\|_{L^p(B'_{R_1})}\leq\bar{\alpha}$. Then, for every $r\in(0,r(\bar{\alpha})]$ and 
$V\in H^1(B^+_r, t^{1-2s} \,dz)$, we have that
\begin{equation}\label{remark}
  \int_{B'_r}\zeta|\mathop{\rm Tr} V|^2dy\leq\tilde{S}_{N,s}\tilde
  c_{N,s,p}
  r^{\overline{\varepsilon}}\bar{\alpha}\frac{2(N-2s)}{r}
  \int_{\partial^+B_r^+}t^{1-2s}\mu V^2 dS+\frac{1}{8\kappa_s}\int_{B^+_r}t^{1-2s}|\nabla V|^2\ dz.
\end{equation}
\end{remark}
\begin{proof}
Applying \eqref{useful} and \eqref{2.6FF}, we obtain \eqref{remark}, taking into account that, for all $0<r\leq r(\bar{\alpha})$, \eqref{<1/8} holds and $\mu \geq 1/4$.
\end{proof}

\subsection{Construction of the approximating sequence and convergence}\label{sec:constr-appr-sequ}

The main  difficulty in the proof of a Pohozaev type
identity for problem \eqref{prob4}, which is needed to
differentiate the Almgren quotient, relies in a substancial  lack of regularity at Dirichlet-Neumann
junctions. Here we face this
difficulty by a double approximation procedure, involving both the
potential $h$ and the $N$-dimensional region $\Gamma_{R_1}^+$ where the  solution to \eqref{prob4} is
forced to vanish.

Let $\eta\in C^{\infty}([0,+\infty))$ be such that
\begin{equation}\label{eq:13}
  \eta\equiv1 \text{ in
    $[0,1/2]$},\quad \eta\equiv0 \text{ in $[1,+\infty)$},\quad 0\leq\eta\leq1 \quad\text{and}\quad
  \eta'\leq0.
\end{equation}
Let 
\begin{equation*}
  f:[0,+\infty)\to\R,\quad f(t)=\eta(t)+(1-\eta(t))t^{1/4}.
\end{equation*}
We observe that 
\begin{equation}\label{eq:pf}
f\in C^{\infty}([0,+\infty)),\quad f(t)=1 \text{ for all
$t\in[0,1/2]$},\quad\text{and}\quad
f(t)-4t\,f'(t) \geq0\text{ for all }t\geq0.
\end{equation}
Furthermore 
\begin{equation}\label{eq:pf2}
f(t)\geq\frac12 \quad\text{and}\quad 
f(t)\geq t^{1/4}\quad \text{for all }t\geq0.
\end{equation}
For every $n\in\mathbb{N}\setminus\{0\}$, we consider the function
\begin{equation*}
f_n(t)=\frac{f(nt)}{n^{1/8}}.
\end{equation*}
Then, \eqref{eq:pf} implies that 
\begin{equation}\label{eq:7}
f_n\in C^{\infty}([0,+\infty)),\quad f_n(t)=n^{-1/8} \text{ for all
$t\in [0,1/2n]$},\quad
f_n(t)-4t\,f_n'(t) \geq0\text{ for all }t\geq0,
\end{equation}
whereas \eqref{eq:pf2} yields
\begin{equation}\label{eq:2}
f_n(t)\geq\frac12 n^{-1/8} \quad\text{and}\quad 
f_n(t)\geq n^{1/8} t^{1/4}\quad \text{for all }t\geq0.
\end{equation}
By \eqref{eq:tildehw1p} and density of smooth
functions in Sobolev spaces,
there exists a sequence of potential terms
$h_n\in C^{\infty}\big(\overline{\Gamma_{R_1}^-}\big)$ such that
\begin{equation}\label{eq:20}
h_n\to \tilde h \quad\text{in }W^{1,p}(\Gamma_{R_1}^-).
\end{equation}
Let 
\begin{equation}\label{eq:alpha0}
\bar{\alpha}_0=\sup_n\|h_n\|_{L^{p}(\Gamma_{R_1}^-)}
\end{equation}
and set 
\begin{equation}\label{eq:R_0}
R_0=r(\bar{\alpha}_0)
\end{equation}
according to the notation introduced in Lemma \ref{lemma3.12FF}.
\begin{remark}\label{r:r0}
Because of the above choice of $R_0$, we have that 
\eqref{coercvity} holds with any $\zeta\in L^p(B'_{R_1})$ such that
$|\zeta|\leq |h_n|$ a.e. (being $h_n$ trivially
extended in $B_{R_1}'\setminus \Gamma_{R_1}^-$), for any $n\in{\mathbb N}\setminus\{0\}$, $r\leq
R_0$, and for
all $V\in H^1(B_r^+,t^{1-2s}\,dz)$.
\end{remark}

Let us define, for all $n\in{\mathbb N}\setminus\{0\}$, 
\begin{equation}\label{gamman}
\gamma_n=\{(y',y_N,t)\in \overline{B_{R_0}^+}: \ y_N=f_n(t)\},
\end{equation}
with $R_0$ as in \eqref{eq:R_0}.
If $(y',y_N,t)\in \gamma_n$, then from \eqref{eq:2}
it follows that 
\begin{equation*}
n^{1/8} t^{1/4}\leq
f_n(t)=y_N\leq R_0,
\end{equation*}
thus proving that
\begin{equation}\label{eq:5}
t\leq \frac{R_0^4}{\sqrt{n}}\quad \text{if}\quad (y',y_N,t)\in \gamma_n.
\end{equation}
The approximating domains are defined as 
\begin{equation}\label{An}
\mathcal U_n=\{(y',y_N,t)\in B ^+_{R_0}:y_N<f_n(t)\}
\end{equation}
with topological boundary
$$\partial \mathcal U_n=\sigma_n\cup\gamma_n\cup\tau_n,$$
where $\gamma_n$ has been defined in \eqref{gamman} and 
\begin{equation*}
\sigma_n=\left\{(y',y_N)\in B'_{R_0} : y_N<\frac{1}{n^{1/8}}\right\},\quad
\tau_n=\left\{(y',y_N,t)\in \partial B_{R_0} : t\geq 0,\ y_N<f_n(t)\right\},
\end{equation*}
see Figure \ref{fig:approx-domains}.

\begin{figure}[ht]
\begin{tikzpicture}[line cap=round,line join=round,>=triangle 45, scale=3]
\draw [->] (-1.2,0)-- (1.2,0) node [right] {$y_N$};
\draw [->] (0,0)-- (0,1.2) node [left] {$t$};
\draw [black, fill=black, opacity=0.3] (-1,0) to [out=90, in=180] (0,1) to [out=0, in=120] (0.866,0.5) to [out=238, in=20] (0.43,0.15) to [out=200, in=90]  (0.33,0.076) to [out=270, in=90] (0.35,0.040) to [out=270, in=90] (0.35,0) -- (-1,0);
\draw  (-1,0) to [out=90, in=180] (-1,0) to [out=90, in=180] (0,1) to [out=0, in=120] (0.866,0.5) to [out=238, in=20] (0.43,0.15) to [out=200, in=90]  (0.33,0.076) to [out=270, in=90] (0.35,0.040) to [out=270, in=90] (0.35,0) -- (-1,0);
\draw[fill] (0.35,0) circle [radius=0.010];
			\node [below] at (0.33,0) {$n^{-1/8}$};
			\draw[color=black] (-0.5,-0.12) node {$\sigma_n$};
\draw[color=black] (-0.6,0.9) node {$\tau_n$};
\draw[color=black] (0.7,0.2) node {$\gamma_n$};			
\end{tikzpicture}
   \caption{Vertical section of the approximating domain $\mathcal U_n$.}
\label{fig:approx-domains}
\end{figure}
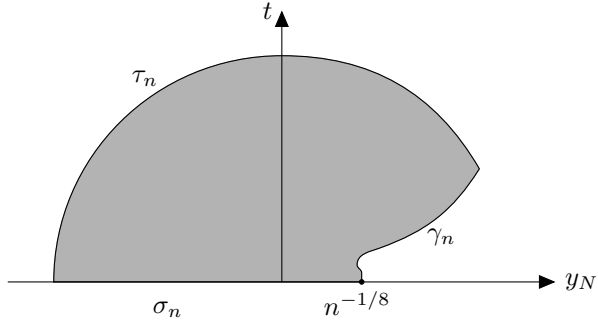

The functions $f_n$ above have been constructed with the aim of
making $\mathcal U_n$ satisfy the following geometric property, which will be needed in the proof of
a
monotonicity formula.
\begin{Lemma}\label{segnogamma}
There exists $\bar n\in{\mathbb N}\setminus\{0\}$ such
that, for all $n\geq\bar n$ and $z=(y,t)\in\gamma_n\cap B_{R_0}^+$,
\begin{equation}\label{eq:gc}
A(y)z\cdot\nu\geq0\quad\text{on }\gamma_n,
\end{equation}
where 
$\gamma_n$ has been defined in \eqref{gamman}
and $\nu=\nu(z)$ is the exterior unit normal vector at $z\in\partial \mathcal
U_n$. 
\end{Lemma}
\begin{proof}
For all $z=(y,t)\in \gamma_n\cap B_{R_0}^+$ we have that 
$\nu=\nu(z)=\frac{\mathfrak n}{|\mathfrak n|}$ where $\mathfrak n=({\mathbf 0},1,-f'_n(t))$.
Hence, from \eqref{eq:1} and \eqref{eq:6} it follows that
\begin{equation*}
A(y)(y,t)\cdot \mathfrak n=(D(y)y,\alpha(y)t)\cdot(({\mathbf 0},1),-f'_n(t))
=y_N(1+O(|y'|)+O(y_N))-\alpha(y)tf'_n(t).
\end{equation*}
Therefore, possibly choosing 
$R_1$ (and consequently $R_0$) smaller from the beginning and recalling
\eqref{eq:det-jac}--\eqref{eq:alpha},
we obtain that 
\begin{align*}
A(y)(y,t)\cdot \mathfrak n\geq \begin{cases}
\frac{y_N}{2}-2tf'_n(t)=\frac{1}{2}(f_n(t)-4tf'_n(t)) &\mathrm{if \ }f'_n(t)\geq0\\
\frac{y_N}{2} &\mathrm{if \ }f'_n(t)\leq0
\end{cases}
\end{align*}
thus concluding that $A(y)(y,t)\cdot \mathfrak n\geq0$ in view of \eqref{eq:7}.
\end{proof}

Now we construct a sequence $U_n$ of solutions to some suitable
approximating problems in the domains $\mathcal U_n$ defined in
\eqref{An}, which converges strongly to $W$ in the weighted Sobolev
space $H^1(B_{R_0}^+,t^{1-2s}\,dz)$. 
The functions $U_n$ will be sufficiently regular to
satisfy a 
Rellich-Ne\u{c}as identity and make it integrable on $\mathcal U_n$,
thus allowing us to obtain a 
Pohozaev type identity for $U_n$ with some remainder terms produced by
the transition of the boundary conditions, whose sign can
anyway be understood thanks to
the geometric property \eqref{eq:gc}; therefore, passing to the limit
in 
the Pohozaev identity satisfied by $U_n$, we end up with inequality
\eqref{pohozaev} for $W$, which will be used to estimate from below the
derivative of the Almgren frequency function \eqref{N} and then to
prove that such frequency has a finite limit at 0 (Proposition \ref{Lemmalimitexists}).

Let $W\in H^1(B_{R_1}^+,t^{1-2s}\,dz)$ be a non-trivial energy
solution to \eqref{prob4}, in the sense clarified in
\eqref{eq:9}. By density, there exists a sequence of functions
$G_n\in C^{\infty}_c(\overline{B_{R_1}^+}\setminus\Gamma_{R_1}^+)$
such that $G_n\to W$ strongly in
$H^1(B_{R_1}^+,t^{1-2s}\,dz)$. Thanks to \eqref{eq:5}, without loss of generality
we
can assume that
$G_n=0$ on $\gamma_n$. 

We construct a sequence of cut-off functions in the following way:
letting $\eta\in C^\infty([0,+\infty))$ be as in
\eqref{eq:13}, we define
\begin{equation}\label{eq:22}
\eta_n:\R^{N}\to\R,\quad \eta_n(y',y_N)=\begin{cases}
1-\eta\left(-\tfrac{n y_N}{2}\right)&\text{if }y_N\leq 0,\\
0&\text{if }y_N> 0.
\end{cases}
\end{equation}
For any fixed $n\in\mathbb{N}$, we consider the following boundary value problem 
\begin{equation}\label{eq:26}
\begin{cases}
-\mathrm{div}\left(t^{1-2s}A\nabla U_n\right)=0 &\mathrm{in \ } \mathcal U_n,\\
\lim_{t\to0^+}\left(
t^{1-2s}A\nabla U_n\cdot\nu\right)=\kappa_s\eta_nh_n \mathop{\rm Tr}U_n &\mathrm{in \ } \sigma_n,\\
U_n=G_n & \mathrm{in \ } \tau_n\cup\gamma_n,
\end{cases}
\end{equation}
in a weak sense, i.e.
\begin{equation}\label{prob5}
\begin{cases}U_n-G_n\in \mathcal H_n,\\[5pt]
{\displaystyle{\int_{\mathcal U_n} t^{1-2s}A\nabla U_n\cdot\nabla\Phi\ dz
-\kappa_s\int_{\sigma_n}\eta_n h_n \mathop{\rm Tr}U_n\,\mathop{\rm Tr}\Phi\ dy}}=0
\quad\text{for all }\Phi\in\mathcal H_n,
\end{cases}
\end{equation}
where $\mathcal{H}_n$ is defined as the closure of
$C^\infty_c({\mathcal U_n}\cup \sigma_n)$ in $H^1(\mathcal
U_n,t^{1-2s}\,dz)$. 
Existence of solutions to \eqref{prob5} and their convergence to $W$
are established in the following proposition.
\begin{Proposition}\label{existenceuniqueness}
For any fixed $n\in\mathbb{N}$, there exists a unique solution $U_n$
to \eqref{prob5}. Moreover $U_n\to W$ strongly in
$H^1(B_{R_0}^+,t^{1-2s}dz)$ 
(where $U_n$ is extended trivially to zero in $B_{R_0}^+\setminus
\mathcal U_n$).
\end{Proposition}
\begin{proof}
$U_n$ solves \eqref{prob5} if and only if
$V_n=U_n-G_n$ satisfies 
\begin{equation}\label{wf}
V_n\in \mathcal H_n\quad\text{and}\quad b_n(V_n,\Phi)=\langle
F_n,\Phi\rangle\quad\text{for all }\Phi\in\mathcal H_n,
\end{equation}
where
\begin{equation}\label{b}
b_n:\mathcal H_n\times\mathcal H_n\to\R,\quad
b_n(V,\Phi)=\int_{\mathcal U_n} t^{1-2s}A\nabla V\cdot\nabla\Phi\, dz
-\kappa_s\int_{\sigma_n}\eta_n h_n \mathop{\rm Tr}V\mathop{\rm Tr}\Phi\, dy
\end{equation}
and
\begin{equation}\label{Fn}
F_n: \mathcal H_n\to\R,\quad\langle F_n,\Phi\rangle= -\int_{\mathcal U_n} t^{1-2s}A\nabla G_n\cdot \nabla\Phi\ dz+\kappa_s\int_{\sigma_n}\eta_nh_n \mathop{\rm Tr}G_n \mathop{\rm Tr}\Phi\ dy.
\end{equation}
From H\"{o}lder's inequality, \eqref{2.6FF}, and the boundedness of
$\{h_n\}$ and $\{G_n\}$ respectively in $L^p(\Gamma_{R_1}^-)$ and in $H^1(B_{R_1}^+,t^{1-2s}\,dz)$, it follows that 
\begin{equation}\label{eq:12}
|\langle F_n,\Phi\rangle|\leq c\,\|\Phi\|_{\mathcal H_n}\quad\text{for all }\Phi\in
\mathcal H_n
\end{equation}
for some constant $c>0$ which does not depend on $n$. In particular 
$F_n\in\mathcal H_n^\star$, being $\mathcal
H_n^\star$ the dual space of $\mathcal H_n$, and $\|F_n\|_{\mathcal
H_n^\star}\leq c$ uniformly in $n$.

The idea is to apply the Lax-Milgram Theorem. In order to do this, we
remark that, using the Hardy inequality in Lemma \ref{l:hardy_boundary},
after extending functions $V_n$ trivially to zero in
$B_{R_0}^+\setminus \mathcal U_n$, the weighted $L^2$-norm of the
gradient
\begin{equation*}
\biggl(\int _{\mathcal U_n}t^{1-2s}|\nabla V_n|^2 dz\biggr)^{1/2}
\end{equation*}
turns out to be an equivalent norm in the space $\mathcal{H}_n$ that
will be denoted as
$\|\cdot\|_{\mathcal H_n}$. It follows that 
$b_n$ is coercive: indeed, for every $V\in\mathcal{H}_n$, we have that
\begin{align}\label{bcoerciva}
b_n(V,V)&=\int_{\mathcal U_n} t^{1-2s}A\nabla V\cdot\nabla V\, dz -
\kappa_s\int_{\sigma_n}\eta_n h_n |\mathop{\rm Tr} V|^2 dy\\
\notag&=\int_{B^+_{R_0}} t^{1-2s}A\nabla V\cdot\nabla V\ dz -\kappa_s\int_{B'_{R_0}}\eta_n h_n |\mathop{\rm Tr} V|^2 dy\\
\notag &\geq \tilde C_{N,s} \int_{B^+_{R_0}} t^{1-2s}|\nabla V|^2 dz=
\tilde C_{N,s}\int_{\mathcal U_n} t^{1-2s}|\nabla V|^2 dz=\tilde
C_{N,s}\Vert V\Vert^2_{\mathcal H_n},
\end{align}
as a consequence of Lemma \ref{lemma3.12FF}, with $\zeta=\eta_nh_n$,
see Remark \ref{r:r0}. Furthermore,
from \eqref{eq:8} and \eqref{remark} it follows that 
\begin{equation}\label{bcontinua}
|b_n(V,W)|\leq \bigg(2+\frac18\bigg)\|V\|_{\mathcal
H_n}\|W\|_{\mathcal H_n}\leq 3\|V\|_{\mathcal H_n}\|W\|_{\mathcal H_n}
\end{equation} 
for all $V,W\in\mathcal H_n$. In particular $b_n$ is continuous.

Hence, from \eqref{bcoerciva}, \eqref{bcontinua}, and the Lax-Milgram
Theorem we can conclude that there exists a unique $V_n\in\mathcal
H_n$ solving \eqref{wf}, which implies also the existence and
uniqueness of a solution $U_n$ to \eqref{prob5}. Moreover, 
combining \eqref{bcoerciva} and \eqref{eq:12} we
also obtain that, extending $V_n$ trivially to zero in $B_{R_0}^+\setminus \mathcal U_n$, 
\begin{equation*}
\|V_n\|_{H^1(B_{R_0}^+,t^{1-2s}\,dz)}\leq \frac{c}{\tilde C_{N,s}}\quad\text{for all }n.
\end{equation*}
From this, it follows 
that there exist $V\in H^1(B_{R_0}^+,t^{1-2s}\,dz)$ and a subsequence
$\{V_{n_k}\}$ of
$\{V_n\}$ such that 
\begin{equation}\label{eq:14}
V_{n_k}\rightharpoonup V\quad\text{weakly in
$H^1(B_{R_0}^+,t^{1-2s}\,dz)$}.
\end{equation}
From the fact that $V_n\in\mathcal H_n$, it follows easily that $V$
has null trace on $\partial^+B^+_{R_0}$ and on $\Gamma_{R_0}^+$. Hence
it can be taken as a test function in \eqref{eq:9} yielding 
\begin{equation}\label{eq:15}
\int_{B_{R_0}^+}t^{1-2s}A\nabla
W\cdot\nabla V \,dz-\kappa_s\int_{\Gamma_{R_0}^-}\tilde
h\mathop{\rm Tr}W\mathop{\rm Tr}V\,dy=0.
\end{equation}
Since $G_n\to W$ strongly in
$H^1(B_{R_1}^+,t^{1-2s}\,dz)$, from \eqref{eq:traccia-H1BR} we deduce
that $\mathop{\rm Tr}G_n\to\mathop{\rm Tr}W$ in
$L^{2^*(s)}(B_{R_1}')$. \eqref{eq:traccia-H1BR} and \eqref{eq:14}
imply that $\mathop{\rm Tr}V_{n_k}\rightharpoonup \mathop{\rm Tr}V$
weakly in
$L^{2^*(s)}(B_{R_1}')$. Furthermore $\eta_nh_n\to \tilde h$ in
$L^{\frac{N}{2s}}(\Gamma_{R_1}^-)$. Hence \eqref{eq:15} yields 
\begin{align*}
0&= \int_{B_{R_0}^+}t^{1-2s}A\nabla
W\cdot\nabla V \,dz-\kappa_s\int_{\Gamma_{R_0}^-}\tilde
h\mathop{\rm Tr}W\mathop{\rm Tr}V\,dy\\
&=
\lim_{k\to\infty}\int_{B_{R_0}^+}t^{1-2s}A\nabla
G_{n_k}\cdot\nabla V_{n_k}
\,dz-\kappa_s\int_{\Gamma_{R_0}^-}\eta_{n_k}h_{n_k}\mathop{\rm
Tr}G_{n_k}\mathop{\rm Tr}V_{n_k}\,dy\\
&=-\lim_{k\to\infty}\langle
F_{n_k},V_{n_k}\rangle =-\lim_{k\to\infty}b_{n_k}(V_{n_k},V_{n_k})
\end{align*}
thus concluding that $\|V_{n_k}\|_{H^1(B_{R_0}^+,t^{1-2s}\,dz)}\to 0$ as
$k\to+\infty$ in view of \eqref{bcoerciva}. Hence $V_{n_k}\to 0$
strongly in
$H^1(B_{R_0}^+,t^{1-2s}\,dz)$ and $U_{n_k}=V_{n_k}+G_{n_k}\to W$
as $k\to+\infty$ strongly in
$H^1(B_{R_0}^+,t^{1-2s}\,dz)$. By
Urysohn's subsequence principle, we finally conclude that
$U_n\to W$ in $H^1(B_{R_0}^+,t^{1-2s}\,dz)$ as $n\to+\infty$.
\end{proof}

\subsection{Pohozaev-type inequalities}
The aim of this section is to prove a Pohozaev-type inequality for the
energy solution $W\in H^1(B_{R_1}^+,t^{1-2s}\,dz)$ to
\eqref{prob4}; in this situation we have to settle for
an inequality instead of a classical Pohozev-type identity because
of the mixed boundary conditions, which produce some extra singular
terms with a recognizable sign when integrating the Rellich-Ne\u{c}as identity.

The idea is to obtain the inequality as limit of ones
for the approximating sequence $U_n$. For every
$r\in(0,R_0)$, $n\in\mathbb N$ such that
$n>r^{-8}$, and
$\delta\in\big(0,\frac1{4n}\big)$, let us consider the following
domain
\begin{equation*}
O_{r,n,\delta}=\mathcal U_n\cap 
\{(y,t)\in 
B_r:t>\delta\}.
\end{equation*}
We note that, if $\delta\in\big(0,\frac1{4n}\big)$, then
$f_n(t)=n^{-1/8}$ for $0\leq t\leq 2\delta$, see \eqref{eq:7}. We
can describe its topological boundary as $\partial O_{r,n,\delta}=\sigma_{r,n,\delta}\cup\gamma_{r,n,\delta}\cup\tau_{r,n,\delta},$
with
\begin{align}
\label{sigmand}
\sigma_{r,n,\delta}&=\left\{(y',y_N,t)\in B_r : y_N<\frac{1}{n^{1/8}}, \ t=\delta\right\},\\
\label{gammand}
\gamma_{r,n,\delta}&=\left\{(y',y_N,t)\in
\overline{B^+_r} : y_N=f_n(t),
\ t\geq\delta\right\},\\
\label{taund}
  \tau_{r,n,\delta}&=\left\{(y',y_N,t)\in \partial^+B^+_r: y_N<f_n(t),
                     \ t \geq \delta\right\},
\end{align}
see Figure \ref{fig:approx-domains-delta}.

\begin{figure}[ht]
\begin{tikzpicture}[line cap=round,line join=round,>=triangle 45, scale=3]
\draw [->] (-1.2,0)-- (1.2,0) node [right] {$y_N$};
\draw [->] (0,0)-- (0,1.2) node [left] {$t$};
\draw [black, fill=black, opacity=0.3] (-0.9995,0.025) to [out=88, in=180] (0,1) to [out=0, in=120] (0.866,0.5) to [out=238, in=20] (0.43,0.15) to [out=200, in=90]  (0.33,0.076) to [out=270, in=90] (0.35,0.040) to [out=270,in=90] (0.35,0.025) -- (-0.9995,0.025);
\draw  (-0.9995,0.025) to [out=88, in=180] (0,1) to [out=0, in=120] (0.866,0.5) to [out=238, in=20] (0.43,0.15) to [out=200, in=90]  (0.33,0.076) to [out=270, in=90] (0.35,0.040) to [out=270,in=90] (0.35,0.025) -- (-0.9995,0.025);
\draw (-1,0) to [out=90, in=268] (-0.9995,0.025);
\draw (0.35,0) to [out=90, in=270] (0.35,0.025);
\draw[fill] (0.35,0) circle [radius=0.010];
			\node [below] at (0.33,0) {$n^{-1/8}$};
\draw[color=black] (-0.5,0.105) node {$\sigma_{r,n,\delta}$};
\draw[color=black] (-0.7,0.9) node {$\tau_{r,n,\delta}$};
\draw[color=black] (0.7,0.13355555555555563) node {$\gamma_{r,n,\delta}$};			
\end{tikzpicture}
   \caption{Vertical section of $O_{r,n,\delta}$.}
\label{fig:approx-domains-delta}
\end{figure}
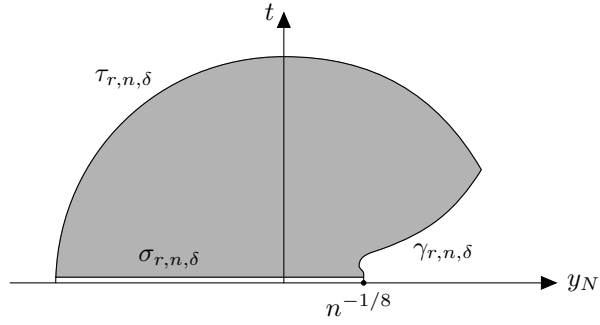

\noindent We also define
\begin{equation*}
  S^-_{r}=\{(y',y_N,t)\in \partial B_r:t=0\text{ and }y_N<0\}.
\end{equation*}
Having the matrix $A$ 
Lipschitz coefficients and being the equation satisfied in a smooth
domain containing $O_{r,n,\delta}$, by classical elliptic regularity theory
(see e.g. \cite[Theorem 2.2.2.3]{grisvard}) we have that $U_n\in
H^2(O_{r,n,\delta})$.
\begin{Proposition}[Pohozaev-type inequality]\label{p:pohozaev}
Let $W\in H^1(B_{R_1}^+,t^{1-2s}\,dz)$ weakly solve \eqref{prob4}. Then, for almost every $r\in(0,R_0)$,
\begin{align}\label{pohozaev}
& \frac{r}{2}\int_{\partial^+B_r^+}t^{1-2s}A\nabla W\cdot\nabla W\,dS- r\int_{\partial^+B_r^+}t^{1-2s}\frac{|A\nabla W\cdot\nu|^2}{\mu}\,dS\\
&\quad+\frac{\kappa_s}{2}\int_{\Gamma^-_{r}}\left(\nabla\tilde h\cdot\beta'+\tilde h\,\mathrm{div}\beta'\right)|\mathrm{Tr}W|^2dy-\frac{\kappa_sr}{2}\int_{S^-_{r}}\tilde h|\mathrm{Tr}W|^2dS'\nonumber\\
&\geq\frac{1}{2}\int_{B_r^+}t^{1-2s}A\nabla W\cdot\nabla W \ \mathrm{div}\beta\,dz-\int_{B_r^+}t^{1-2s}{J}_\beta(A\nabla W)\cdot\nabla W\,dz\nonumber\\
&\notag\quad+\frac{1}{2}\int_{B_r^+}t^{1-2s}(dA\nabla W\nabla W)\cdot\beta\,dz+\frac{1-2s}{2}\int_{B_r^+}t^{1-2s}\frac{\alpha}{\mu}A\nabla W\cdot\nabla W\,dz
\end{align}
and 
\begin{align}\label{pohozaev2}
\int_{B_r^+}t^{1-2s}A\nabla W\cdot\nabla W\,dz=\int_{\partial^+B_r^+}t^{1-2s}(A\nabla W\cdot\nu)\,W\,dS+
\kappa_s\int_{\Gamma^-_{r}}\tilde h|\mathrm{Tr}W|^2dy.
\end{align}
\end{Proposition}
\begin{remark}
The term $\int_{S^-_{r}}\tilde
h|\mathrm{Tr}W|^2dS'$ is understood for a.e. $r\in
(0,R_0)$ as the $L^1$-function given by the weak derivative of the $W^{1,1}(0,R_0)$-function
$r\mapsto
\int_{\Gamma^-_{r}}\tilde
h|\mathrm{Tr}W|^2dy$.
Likewise, the two terms
\begin{equation*}
\int_{\partial^+B_r^+}t^{1-2s}A\nabla
W\cdot\nabla W\,dS\quad\text{and}\quad \int_{\partial^+B_r^+}t^{1-2s}\frac{|A\nabla
  W\cdot\nu|^2}{\mu}\,dS
\end{equation*}
are understood for a.e. $r\in
(0,R_0)$ as the $L^1$-functions given by the weak derivative of the $W^{1,1}(0,R_0)$-functions
$r\mapsto \int_{B_r^+}A\nabla
W\cdot\nabla W\,dz$ and 
$r\mapsto \int_{B_r^+}t^{1-2s}\frac{|A\nabla W\cdot\nu|^2}{\mu}\,dz$
respectively.
\end{remark}
\begin{proof}
Since $U_n\in
H^2(O_{r,n,\delta})$, the following Rellich-Ne\u{c}as identity holds
in a distributional sense in $O_{r,n,\delta}$:
\begin{multline}\label{RN}
\mathop{\rm div}\left((\tilde A\nabla U_n\!\cdot\!\nabla
U_n)\beta-2(\beta\!\cdot\!\nabla U_n)\tilde A\nabla U_n\right)=(\tilde
A\nabla U_n\!\cdot\!\nabla U_n)\mathop{\rm div}\beta-2(\beta\!\cdot\!\nabla
U_n)\mathop{\rm div}(\tilde A\nabla U_n)\\
+(d\tilde A\nabla U_n\nabla U_n)\cdot\beta-2J_\beta(\tilde A\nabla U_n)\!\cdot\!\nabla U_n,
\end{multline}
where $\tilde A(z)=t^{1-2s}A(y)$
and $\beta$ has been defined in \eqref{beta}.
Since $U_n\in
H^2(O_{r,n,\delta})$ and $\tilde A$ and $\beta$ have Lipschitz
components, we have that 
\begin{equation*}
(\tilde A\nabla U_n\!\cdot\!\nabla
U_n)\beta-2(\beta\!\cdot\!\nabla U_n)\tilde A\nabla U_n\in
W^{1,1}(O_{r,n,\delta})
\end{equation*}
so that we can use the integration by parts formula for Sobolev
functions on the Lipschitz domain $O_{r,n,\delta}$ and obtain, in view of
\eqref{RN} and \eqref{prob5},
\begin{align*}
  &r\int_{\tau_{r,n,\delta}}t^{1-2s}A\nabla U_n\cdot\nabla
    U_n\,dS-2r\int_{\tau_{r,n,\delta}}t^{1-2s}\frac{|A\nabla
    U_n\cdot\nu|^2}{\mu}\,dS\\
  &\quad\quad-\int_{\gamma_{r,n,\delta}}\frac{t^{1-2s}}{\mu}|\partial_\nu
    U_n|^2(A\nu\cdot \nu)(Az\cdot\nu)
    \,dS
    -\int_{\sigma_{r,n,\delta}}\delta^{2-2s}\frac{\alpha}{\mu}A\nabla
    U_n\cdot\nabla U_n\,dy\\
  &\quad\quad+2\int_{\sigma_{r,n,\delta}}\delta^{1-2s}\frac{\alpha}{\mu}\partial_tU_n(D\nabla_y U_n\cdot y)\,dy+2\int_{\sigma_{r,n,\delta}}\delta^{2-2s}\frac{\alpha^2}{\mu}|\partial_t U_n|^2 \,dy\\
  &=\int_{O_{r,n,\delta}}t^{1-2s}A\nabla U_n\cdot\nabla U_n \ \mathrm{div}\beta \,dz-2\int_{O_{r,n,\delta}}t^{1-2s}J_\beta(A\nabla U_n)\cdot\nabla U_n \,dz\\
  &\quad\quad+\int_{O_{r,n,\delta}}t^{1-2s}(dA\nabla U_n\nabla U_n)\cdot\beta \,dz+(1-2s)\int_{O_{r,n,\delta}}t^{1-2s}\frac{\alpha}{\mu}A\nabla U_n\cdot\nabla U_n \,dz.
\end{align*}
In the previous computation we have used the following facts: on
$\tau_{r,n,\delta}$ the outward unit normal vector is
$\nu=\frac{z}{r}$, on $\gamma_{r,n,\delta}$ we have that
$\nabla U_n=\pm|\nabla U_n|\nu$ due to vanishing of $U_n$ on
$\gamma_{r,n,\delta}$, and on $\sigma_{r,n,\delta}$ one has
$\nu=(0,0,\dots,0,-1)$, $A\nabla U_n\cdot\nu=-\alpha\partial_tU_n$ and
\begin{equation*}
A\nabla U_n\cdot z=(D\nabla_yU_n,\alpha\partial_tU_n)\cdot
z=D\nabla_yU_n\cdot y+\alpha\delta\partial_tU_n.
\end{equation*}
From Lemma \ref{segnogamma} and uniform ellipticity of
$A$ it follows that
$$\int_{\gamma_{r,n,\delta}}\frac{t^{1-2s}}{\mu}|\partial_\nu U_n|^2(A\nu\cdot \nu)(Az\cdot\nu)\geq0.$$
Hence, we get the following inequality
\begin{align}\label{eq:17}
  &r\int_{\tau_{r,n,\delta}}\!\!t^{1-2s}A\nabla U_n\!\cdot\!\nabla
    U_n\,dS-2r\int_{\tau_{r,n,\delta}}\!\!t^{1-2s}\frac{|A\nabla
    U_n\!\cdot\!\nu|^2}{\mu}\,dS
    -\int_{\sigma_{r,n,\delta}}\delta^{2-2s}\frac{\alpha}{\mu}A\nabla
    U_n\!\cdot\!\nabla U_n\,dy\\
  \notag&\quad\quad+2\int_{\sigma_{r,n,\delta}}\delta^{1-2s}\frac{\alpha}{\mu}\partial_tU_n(D\nabla_y U_n\cdot y)\,dy+2\int_{\sigma_{r,n,\delta}}\delta^{2-2s}\frac{\alpha^2}{\mu}|\partial_t U_n|^2 \,dy\\
  \notag&\geq\int_{O_{r,n,\delta}}t^{1-2s}A\nabla U_n\cdot\nabla U_n \ \mathrm{div}\beta \,dz-2\int_{O_{r,n,\delta}}t^{1-2s}J_\beta(A\nabla U_n)\cdot\nabla U_n \,dz\\
  \notag&\quad\quad+\int_{O_{r,n,\delta}}t^{1-2s}(dA\nabla U_n\nabla U_n)\cdot\beta \,dz+(1-2s)\int_{O_{r,n,\delta}}t^{1-2s}\frac{\alpha}{\mu}A\nabla U_n\cdot\nabla U_n \,dz.
\end{align}
We are going to pass to the limit as $\delta\to0$. We denote as $O_{r,n}$
the limit domain with its boundary
$\partial O_{r,n}=\sigma_{r,n}\cup\gamma_{r,n}\cup\tau_{r,n}$, 
i.e. 
\begin{align*}
& O_{r,n}=\mathcal U_n\cap B_r,\quad \tau_{r,n}=\left\{(y',y_N,t)\in \partial B_r: y_N<f_n(t),
\ t\geq0\right\},\\
&\gamma_{r,n}=\left\{(y',y_N,t)\in
\overline{B^+_r} : y_N=f_n(t)\right\},\quad 
\sigma_{r,n}=\left\{(y',y_N)\in B_r' :
y_N<n^{-1/8}\right\}.
\end{align*}
We claim that, for every fixed 
$r\in(0,R_0)$ and $n>r^{-8}$, there exists a sequence $\delta_k\to
0^+$ such that
\begin{equation*}
-\int_{\sigma_{r,n,\delta_k}}\delta_k^{2-2s}\frac{\alpha}{\mu}A\nabla
U_n\cdot\nabla
U_n\,dy+2\int_{\sigma_{r,n,\delta_k}}\delta_k^{2-2s}\frac{\alpha^2}{\mu}|\partial_t
U_n|^2\,dy\to0\quad\text{as } k\to\infty.
\end{equation*}
Since $\alpha$ is bounded,
$\mu\geq1/4$ in $B_{R_0}$, and $A$ has bounded coefficients, it is
enough to prove that there exists a sequence $\delta_k\to 0^+$ such
that
$\lim_{k\to\infty}\int_{\sigma_{r,n,\delta_k}}\delta_k^{2-2s}|\nabla
U_n|^2\,dy=0$.
To prove this, we argue by contradiction and assume that there
exist a positive constant $c>0$ and $\delta_0>0$ such that, for any
$\delta\in(0,\delta_0)$,
\begin{equation*}
\frac{c}{\delta}\leq \int_{\sigma_{r,n,\delta}}\delta^{1-2s}|\nabla
U_n(y,\delta)|^2\,dy,
\end{equation*}
which, after  integration in $(0,\delta_0)$, gives the contradiction
\begin{equation*}
\int_0^{\delta_0}\frac{c}{\delta}\,d\delta\leq\int_0^{\delta_0}\delta^{1-2s}\left(\int_{\sigma_{r,n,\delta}}|\nabla
U_n(y,\delta)|^2\,dy\right)\,d\delta\leq\|U_n\|^2_{H^1(B_{R_0}^+,t^{1-2s}\,dz)},
\end{equation*}
since the first integral diverges.

In order to prove the convergence 
\begin{equation*}
2\int_{\sigma_{r,n,\delta}}\delta^{1-2s}\frac{\alpha}{\mu}\partial_tU_n(D\nabla_y
U_n\cdot y)\,dy
\mathop{\longrightarrow}_{\delta\to0}-2\kappa_s\int_{\sigma_{r,n}}\frac{1}{\mu}\eta_nh_n\mathrm{Tr}U_n(D\nabla_y
\mathrm{Tr}U_n\cdot y) \,dy,
\end{equation*}
we exploit a continuity result for
  $t^{1-2s}\partial_tU_n$ and $\nabla_y U_n$ over
  $\overline{\mathcal U_n\cap B_r}$, which allows us to pass to the limit by
the Dominated Convergence Theorem. 
More precisely we claim that, for all $r\in(0,R_0)$
  and $n>r^{-8}$, 
  \begin{equation}\label{eq:16}
    t^{1-2s}\partial_tU_n\in C^0(\overline{\mathcal U_n\cap
      B_r}),\quad  \nabla_y U_n\in C^0(\overline{\mathcal U_n\cap B_r}).
  \end{equation}   
The continuity of $t^{1-2s}\partial_tU_n$ and $\nabla_y U_n$ away from
$\{t=0\}$ easily follows from classical elliptic regularity theory, since $U_n$ is solution of an uniformly
elliptic equation. Nevertheless, Lemma 3.3 in \cite{FalFel} allows us
to prove continuity of $t^{1-2s}\partial_tU_n$ and
  $\nabla_y U_n$ up to $\{t=0\}$ when we stay away from the
corner between $\sigma_{r,n}$ and $\gamma_{r,n}$,
  i.e. away from the edge $\{(y',y_N,t)\in \overline{B_r}:t=0\text{
    and }y_N=n^{-1/8}\}$: to this aim it is enough to
  apply \cite[Lemma 3.3]{FalFel} 
to the function $U_n\circ F^{-1}$. Eventually, we can
deduce continuity of $t^{1-2s}\partial_tU_n$ and
  $\nabla_y U_n$ also in the set $\{(y',y_N,t)\in \overline{B_r}:t\in[0,\frac1{2n}]\text{
    and }y_N\in[0,n^{-1/8}]\}$ as a consequence of the regularity
  result given in  Lemma
  \ref{regularitycorner} applied to the function $U_n\circ F^{-1}$.

Using the fact that, for all $r\in(0,R_0)$
  and $n>r^{-8}$, the terms integrated over
$\tau_{r,n,\delta}$ belong to $L^1(\tau_{r,n})$ in view of \eqref{eq:16}
 and the terms integrated over $O_{r,n,\delta}$ belong to
 $L^1(\mathcal U_n\cap B_r)$ since 
$U_n\in H^1(\mathcal U_n,t^{1-2s}\,dz)$,
then by absolute continuity of the Lebesgue
integral,
  we can pass to the
  limit in \eqref{eq:17} along $\delta=\delta_k$ as $k\to\infty$, thus ending up with the inequality
\begin{align}\label{pohon}
  &r\int_{\tau_{r,n}}\!\!t^{1-2s}A\nabla U_n\!\cdot\!\nabla U_n\,dS-
    2r\int_{\tau_{r,n}}\!\!t^{1-2s}\frac{|A\nabla U_n\!\cdot\!\nu|^2}{\mu}\,dS-2\kappa_s\int_{\sigma_{r,n}}\frac{1}{\mu}
    \eta_nh_n\mathrm{Tr}U_n(D\nabla_y \mathrm{Tr}U_n\!\cdot\! y)\,dy\\
  &\geq\int_{\mathcal U_n\cap B_r}\!\!t^{1-2s}A\nabla U_n \!\cdot \!\nabla U_n 
    \mathrm{div}\beta\,dz-2\int_{\mathcal U_n\cap
    B_r}\!\!t^{1-2s}J_\beta(A\nabla U_n)
    \cdot\nabla U_n\,dz\nonumber\\
  &\quad\quad+\int_{\mathcal U_n\cap B_r}\!\!t^{1-2s}(dA\nabla U_n\nabla
    U_n)
    \!\cdot \!\beta\,dz+(1-2s)\int_{\mathcal U_n\cap B_r}\!\!t^{1-2s}
    \frac{\alpha}{\mu}A\nabla U_n \!\cdot \!\nabla U_n\,dz, \nonumber
\end{align}
for all $r\in(0,R_0)$ and $n>r^{-8}$.
For $r\in(0,R_0)$ fixed, we are going
to pass to the limit in \eqref{pohon} as $n\to+\infty$. We extend the functions
$U_n$ to be zero in $B_r^+\setminus \mathcal U_n$. By the strong
convergence $U_n\to W$ in $H^1(B_{R_0}^+,t^{1-2s}dz)$ (see Proposition
\ref{existenceuniqueness}), it follows  that
\begin{equation*}
\int_0^{R_0}\left(\int_{\partial^+B_r^+}t^{1-2s}\left(|\nabla
    (U_n-W)|^2+|U_n-W|^2\right)\,dS\right)dr\to0,
\end{equation*}
i.e. the sequence of functions $u_n(r)=\int_{\partial^+B_r^+}t^{1-2s}\left(|\nabla
    (U_n-W)|^2+|U_n-W|^2\right)\,dS$ converges to $0$ in $L^1(0,R_0)$ and hence
    a.e. along a subsequence $u_{n_k}$. In particular we have that
\begin{equation}\label{eq:18}
U_{n_k}\to  W\quad\text{as $k\to\infty$ in
    $H^1(\partial^+B_r^+,t^{1-2s}\,dS)$ for a.e. $r\in(0,R_0)$,}
\end{equation}
where $H^1(\partial^+B_r^+,t^{1-2s}\,dS)$ is the completion of
$C^\infty(\overline{\partial^+B_r^+})$ with respect to the norm
\begin{equation*}
\|\psi\|_{H^1(\partial^+B_r^+,t^{1-2s}\,dS)}=\bigg(
\int_{\partial^+B_r^+}t^{1-2s}\left(|\nabla
    \psi|^2+\psi^2\right)\,dS\bigg)^{\!\!1/2}.
\end{equation*}
Let us now discuss the behavior of the term  $\int_{\sigma_{r,n}}\frac{1}{\mu}
    \eta_nh_n\mathrm{Tr}U_n(D\nabla_y \mathrm{Tr}U_n\cdot y)\,dy$ as
    $n\to\infty$.
Since $\eta_n(y',y_N)=0$ if $y_N>-\frac1n$, by the Divergence Theorem we have that 
\begin{align}\label{eq:24}
\int_{\sigma_{r,n}}&\frac{1}{\mu}
    \eta_nh_n\mathrm{Tr}U_n(D\nabla_y \mathrm{Tr}U_n\cdot y)\,dy=
\int_{\Gamma^-_{r}}\frac{1}{\mu}\eta_nh_n\mathrm{Tr}U_n(D\nabla_y
  \mathrm{Tr}U_n\cdot y) \,dy\\
\notag&=\frac{1}{2}\int_{\Gamma^-_{r}}\mathrm{div}_y\left(\frac{1}{\mu}\eta_nh_n|\mathrm{Tr}U_n|^2Dy\right) \,dy-\frac{1}{2}\int_{\Gamma^-_{r}}|\mathrm{Tr}U_n|^2\mathrm{div}_y\left(\frac{1}{\mu}\eta_nh_nDy\right) \,dy\\
\notag&=\frac{1}{2}\int_{S^-_{r}}\frac{1}{\mu}\eta_nh_n|\mathrm{Tr}U_n|^2Dy\cdot\nu\,dS'-\frac{1}{2}\int_{\Gamma^-_{r}}|\mathrm{Tr}U_n|^2\mathrm{div}_y\left(\eta_nh_n\beta'\right) \,dy\\
\notag&=\frac{r}{2}\int_{S^-_{r}}\eta_nh_n|\mathrm{Tr}U_n|^2\,dS'-\frac{1}{2}\int_{\Gamma^-_{r}}|\mathrm{Tr}U_n|^2\left(\eta_n\nabla_yh_n\cdot\beta'+\eta_nh_n\mathrm{div}_y\beta'\right) \,dy\\
\notag&\qquad-\frac{1}{2}\int_{\Gamma^-_{r}}|\mathrm{Tr}U_n|^2h_n\nabla_y\eta_n\cdot\beta' \,dy,
\end{align}
where $\beta'$ has been defined in \eqref{beta}.
 From the strong
convergence $U_n\to W$ in $H^1(B_{R_0}^+,t^{1-2s}dz)$ proved in Proposition
\ref{existenceuniqueness} and \eqref{eq:traccia-H1BR}, it follows  that
\begin{equation*}
\int_0^{R_0}\left(\int_{S_r^-}\left(\eta_nh_n|\mathrm{Tr}U_n|^2-\tilde
    h|\mathrm{Tr}W|^2\right)\,dS'\right) dr\to0,
\end{equation*}
i.e. the sequence of functions
$r\mapsto\int_{S_r^-}\left(\eta_nh_n|\mathrm{Tr}U_n|^2-\tilde
  h|\mathrm{Tr}W|^2\right)\,dS'$
converges to $0$ in $L^1(0,R_0)$ and hence a.e. along a further subsequence,
which we still index by $n_k$. In particular we have that
\begin{equation}\label{eq:19}
\int_{S_r^-}\eta_{n_k}h_{n_k}|\mathrm{Tr}U_{n_k}|^2\,dS'\to
\int_{S_r^-}\tilde h|\mathrm{Tr}W|^2\,dS'\quad\text{as $k\to\infty$ for a.e. }r\in(0,R_0).
\end{equation}
The strong
convergence of $U_n$ to $W$ in $H^1(B_{R_0}^+,t^{1-2s}dz)$, which
implies that $\mathrm{Tr}U_n\to \mathrm{Tr}W$ in $L^{2^*(s)}(B_{R_0}')$ by 
\eqref{eq:traccia-H1BR}, 
the strong convergence \eqref{eq:20} of $h_n$ to $\tilde h$ in
$W^{1,p}(\Gamma_{R_0}^-)$, and the fact that 
$\eta_n\to1$ a.e. in $\Gamma^-_{R_0}$ imply that 
\begin{equation}\label{eq:21}
\lim_{n\to\infty}\int_{\Gamma^-_{r}}|\mathrm{Tr}U_n|^2\left(\eta_n\nabla_yh_n\cdot\beta'+\eta_nh_n\mathrm{div}_y\beta'\right) dy=
\int_{\Gamma^-_{r}}|\mathrm{Tr}W|^2\left(\nabla_y\tilde
  h\cdot\beta'+\tilde h\mathrm{div}_y\beta'\right)dy
\end{equation}
for all $r\in(0,R_0)$.
Finally, we have that, by \eqref{eq:22} and \eqref{beta}, 
\begin{equation*}
\nabla_y\eta_n\cdot \beta'=\frac1\mu\,\frac n{2} \,\eta'\big(-\tfrac{n
  y_N}2\big)(D(y)y)_N.
\end{equation*}
Hence, since \eqref{eq:6} implies that $(D(y)y)_N=O(y_N)$ as $y_N\to0$
and \eqref{eq:13} yields that $\eta'\big(-\tfrac{n
  y_N}2\big)\neq 0$ only for $y_N\in\big(-\frac2n,-\frac1n\big)$, we
conclude that 
\begin{equation*}
\nabla_y\eta_n\cdot \beta'\quad\text{is bounded in $\Gamma_{r}^-$
  uniformly with respect to $n$}.
\end{equation*}
Therefore, by H\"older's inequality,
\begin{multline*}
\left|\int_{\Gamma^-_{r}}|\mathrm{Tr}U_n|^2h_n\nabla_y\eta_n\cdot\beta'
  \,dy\right|\\
\leq {\rm
  const\,}\|\mathrm{Tr}U_n\|_{L^{2^*(s)}(\Gamma^-_{r})}^2\|h_n\|_{L^p
  (\Gamma^-_{r})}\Big|\{(y',y_N)\in \Gamma^-_{r}:-\tfrac2n <y_N<-\tfrac1n\}\Big|_N^{\frac{2s}{pN}(p-\frac{N}{2s})}
\end{multline*}
where $|\cdot|_N$ stands for the $N$-dimensional Lebesgue measure; 
hence, by boundedness of  $\{\mathrm{Tr}U_n\}$ in
$L^{2^*(s)}(\Gamma_r^-)$, ensured by 
\eqref{eq:traccia-H1BR}, and boundedness  of $\{h_n\}$ in
$L^p(\Gamma_{r}^-)$, we conclude that 
\begin{equation}\label{eq:23}
\lim_{n\to\infty}\int_{\Gamma^-_{r}}|\mathrm{Tr}U_n|^2h_n\nabla_y\eta_n\cdot\beta' \,dy=0
\end{equation}
Combining \eqref{eq:19}, 
\eqref{eq:21}, and \eqref{eq:23}, we can pass to the limit in
\eqref{eq:24} along the subsequence, obtaining that 
\begin{multline}\label{eq:25}
\lim_{k\to\infty}\int_{\sigma_{r,n_k}}\frac{1}{\mu}
    \eta_{n_k}h_{n_k}\mathrm{Tr}U_{n_k}(D\nabla_y
    \mathrm{Tr}U_{n_k}\cdot y)\,dy\\
=
\frac r2 \int_{S_r^-}\tilde h|\mathrm{Tr}W|^2\,dS'-\frac12 
\int_{\Gamma^-_{r}}|\mathrm{Tr}W|^2\left(\nabla_y\tilde
  h\cdot\beta'+\tilde h\mathrm{div}_y\beta'\right)dy.
\end{multline}
In view of \eqref{eq:18}, \eqref{eq:25}, and  the strong
convergence of $U_n$ to $W$ in $H^1(B_{R_0}^+,t^{1-2s}dz)$, we can
pass to the limit as $n=n_k\to\infty$ in \eqref{pohon} obtaining
the desired Pohozaev-type inequality \eqref{pohozaev} for the solution $W$.

Finally, to prove \eqref{pohozaev2}, we first multiply equation
  \eqref{eq:26} by $U_n$ and integrate by parts over
  $O_{r,n,\delta}$; then we pass to the limit as $\delta\to0^+$ using
  \eqref{eq:16}  and as $n=n_k\to\infty$, taking into account \eqref{eq:18}.
\end{proof}

\section{Almgren type frequency function}\label{sec:almgr-type-freq}
In this section we analyze the properties of the Almgren
frequency function $\mathcal N(r)$ associated to \eqref{prob4}, see \eqref{N}. To perform a blow-up analysis, the
boundedness of the frequency  will be
crucial; to this aim we are going to prove that $\mathcal N$ possesses a
nonnegative finite limit as $r\to0^+$.

Let $W\in H^1_{\Gamma_{R_1}^+}(B_{R_1}^+,t^{1-2s}\,dz)$ be a
nontrivial weak solution of \eqref{prob4}.
For all $r\in(0,R_1)$ we define 
\begin{equation}\label{E}
E(r)=\frac{1}{r^{N-2s}}\left(\int_{B_r^+}t^{1-2s}A\nabla W\cdot\nabla
  W\, dz-\kappa_s\int_{\Gamma^-_{r}}\tilde h|\mathop{\rm Tr}W|^2\, dy\right)
\end{equation}
and
\begin{equation}\label{H}
H(r)=\frac{1}{r^{N+1-2s}}\int_{\partial^+B_r^+}t^{1-2s}\mu(z) W^2(z)\, dS.
\end{equation}
Let us first estimate the derivative of $H$. 
\begin{Lemma}\label{lemH}
Let $E$ and $H$ be the functions defined as in \eqref{E} and
\eqref{H}. Then
 $H\in W^{1,1}_{\mathrm{loc}}(0,R_1)$ and
\begin{equation}\label{H'}
H'(r)=\frac{2}{r^{N+1-2s}}\int_{\partial^+B_r^+}t^{1-2s}\mu W \frac{\partial W}{\partial \nu}\, dS +H(r)O(1)\quad \text{as $r\rightarrow 0^+$}
\end{equation} 
in a distributional sense and for a.e. $r\in (0,R_1)$, where $\nu=\nu(z)=\frac{z}{|z|}$ denotes the unit outer normal vector to $\partial^+B_r^+$. Moreover 
\begin{equation}\label{eq1}
H'(r)=\frac{2}{r^{N+1-2s}}\int_{\partial^+B_r^+}t^{1-2s} (A\nabla W\cdot\nu)W\, dS+H(r)O(1)
\end{equation}
and
\begin{equation}\label{E=rH'/2}
H'(r)=\frac{2}{r}E(r)+H(r)O(1)
\end{equation}
as $r\rightarrow 0^+$.
\end{Lemma}
\begin{proof}
We observe that $H\in L^1_{\mathrm{loc}}(0,R_1)$ by definition and it can be rewritten as
\begin{equation*}
  H(r)=\int_{{\mathbb S}^{N}_+}\theta_{N+1}^{1-2s}\mu(r\theta)W^2(r\theta)\, dS.
\end{equation*}
Thus, for all test functions $\varphi\in C^\infty_c(0,R_1)$, we have that
\begin{equation*}
\begin{split}
  -\int_0^{R_1}
  H(r)&\varphi'(r)dr=-\int_0^{R_1}
  \biggl(\int_{\mathbb S^N_+}\theta_{N+1}^{1-2s}\mu(r\theta)W^2(r\theta)\, dS\biggr)\varphi'(r)dr\\
  &=-\int_{B_{R_1}^+}t^{1-2s}\frac{\mu(z)W^2(z)}{|z|^{N+2-2s}}\nabla
  \tilde{\varphi}(z)\cdot z\ dz
  =\int_{B_{R_1}^+}\mathrm{div}\biggl(\frac{t^{1-2s}\mu(z)W^2(z)z}{|z|^{N+2-2s}}\biggr) \tilde{\varphi}(z) dz\\
  &=\int_{B_{R_1}^+}t^{1-2s}\biggl(\frac{2\mu(z)W(z)\nabla
    W(z)+W^2(z)\nabla \mu(z)}{|z|^{N+2-2s}}\cdot z\biggr)
  \tilde{\varphi}(z) dz\\
  &=\int_0^{R_1}\biggl(\int_{\mathbb
    S^N_+}\theta_{N+1}^{1-2s}\bigl[2\mu(r\theta)W(r\theta)\nabla
  W(r\theta)\cdot\theta+W^2(r\theta)\nabla\mu(r\theta)\cdot\theta\bigr]dS\biggr)\varphi(r)\,dr,
\end{split}
\end{equation*}
where $\tilde{\varphi}(z):=\varphi(|z|)$, so that
$\varphi'(|z|)=\nabla \tilde{\varphi}(z)\cdot \frac{z}{|z|}$. Hence
the distributional derivative of $H$ in $(0,R_1)$ is given by
\begin{equation}\label{H'distribuz}
H'(r)=\frac{2}{r^{N+1-2s}}\int_{\partial^+B_r^+}t^{1-2s}\mu W\frac{\partial W}{\partial\nu}\, dS+\frac{1}{r^{N+1-2s}}\int_{\partial^+B_r^+}t^{1-2s}W^2\nabla\mu\cdot\nu\, dS.
\end{equation}
Since $W,\nabla W\in L^2(B^+_{R_1},t^{1-2s}dz)$, from
\eqref{sviluppomu} and \eqref{sviluppogradmu} we easily infer that
$H\in W^{1,1}_{\mathrm{loc}}(0,R_1)$ and  \eqref{H'distribuz} also
holds for a.e. $r\in (0,R_1)$. Moreover, combining
\eqref{sviluppomu}, \eqref{sviluppogradmu}, \eqref{H} and
\eqref{H'distribuz}, we obtain \eqref{H'}.

 In order to prove \eqref{eq1}, we define 
\begin{equation*}\label{gamma}
\gamma(z)=\frac{\mu(z)(\beta(z)-z)}{|z|}.
\end{equation*}
Observing that $\gamma(z)\cdot z=0$ by definition, we deduce that,
for a.e. $r\in(0,R_1)$,
\begin{align}\label{eq2}
\int_{\partial^+B^+_r}t^{1-2s}(A\nabla W\cdot \nu) W\, dS&=\int_{\partial^+B^+_r}t^{1-2s} \mu W\frac{\partial W}{\partial\nu}dS+\frac{1}{2}\int_{\partial^+B^+_r}t^{1-2s}\gamma\cdot\nabla(W^2)\, dS\\
\notag&=\int_{\partial^+B^+_r}t^{1-2s} \mu W\frac{\partial W}{\partial\nu}dS-\frac{1}{2}\int_{\partial^+B_r^+}\mathrm{div}(t^{1-2s}\gamma)W^2\, dS\\
\notag&=\int_{\partial^+B^+_r}t^{1-2s} \mu W\frac{\partial W}{\partial\nu}dS+H(r)O(r^{N+1-2s}),
\end{align}
where we used that 
\begin{equation*}
\mathrm{div}(t^{1-2s}\gamma)= t^{1-2s}\mathrm{div}\gamma+(1-2s)\gamma_{N+1}t^{-2s}
\end{equation*}
and
\begin{equation*}
\begin{aligned}
&\gamma_{N+1}(z)=t\,O(1) \quad \text{as $|z|\rightarrow 0^+$},\\
&\mathrm{div}\gamma=\biggl(\frac{\nabla \mu(z)}{|z|}-\frac{\mu(z) z}{|z|^3}\biggr)(\beta(z)-z)+\frac{\mu(z)}{|z|}(\mathrm{div}\beta-(N+1))=O(1)\quad \text{as $|z|\rightarrow 0^+$},
\end{aligned}
\end{equation*}
as a consequence of \eqref{beta}, \eqref{eq:det-jac},
\eqref{sviluppomu}, \eqref{sviluppogradmu}, \eqref{sviluppobeta},
\eqref{sviluppodivbeta}. Hence, from \eqref{H'} and \eqref{eq2} it
follows \eqref{eq1}. From \eqref{E}, \eqref{pohozaev2} and \eqref{eq1} we infer that 
\begin{equation*}
r^{N-2s}E(r)=\int_{\partial^+B_r^+}t^{1-2s} (A\nabla W\cdot\nu)W\, dS=\frac{r^{N+1-2s}}{2}H'(r)+H(r)O(r^{N+1-2s}),
\end{equation*}
as $r\rightarrow 0^+$, which gives \eqref{E=rH'/2}, thus proving the lemma. 
\end{proof}
\begin{Lemma}\label{lemmaH>0}
 The function $H$ defined as in \eqref{H} is strictly positive for
 every $0<r\leq R_0$, with $R_0$ being defined in \eqref{eq:R_0}.
\end{Lemma}
\begin{proof}
We prove the statement by contradiction. To this aim, we suppose that
there exists $\overline{R}\leq R_0$ such that $H(\overline{R})=0$. Then, since for every $r\leq R_0$ it holds that $\mu\geq 1/4$ in $B_r$,
$\int_{\partial^+B_{\overline{R}}^+} t^{1-2s} W^2 = 0$ and hence
$W\equiv 0$ on $\partial^+B_{\overline{R}}^+$. 
From \eqref{E=rH'/2} it follows that $H$ is differentiable in a
classical sense in $\overline{R}$ and
$H'(\overline{R})=2 \overline{R}^{-1} E(\overline{R})$; on the other
hand, $H(r)\geq 0=H(\overline{R})$ implies that
$0=H'(\overline{R})=2 \overline{R}^{-1}E(\overline{R})$ and hence
$E(\overline{R})=0$. 
Then from \eqref{coercvity} it follows that
\begin{equation}\label{eq3}
0=\int_{B^+_{\overline{R}}}t^{1-2s}A\nabla W\cdot\nabla
W\,dz-\kappa_s\int_{\Gamma^-_{\overline{R}}}\tilde h|\mathop{\rm
  Tr}W|^2\, dy\geq \tilde C_{N,s} \int_{B^+_{\overline{R}}}t^{1-2s}|\nabla W|^2\,dz.
\end{equation} 
By \eqref{eq3} and Lemma \ref{l:hardy_boundary} we can conclude that
$W\equiv 0$ in $B_{\overline{R}}^+$, which in turn leads to
$W\equiv 0$ in $B_{R_1}^+\cap\{t>\delta\}$ from classical unique
continuation principles for second order elliptic equations with
Lipschitz coefficients (see \cite{Garlin}).  Since $\delta>0$ can be
taken arbitrarily small, we end up with $W\equiv 0$ in $B_{R_1}^+$,
which is a contradiction.
\end{proof}
As a consequence of Lemma \ref{lemmaH>0}, the \emph{Almgren type frequency} function 
\begin{equation}\label{N}
\mathcal N(r)=\frac{E(r)}{H(r)}
\end{equation}
is well defined in $(0,R_0]$, with $R_0$ as in \eqref{eq:R_0}. 

In the following lemma we provide an estimate for the derivative of the function $E$.
\begin{Lemma}\label{lemE}
Let $E$ be the function defined as in \eqref{E}. Then $E\in W^{1,1}_{\mathrm{loc}}((0,R_0])$ and 
\begin{equation}\label{E'>=}
E'(r)\geq \frac{2}{r^{N-2s}}\int_{\partial^+B_r^+}t^{1-2s}\frac{|A\nabla W\cdot\nu|^2}{\mu}+O(r^{-1+\bar\delta})\biggl[E(r)+\frac{N-2s}{2}H(r)\biggr]\quad\text{as $r\rightarrow 0^+$}
\end{equation}
for a.e. $r\in (0,R_0)$, where 
\begin{equation}\label{delta}
\bar\delta=\min\{\overline\varepsilon,1\}\in(0,1]
\end{equation}
and $\overline{\varepsilon}$ is defined as in \eqref{overepsilon}.
\end{Lemma}
\begin{proof}
From \eqref{E} we deduce that $E\in L^1_{\mathrm{loc}}(0,R_0)$.
From coarea formula $E\in W^{1,1}_{\mathrm{loc}}((0,R_0])$ and
\begin{align}\label{E'=}
E'(r)=&\frac{2s-N}{r^{N+1-2s}}\left(\int_{B_r^+}t^{1-2s}A\nabla
  W\cdot\nabla W\, dz-\kappa_s\int_{\Gamma^-_{r}}\tilde h|\mathop{\rm Tr}W|^2\,dy\right)\\
\notag&+\frac{1}{r^{N-2s}}\left(\int_{\partial^+B_r^+}t^{1-2s}A\nabla W\cdot\nabla W\, dS-\kappa_s\int_{S^-_{r}}\tilde h|\mathop{\rm Tr}W|^2\,dS'\right)
\end{align}
in a distributional sense and a.e. in $(0,R_0)$. Using \eqref{pohozaev}, Lemma \ref{lemmastimamu} and Lemma \ref{lemmabeta}, we obtain that
\begin{align}\label{E'2}
E'(r)\geq &\frac{2}{r^{N-2s}}\int_{\partial^+B_r^+}t^{1-2s}\frac{|A\nabla W\cdot\nu|^2}{\mu}\,dS+\frac{O(r)}{r^{N+1-2s}}\int_{B^+_r}t^{1-2s}A\nabla W\cdot\nabla W\, dz \\
\notag&+\frac{O(1)}{r^{N+1-2s}}\int_{\Gamma^-_{r}}(|\tilde{h}|+|\nabla
        _y\tilde{h}|)|\mathop{\rm Tr}W|^2\,dy
\end{align}
as $r\rightarrow 0^+$, for a.e. $r\in (0,R_0)$. One can estimate the
last two terms of the right hand side in \eqref{E'2} using
\eqref{coercvity}, thus obtaining 
\begin{equation}\label{term1}
\frac{O(r)}{r^{N+1-2s}}\int_{B^+_r}t^{1-2s}A\nabla W\cdot\nabla W\, dz =O(1)\biggl[E(r)+\frac{N-2s}{2}H(r)\biggr]
\end{equation}
and
\begin{align}\label{term2}
\frac{O(1)}{r^{N+1-2s}}\int_{\Gamma^-_{r}}(|\tilde{h}|+|\nabla
_y\tilde{h}|)W^2\ dy&
=\frac{O(r^{\overline{\varepsilon}})}{r^{N+1-2s}}\biggl(\int_{\Gamma^-_{r}}|\mathop{\rm
  Tr}W|^{2^\ast(s)}\,dy\biggr)^{2/2^\ast(s)}\\
\notag&=O(r^{-1+\overline{\varepsilon}})\biggl[E(r)+\frac{N-2s}{2}H(r)\biggr],
\end{align}
where in \eqref{term2} we have taken into account \eqref{useful} as
well. Estimate \eqref{E'>=} follows from \eqref{E'2},  
\eqref{term1}, and \eqref{term2}.
\end{proof}

\begin{Lemma}\label{lemmaliminf}
Let $\mathcal N$ be the function defined  in \eqref{N}. Then, for every
$0<r\leq R_0$, 
\begin{equation}\label{lower} 
\mathcal N(r)>-\frac{N-2s}{2}
\end{equation}
and
\begin{equation}\label{liminf}
\liminf_{r\rightarrow 0^+} \mathcal N(r)\geq 0.
\end{equation}
\end{Lemma}
\begin{proof}
We deduce \eqref{lower} from \eqref{coercvity}. By \eqref{E},
\eqref{H}, \eqref{remark} and \eqref{eqcoerc1}, we obtain that for
all $0<r\leq R_0$
\begin{align*}
r^{N-2s}&E(r)=\int_{B_r^+}t^{1-2s}A\nabla W\cdot\nabla W\, dz-\kappa_s\int_{\Gamma_r^-}\tilde h |\mathop{\rm
  Tr}W|^2\, dy\\
&\geq\frac{3}{8}\int_{B_r^+}t^{1-2s}|\nabla W|^2\, dz-\kappa_s
\tilde{S} _{N,s}\tilde c_{N,s,p} 
r^{\overline{\varepsilon}}\bar{\alpha}_0\frac{2(N-2s)}{r}\int_{\partial^+B_r^+}t^{1-2s}\mu W^2\, dS\geq-\tilde{C}r^{\overline\varepsilon+N-2s}H(r),
\end{align*}
with $\bar{\alpha}_0$ as in \eqref{eq:alpha0} and 
$\tilde{C}:=2(N-2s)\kappa_s\tilde{S}_{N,s}\tilde c_{N,s,p}
\bar{\alpha}_0>0$. From this and \eqref{N} it follows that, for every
$0<r\leq R_0$,
\begin{equation*}
\N(r)\geq-\tilde{C} r^{\overline\varepsilon},
\end{equation*}
which leads to \eqref{liminf}.
\end{proof}

\begin{Lemma}\label{deriveN}
Let $\N$ be the function defined in \eqref{N}. Then $\N\in W^{1,1}_{\mathrm{loc}}((0,R_0])$ and
\begin{equation}\label{N'>=nu1+nu2}
\N'(r)\geq V_1(r)+V_2(r)
\end{equation}
for almost every $r\in(0,R_0)$, where
\begin{equation*}
V_1(r)=\frac{2r\left[\left(\int_{\partial^+B_r^+}t^{1-2s}\frac{|A\nabla W\cdot\nu|^2}{\mu}\,dS\right)\left(\int_{\partial^+B_r^+}t^{1-2s}\mu W^2 \,dS\right)\!-\!\left(\int_{\partial^+B_r^+}t^{1-2s}(A\nabla W\!\cdot\!\nu) W \,dS\right)^2\right]}{\left(\int_{\partial^+B_r^+}t^{1-2s}\mu W^2\right)^2}
\end{equation*}
and
\begin{equation*}
V_2(r)=O(r^{-1+\bar\delta})\left(\N(r)+\frac{N-2s}{2}\right)\quad\text{as $r\rightarrow 0^+$},
\end{equation*}
with $\bar\delta$ as in \eqref{delta}.
\end{Lemma}
\begin{proof}
The fact that $\N\in W^{1,1}_{\mathrm{loc}}((0,R_0])$
  follows from Lemmas \ref{lemH}, \ref{lemmaH>0},  and \ref{lemE}.
Moreover, exploiting \eqref{eq1}, \eqref{E=rH'/2} and \eqref{E'>=}, we obtain that
\begin{align}\label{eqN'1}
  &\N'(r)=\frac{E'(r)H(r)-H'(r)E(r)}{H^2(r)}=\frac{E'(r)H(r)}{H^2(r)}-\frac{H'(r)}{H^2(r)}\left(\frac{r}{2}H'(r)+H(r)O(r)\right)\\
  \notag&\geq\frac{2r\left[\left(\int_{\partial^+B_r^+}\!t^{1-2s}\frac{|A\nabla W\cdot\nu|^2}{\mu}dS\right)\!\left(\int_{\partial^+B_r^+}\!t^{1-2s}\mu W^2dS\right)\!-\!\left(\int_{\partial^+B_r^+}\!t^{1-2s}(A\nabla W\cdot\nu) WdS\right)^{\!2}\right]}{\left(\int_{\partial^+B_r^+}t^{1-2s}\mu W^2dS\right)^{\!2}}\\
  &\notag\qquad+O(r^{-1+\bar\delta})\left(\N(r)+\frac{N-2s}{2}\right)+O(r)+O(1)\frac{1}{H(r)}\frac{1}{r^{N-2s}}\int_{\partial^+B_r^+}t^{1-2s}(A\nabla W\cdot\nu) W\,dS.
\end{align}
In order to estimate the last term in \eqref{eqN'1}, we observe that
\begin{equation}\label{eqN'2}
O(1)\frac{1}{H(r)}\frac{1}{r^{N-2s}}\int_{\partial^+B_r^+}t^{1-2s}(A\nabla W\cdot\nu) W\,dS
=\frac{H'(r)}{H(r)}O(r)+O(r)=\N(r)O(1)+O(r),
\end{equation}
where we used \eqref{eq1} and \eqref{E=rH'/2}. Combining
\eqref{eqN'1} and \eqref{eqN'2} we obtain that
\begin{align*}
\N'(r)&\geq
\frac{2r\left[\left(\int_{\partial^+B_r^+}\!t^{1-2s}\frac{|A\nabla W\cdot\nu|^2}{\mu}dS\right)\!\left(\int_{\partial^+B_r^+}\!t^{1-2s}\mu W^2dS\right)\!-\!\left(\int_{\partial^+B_r^+}\!t^{1-2s}(A\nabla W\cdot\nu) WdS\right)^{\!2}\right]}{\left(\int_{\partial^+B_r^+}t^{1-2s}\mu W^2dS\right)^{\!2}}\\
&\qquad+\N(r)O(1)+O(r)+O(r^{-1+\bar\delta})\left(\N(r)+\frac{N-2s}{2}\right)\quad\text{as $r\rightarrow 0^+$}, 
\end{align*}
which yields \eqref{N'>=nu1+nu2} in view of \eqref{liminf}.
\end{proof}

\begin{Proposition}\label{Lemmalimitexists}
Let $\N$ be the function defined  in \eqref{N}. Then there exists
$C_1>0$ such that, for every $r\in (0,R_0]$,
\begin{equation}\label{Nlimitata}
\N(r)\leq C_1.
\end{equation}
Moreover the limit
\begin{equation}\label{limitexists}
\gamma:=\lim_{r\to0^+}\N(r)
\end{equation}
exists, is  finite and nonnegative.
\end{Proposition}
\begin{proof}
From Lemma \ref{deriveN}, we deduce that
$\N'(r)\geq V_2(r)$ a.e. in $(0,R_0)$,
since $V_1(r)\geq 0$ as a consequence of Schwarz's inequality. 
Hence there exist
$0<\hat R<R_0$ and $C_2>0$ such that
\begin{equation}\label{C1}
\N'(r)\geq -C_2r^{-1+\bar\delta}\left(\N(r)+\frac{N-2s}{2}\right),
\end{equation}
for a.e. $r\in (0,\hat R)$. 
Then 
\begin{equation*}
\frac{d}{dr}\left[\log\left(\N(r)+\frac{N-2s}{2}\right)\right]\geq-C_2r^{-1+\bar\delta}
\quad\text{a.e. in }(0,\hat R),
\end{equation*}
and, integrating the above inequality between $(r,\hat R)$ with $r<
\hat R$, we obtain the upper bound
\begin{equation*}
\N(r)\leq\left(\N(\hat R)+\frac{N-2s}{2}\right)e^{C_2\frac{\hat
    R^{\bar\delta}}{\bar \delta}}-\frac{N-2s}{2}\quad\text{for all
}r\in(0,\hat R), 
\end{equation*}
which yields  \eqref{Nlimitata}, in view of the continuity of
$\mathcal N$ in $(0,R_0]$. From \eqref{C1}, we derive that
\begin{equation*}
\frac{d}{dr}\left[e^{C_2\frac{r^{\bar\delta}}{\bar\delta}}\left(\N(r)+\frac{N-2s}{2}\right)\right]\geq
0\quad\text{a.e. in }(0,\hat R),
\end{equation*}
hence
\begin{equation*}
r\mapsto e^{C_2\frac{r^{\bar\delta}}{\bar\delta}}\left(\N(r)+\frac{N-2s}{2}\right)
\end{equation*}
is a monotonically increasing function in $(0,\hat R)$, thus its limit as $r\to 0^+$ does exist, and the same holds true for the limit of the function $\N$. From Lemma \ref{lemmaliminf} and \eqref{Nlimitata}, we can conclude that the limit $\gamma:=\lim_{r\to0^+}\N(r)$ is finite and nonnegative.
\end{proof}
\begin{Lemma}\label{l:k1k2}
Let $\gamma=\lim_{r\to0^+}\N(r)$. Then:
\begin{itemize}
\item[\rm (i)] there exists $k_1>0$ such that, for all $r\in(0, R_0)$,
\begin{equation}\label{Hupper}
H(r)\leq k_1r^{2\gamma};
\end{equation}
\item[\rm (ii)] for any $\sigma>0$, there exists $k_2(\sigma)>0$ such that, for all $r\in(0, R_0)$,
$$H(r)\geq k_2(\sigma)r^{2\gamma+\sigma}.$$
\end{itemize}
\begin{proof}
(i) By \eqref{Nlimitata}, \eqref{C1}, and \eqref{limitexists} $\N'\in L^1(0,R_0)$, then using \eqref{C1} and \eqref{Nlimitata}
\begin{align*}
\N(r)-\gamma&=\int_0^r\N'(s)ds\geq-C_2\int_0^rs^{-1+\bar\delta}\left(\N(s)+\frac{N-2s}{2}\right)ds\\
&\geq-C_2\left(C_1+\tfrac{N-2s}{2}\right)\int_0^rs^{-1+\bar
  \delta}ds=
-C_2\left(C_1+\tfrac{N-2s}{2}\right)
\frac{r^{\bar\delta}}{\bar\delta}\quad\text{for all }r\in(0,R_0).
\end{align*}
Using \eqref{E=rH'/2}, one has
\begin{equation*}
\frac{H'(r)}{H(r)}=\frac{2\N(r)}{r}+O(1)\geq \frac{2\gamma}{r}
-2C_2\left(C_1+\tfrac{N-2s}{2}\right)
\frac{r^{-1+\bar\delta}}{\bar\delta}+O(1)\quad\text{as }r\to0^+.
\end{equation*}
Integrating the above estimate and taking into
  account that $H$ is continuous on $(0,R_0]$, we obtain
  \eqref{Hupper}.

\medskip  (ii) Since $\gamma= \lim_{r\to0^+}\N(r)$, for any $\sigma>0$ there exists $r_\sigma>0$ such that, for any $r\in(0,r_\sigma)$,
$$\N(r)<\gamma+\frac{\sigma}{2},$$ 
and hence
$$\frac{H'(r)}{H(r)}=\frac{2\N(r)}{r}+O(1)<\frac{2\gamma+\sigma}{r}+c,$$
for some positive constant $c$.
Integrating over the interval $(r,r_\sigma)$ and taking into account
that $H$ is continuous and positive in $(0,R_0]$, we prove the
second statement.
\end{proof}
\end{Lemma}

\section{Blow-up analysis and local asymptotics}\label{sec:blow-up-analysis}

\subsection{Blow-up analysis}
As in Section \ref{sec:almgr-type-freq}, let
$W\in H^1_{\Gamma_{R_1}^+}(B_{R_1}^+,t^{1-2s}\,dz)$ be a nontrivial
weak solution of \eqref{prob4}.  For every $\lambda\in (0,R_0)$,
with $R_0$ being as in \eqref{eq:R_0}, let us define
\begin{equation}\label{wlam}
  w^\lambda(z)=\frac{W(\lambda z)}{\sqrt{H(\lambda)}}.
\end{equation}
 We have that $w^\lambda$ is a weak solution to
\begin{equation}\label{prob3lambda}
\begin{cases}
  -\mathrm{div}\left(t^{1-2s}A(\lambda\,\cdot)\nabla w^\lambda\right)=0 &\mathrm{in \ } B_{{R_1}/{\lambda}}^+,\\
  \lim_{t\to0^+}\left( t^{1-2s}A(\lambda\,\cdot)\nabla
    w^\lambda\cdot\nu\right)=\kappa_s\lambda^{2s} \tilde h(\lambda
  \cdot)
  \mathop{\rm Tr}w^\lambda &\mathrm{on \ } \Gamma_{R_1/\lambda}^-,\\
  w^\lambda=0 & \mathrm{on \ } \Gamma_{R_1/\lambda}^+.
\end{cases}
\end{equation}
Moreover we have that
\begin{equation}\label{intwlambda1}
\int_{{\mathbb S}^{N}_+}\theta_{N+1}^{1-2s}\mu(\lambda\theta) |w^\lambda(\theta)|^2\, dS=1.
\end{equation}
\begin{Lemma}\label{Lemmawlamlimit}
The family of functions $\{w^\lambda\}_{\lambda\in (0,R_0)}$ is bounded in $H^1(B_1^+,t^{1-2s}dz)$. 
\end{Lemma}
\begin{proof}By \eqref{N} and using \eqref{eqcoerc1}, \eqref{<1/8} and
  \eqref{remark}, we obtain that, for every $\lambda\in (0,R_0)$, 
\begin{equation*}
  \begin{split}
    \N(\lambda)&=\frac{\lambda^{2s-N}}{H(\lambda)}
    \left(\int_{B^+_\lambda}t^{1-2s}A\nabla W\cdot\nabla W\
      dz-\kappa_s\int_{\Gamma^-_\lambda}
      \tilde h |\mathop{\rm Tr} W|^2\ dy\right)\\
    &\geq \frac{\lambda^{2s-N}}{H(\lambda)}\biggl[\frac{3}{8}
    \int_{B^+_\lambda}t^{1-2s}|\nabla W|^2\ dz-\kappa_s\tilde{S}_{N,s}
    \tilde c_{N,s,p}
    \lambda^{\overline{\varepsilon}}\bar{\alpha}_0                                        \frac{2(N-2s)}{\lambda}
    \int_{\partial^+B^+_\lambda}t^{1-2s}\mu W^2 dS\biggr]\\
    &\geq \frac{3}{8}\int_{B^+_1}t^{1-2s}|\nabla w^\lambda|^2\ dz
    -2(N-2s)\kappa_s\tilde{S}_{N,s} \tilde c_{N,s,p} \lambda^{\overline{\varepsilon}}\bar{\alpha}_0\\
    &\geq \frac{3}{8}\int_{B^+_1}t^{1-2s}|\nabla w^\lambda|^2\
    dz-\frac{N-2s}{4},
  \end{split}
\end{equation*}
which together with \eqref{Nlimitata} implies that
$\left\{\|\nabla
  w^\lambda\|_{L^2(B_1^+,t^{1-2s}dz)}\right\}_{\lambda\in (0,R_0)}$
is bounded. From this and \eqref{intwlambda1}, the boundedness of
$\{w^\lambda\}_{\lambda\in (0,R_0)}$ in $H^1(B_1^+,t^{1-2s}dz)$
follows by Lemma \ref{l:hardy_boundary}.
\end{proof}
We are going to prove strong convergence in $H^1(B_1^+,t^{1-2s}dz)$ of
  $\{w^\lambda\}$ along a proper vanishing sequence of $\lambda$'s; to
  this aim, we first need to establish the following doubling properties.
\begin{Lemma}\label{lemdoub}
There exists $C_3>0$ such that
\begin{equation}\label{doubling}
  \frac{1}{C_3}H(\lambda)\leq H(R\lambda)\leq C_3H(\lambda),
\end{equation}
\begin{equation}\label{stima1}
  \int_{B_R^+}t^{1-2s}|\nabla w^\lambda|^2dz\leq C_32^{N-2s}\int_{B_1^+}t^{1-2s}|\nabla w^{R\lambda}|^2dz,
\end{equation}
and
\begin{equation}\label{stima2}
  \int_{B_R^+}t^{1-2s}|w^\lambda|^2dz\leq C_32^{N+2-2s}\int_{B_1^+}t^{1-2s}|w^{R\lambda}|^2 dz.
\end{equation}
for any $\lambda<R_0/2$ and $R\in[1,2]$.
\end{Lemma}
\begin{proof}
From \eqref{E=rH'/2} we deduce that, for a.e. $r\in (0,R_0)$,
\begin{equation*}
  \frac{H'(r)}{H(r)}=\frac{2\N(r)}{r}+O(1)\quad\text{as }r\to0^+.
\end{equation*}
Hence there exist a positive constant $C>0$ and $\tilde R_0<R_0$ such
that, for all $r\in(0,\tilde R_0)$,
\begin{equation*}
  -C-\frac{N-2s}{r}\leq \frac{H'(r)}{H(r)}\leq C+\frac{2C_1}{r},
\end{equation*}
where we used \eqref{lower} and \eqref{Nlimitata}. 
Integrating the above inequalities over the interval
$(\lambda,R\lambda)$, with $R\in(1,2]$ and $\lambda<\tilde R_0/R$, we
obtain that 
\begin{equation}\label{eq:29}
  2^{-(N-2s)}e^{-C\frac{\tilde R_0}{R}(R-1)}\leq\frac{H(R\lambda)}{H(\lambda)}\leq 4^{C_1}e^{C\frac{\tilde R_0}{R}(R-1)}.
\end{equation}
The above chain of inequalities trivially extends to the case $R=1$.
Estimate \eqref{doubling} follows from \eqref{eq:29} and the
  fact that $H$ is continuous and strictly positive in $(0,R_0]$
  (Lemmas \ref{lemH} and \ref{lemmaH>0}).
By scaling and \eqref{doubling}, we easily deduce \eqref{stima1} and
\eqref{stima2} (see \cite{DelFel} for details in a
  similar situation).
\end{proof}
\begin{Lemma}\label{Rlam}
  Let $w^\lambda$ be as in \eqref{wlam}, with
  $\lambda\in(0,R_0)$. Then there exist $M>0$ and $\lambda_0>0$ such
  that, for any $\lambda\in(0,\lambda_0)$, there exists
  $R_\lambda \in[1,2]$ such that
\begin{equation*}
\int_{\partial^+ B_{R_\lambda}^+}t^{1-2s}|\nabla w^\lambda|^2\,dS\leq M\int_{B_{R_\lambda}^+}t^{1-2s}|\nabla w^\lambda|^2\,dz.
\end{equation*}
\end{Lemma}
\begin{proof}
We recall that, by Lemma \ref{Lemmawlamlimit}, the set
$\{w^\lambda\}_{\lambda\in(0,R_0)}$ is bounded in
$H^1(B_1^+,t^{1-2s}dz)$ and 
\begin{equation}\label{eq:31}
 w^\lambda\in H^1_{\Gamma_{1}^+}(B_{1}^+,t^{1-2s}\,dz)\quad\text{for
    all }\lambda\in(0,R_0). 
\end{equation}
 Moreover, by Lemma \ref{lemdoub}, we have that $\{w^\lambda\}_{\lambda\in(0,R_0/2)}$ is bounded in $H^1(B_2^+,t^{1-2s}dz)$, hence
\begin{equation}\label{limsup}
  \limsup_{\lambda\to0^+}\int_{B_2^+}t^{1-2s}|\nabla w^\lambda|^2\,dz<+\infty.
\end{equation}
For every $\lambda\in(0,R_0/2)$, let 
\begin{equation*}
  f_\lambda(r)=\int_{B_r^+}t^{1-2s}|\nabla w^\lambda|^2\,dz.
\end{equation*}
Then $f_\lambda$ is absolutely continuous in $[0,2]$ with
distributional derivative given by 
\begin{equation*}
  f_\lambda'(r)=\int_{\partial^+B_r^+}t^{1-2s}|\nabla w^\lambda|^2\,dS\quad\text{for almost every $r\in(0,2)$}.
\end{equation*}
Let us suppose by contradiction that for any $M>0$ there exists a sequence $\lambda_n\to0^+$ such that
\begin{equation*}
  \int_{\partial^+B_r^+}t^{1-2s}|\nabla w^{\lambda_n}|^2\,dS>M\int_{B_r^+}t^{1-2s}|\nabla w^{\lambda_n}|^2\,dz
\end{equation*}
for all $r\in[1,2]$ and $n\in\mathbb{N}$, i.e.
\begin{equation}\label{f'Mf}
f'_{\lambda_n}(r)>Mf_{\lambda_n}(r)
\end{equation}
for a.e. $r\in(1,2)$ and any $n\in\mathbb N$. Integrating \eqref{f'Mf}
over $[1,2]$, we obtain that, for any $n\in\mathbb N$,
$f_{\lambda_n}(2)>e^Mf_{\lambda_n}(1)$,
and hence
\begin{equation*}
  \liminf_{n\to+\infty}f_{\lambda_n}(1)\leq
  \limsup_{n\to+\infty}f_{\lambda_n}(1)
  \leq e^{-M}\limsup_{n\to+\infty}f_{\lambda_n}(2),
\end{equation*}
  which implies that 
\begin{equation}\label{eq:30}
 \liminf_{\lambda\to 0^+}f_{\lambda}(1)\leq e^{-M}\limsup_{\lambda\to 0^+}f_{\lambda}(2),
\end{equation}
for all $M>0$. From \eqref{eq:30} and \eqref{limsup},
  letting $M\to+\infty$ we deduce that $\liminf_{\lambda\to
    0^+}f_{\lambda}(1)=0$.
 Then there exist a sequence $\tilde
   \lambda_n\to0^+$ and some $w\in H^1(B_1^+,t^{1-2s}dz)$  such that $w^{\tilde \lambda_n}\rightharpoonup w$ in $H^1(B_1^+,t^{1-2s}dz)$ with
\begin{equation*}
  \lim_{n\to+\infty}\int_{B_1^+}t^{1-2s}|\nabla
  w^{\tilde\lambda_n}|^2dz=0.
\end{equation*}
However, by compactness of trace map
$H^1(B_1^+,t^{1-2s}dz)\hookrightarrow \hookrightarrow
  L^2(\partial^+B_1^+,t^{1-2s}dS)$, \eqref{intwlambda1}, \eqref{sviluppomu},
and weak lower semicontinuity of norms, we necessarily have that
\begin{equation*}
  \int_{B_1^+}t^{1-2s}|\nabla
  w|^2dz=0\qquad\mathrm{and}\qquad\int_{\partial^+B_1^+}t^{1-2s}w^2dS=1.
\end{equation*}
 Hence there exists $c\in\R$ such that $w\equiv c$ in
  $B_{1}^+$ and $c\neq0$. Since
  $H^1_{\Gamma_{1}^+}(B_{1}^+,t^{1-2s}\,dz)$ is weakly closed in
  $H^1(B_{1}^+,t^{1-2s}\,dz)$, from \eqref{eq:31} we deduce  that
  $w\equiv c\in H^1_{\Gamma_{1}^+}(B_{1}^+,t^{1-2s}\,dz)$, so that
  $0=\mathop{\rm Tr}w\big|_{\Gamma_{1}^+}=c$, a contradiction.
\end{proof}
\begin{Lemma}\label{lemmgradlimit}
Let $w^\lambda$ and $R_\lambda$ be as in the statement of Lemma
\ref{Rlam}.
 Then there exists $\overline M>0$ such that
\begin{equation*}
\int_{\partial^+B_1^+}t^{1-2s}|\nabla w^{\lambda R_\lambda}|^2\,dz\leq \overline M
\end{equation*}
for any $\lambda\in(0,\min\{\lambda_0,R_0/2\})$.
\end{Lemma}
\begin{proof}
We observe that, by scaling  and \eqref{wlam},
\begin{equation*}
  \int_{\partial^+ B_1^+} t^{1-2s}|\nabla w^{\lambda R_\lambda}|^2\,dS=\frac{R_\lambda^{1-N+2s}H(\lambda)}{H(\lambda R_\lambda)}\int_{\partial^+ B_{R_\lambda}^+}t^{1-2s}|\nabla w^\lambda|^2\,dS,
\end{equation*}
so that, in view of  Lemmas \ref{lemdoub}, \ref{Rlam}, and
\ref{Lemmawlamlimit}, we have that 
\begin{align*}
   \int_{\partial^+ B_1^+} t^{1-2s}|\nabla w^{\lambda
     R_\lambda}|^2\,dS&\leq 2C_3M\int_{B_{R_\lambda}^+}t^{1-2s}|\nabla
   w^{\lambda}|^2\,dz\\
&\leq 2^{1+N-2s}M C_3^2\int_{B_1^+}t^{1-2s}|\nabla
   w^{\lambda R_\lambda}|^2\,dz\leq \overline{M}<+\infty.
\end{align*}
The proof is thereby complete.
\end{proof}

\begin{Proposition}\label{propH}
  Let $W\in H^1_{\Gamma_{R_1}^+}(B_{R_1}^+,t^{1-2s}\,dz)$,
  $W\not\equiv0$, be a nontrivial weak solution of \eqref{prob4}. Let
  $\gamma$ be as in Proposition \ref{Lemmalimitexists}. Then
\begin{itemize}
\item[\rm(i)] there exists $k_0\in\mathbb N$ such that
  $\gamma=s+k_0$;
\item[\rm (ii)] for any sequence $\lambda_n\to0^+$, there exist a
  subsequence $\{\lambda_{n_k}\}$ and an eigenfunction $\psi$
  of problem \eqref{eig} associated to the eigenvalue
  $\mu_{k_0}=(k_0+s)(k_0+N-s)$ such that
  $\|\psi\|_{L^2({\mathbb S}^{N}_+,\theta_{N+1}^{1-2s}dS)}=1$ and
\begin{equation*}
  w^{\lambda_{n_k}}(z)=\frac{W(\lambda_{n_k}
    z)}{\sqrt{H(\lambda_{n_k})}}\to 
  |z|^\gamma\psi\left(\frac{z}{|z|}\right)
\end{equation*}
strongly in $H^1(B_1^+,t^{1-2s}dz)$.
\end{itemize}
\end{Proposition}

\begin{proof} Let $w^\lambda\in H^1_{\Gamma_{1}^+}(B_{1}^+,t^{1-2s}\,dz)$ be as in \eqref{wlam} and $R_\lambda$ as
  in Lemma \ref{Rlam}. From Lemma \ref{Lemmawlamlimit} we deduce that the set
  $\{w^{\lambda R_\lambda}\}_{\lambda\in (0,\min\{\lambda_0,R_0/2\})}$ is bounded in
  $H^1(B_1^+,t^{1-2s}dz)$. Let us consider a sequence 
  $\lambda_n\to0^+$. Then there exist a subsequence $\{\lambda_{n_k}\}$ and
  $w\in H^1_{\Gamma_{1}^+}(B_{1}^+,t^{1-2s}\,dz)$ such that
  $w^{\lambda_{n_k} R_{\lambda_{n_k}}}\rightharpoonup w$ weakly in
  $H^1(B_1^+,t^{1-2s}dz)$. Moreover we have that
\begin{equation}\label{intw=1}
\int_{\partial^+B_1^+}t^{1-2s}w^2dS=1
\end{equation}
by compactness of trace map
$H^1(B_1^+,t^{1-2s}dz)\hookrightarrow \hookrightarrow
  L^2(\partial^+B_1^+,t^{1-2s}dS)$, \eqref{intwlambda1}, and
  \eqref{sviluppomu}. This allows us to conclude that $w$ is non-trivial. 

We now claim strong convergence
\begin{equation}\label{strongconv}
w^{\lambda_{n_k}R_{\lambda_{n_k}}}\rightarrow w \quad\text{in $H^1(B_1^+,t^{1-2s}dz)$}.
\end{equation}
We note that $w^{\lambda_{n_k}R_{\lambda_{n_k}}}$ weakly solves
\eqref{prob3lambda} with
$\lambda=\lambda_{n_k}R_{\lambda_{n_k}}$. Since  $B^+_1\subset
B^+_{R_1/(\lambda_{n_k}R_{\lambda_{n_k}})}$ for sufficiently large
$k$, we then have that
\begin{align}\label{eqnk}
  \int_{B^+_1}t^{1-2s}&A(\lambda_{n_k}R_{\lambda_{n_k}}y)\nabla
  w^{\lambda_{n_k}R_{\lambda_{n_k}}}(z)\cdot\nabla \Phi (z)\,dz\\
  \notag&=\kappa_s(\lambda_{n_k}R_{\lambda_{n_k}})^{2s}\int_{\Gamma^-_1}\tilde
     h(\lambda_{n_k}R_{\lambda_{n_k}}y)
     \mathop{\rm Tr}w^{\lambda_{n_k}R_{\lambda_{n_k}}}(y) \mathop{\rm Tr}\Phi(y)\,dy\\
 \notag  &\qquad+
     \int_{\partial^+B_1^+}\bigl(t^{1-2s}A(\lambda_{n_k}R_{\lambda_{n_k}}y)\nabla
     w^{\lambda_{n_k}R_{\lambda_{n_k}}}(z)\cdot\nu\bigr)\Phi(z)\,dS
\end{align}
for sufficiently large $k$ and for every $\Phi\in 
C^\infty_{c}(\overline{B_{1}^+}\setminus \Gamma_{1}^+)$,
hence by density for every $\Phi\in
H^1_{\Gamma^+_1}(B^+_1,t^{1-2s}dz)$. We are going to pass to the limit
in \eqref{eqnk}. 
To this aim, we observe that \eqref{sviluppodellaA} implies that 
\begin{align}\label{primotermvaa0}
&\bigg|\int_{B^+_1}t^{1-2s}(A(\lambda y)\nabla w^{\lambda}(z)-\nabla
                            w(z))\cdot\nabla \Phi(z)\,dx\bigg|\\
&\notag\leq \bigg|\int_{B^+_1}t^{1-2s} \nabla (w^{\lambda}- w)\cdot\nabla \Phi\,dz\bigg|+C \lambda\int_{B^+_1}t^{1-2s}|\nabla w^\lambda||\nabla \Phi|\,dz\\
\notag&\leq \bigg|\int_{B^+_1}t^{1-2s} \nabla (w^\lambda- w)\cdot\nabla \Phi\,dz\bigg|+C\lambda\left(\int_{B^+_1}t^{1-2s}|\nabla w^\lambda|^2\,dz\right)^{\!1/2}\!\!\!\left(\int_{B^+_1}t^{1-2s}|\nabla \Phi|^2\,dz\right)^{\!1/2}
\end{align}
for some $C>0$ and for sufficiently small $\lambda$, and 
\begin{align}\label{ho(1)}
\lambda^{2s}&\bigg|\int_{\Gamma_1^-}\tilde h(\lambda y) \mathop{\rm Tr}w^\lambda(y) \mathop{\rm Tr}\Phi(y)\,dy\bigg|\\
\notag&\leq
  \lambda^{2s}\left(\int_{B_1'}| \mathop{\rm Tr}w^\lambda(y)|^{2^*(s)}\,dy\right)^{\!\!\frac1{2^*(s)}}
\left(\int_{B_1'}| \mathop{\rm Tr}\Phi(y)|^{2^*(s)}\,dy\right)^{\!\!\frac1{2^*(s)}}
\left(\int_{\Gamma_1^-}|\tilde h(\lambda y)|^{\frac{N}{2s}}\,dy\right)^{\!\!\frac{2s}N}\\
\notag&=O(1)\left(\int_{\Gamma_\lambda^-}|\tilde h(y)|^{\frac{N}{2s}}dy\right)^{\!\!\frac{2s}N}=o(1)\quad\text{as $\lambda\rightarrow 0^+$},
\end{align}
from H\"older's inequality, Lemmas \ref{lemma2.6FF},
\ref{Lemmawlamlimit}, 
and \eqref{intwlambda1}, since
$\mu(\lambda y)\geq1/4$ for all $\lambda\leq R_0$.
Taking $\lambda=\lambda_{n_k}R_{\lambda_{n_k}}$ in
\eqref{primotermvaa0} and \eqref{ho(1)}, letting $k\to\infty$, and
recalling that   $w^{\lambda_{n_k} R_{\lambda_{n_k}}}\rightharpoonup w$ weakly in
  $H^1(B_1^+,t^{1-2s}dz)$, we
obtain that 
\begin{equation}\label{eq:32}
  \lim_{k\to\infty}  \int_{B^+_1}t^{1-2s}A(\lambda_{n_k}R_{\lambda_{n_k}}y)\nabla
  w^{\lambda_{n_k}R_{\lambda_{n_k}}}(z)\cdot\nabla \Phi (z)\,dz=
\int_{B^+_1}t^{1-2s}\nabla
  w\cdot\nabla \Phi \,dz
\end{equation}
and 
\begin{equation}\label{eq:33}
  \lim_{k\to\infty}
(\lambda_{n_k}R_{\lambda_{n_k}})^{2s}\int_{\Gamma^-_1}\tilde
     h(\lambda_{n_k}R_{\lambda_{n_k}}y)
     \mathop{\rm Tr}w^{\lambda_{n_k}R_{\lambda_{n_k}}}(y) \mathop{\rm
       Tr}\Phi(y)\,dy=0.
   \end{equation}
Thanks to \eqref{sviluppodellaA}, we have that 
\begin{align}\label{lastterm}
&\int_{\partial^+B_1^+}\bigl(t^{1-2s}A(\lambda y)\nabla
                                      w^{\lambda}(z)\cdot\nu\bigr)\Phi(z)\,dS\\
\notag&=\int_{\partial^+B_1^+}t^{1-2s}\frac{\partial w^{\lambda}}{\partial\nu}\Phi\,dS+\int_{\partial^+B_1^+}
t^{1-2s}(A(\lambda y)-\mathrm{Id}_N)\nabla
        w^{\lambda}(z)\cdot\nu\,\Phi(z)\,dS\\
\notag&=\int_{\partial^+B_1^+}t^{1-2s}\frac{\partial
        w^{\lambda}}{\partial\nu}\Phi\,dS+O(\lambda)
\left(\int_{\partial^+B_1^+ }t^{1-2s}|\nabla w^\lambda|^2\,dS\right)^{\!1/2}\!\!\!\left(\int_{\partial^+B_1^+}t^{1-2s}\Phi^2\,dS\right)^{\!1/2}.
\end{align} 
Moreover, from Lemma \ref{lemmgradlimit}, up to a further subsequence,
we have that
\begin{equation}\label{dernormconv}
\frac{\partial
  w^{\lambda_{n_k}R_{\lambda_{n_k}}}}{\partial\nu}\rightharpoonup f
	\quad\text{weakly in $L^2(\partial^+B_1^+,t^{1-2s}dS)$}
\end{equation}
for some $f\in L^2(\partial^+B_1^+,t^{1-2s}dS)$. 
Then, taking $\lambda=\lambda_{n_k}R_{\lambda_{n_k}}$ in
\eqref{lastterm}, letting $k\to\infty$, and
taking into account Lemma \ref{lemmgradlimit}, we obtain that 
\begin{equation}\label{eq:34}
  \lim_{k\to\infty}
     \int_{\partial^+B_1^+}\bigl(t^{1-2s}A(\lambda_{n_k}R_{\lambda_{n_k}}y)\nabla
     w^{\lambda_{n_k}R_{\lambda_{n_k}}}(z)\cdot\nu\bigr)\Phi(z)\,dS=
\int_{\partial^+B_1^+}t^{1-2s}f\Phi\,dS.
   \end{equation}
Hence, passing to the
limit as $k\to\infty$ in \eqref{eqnk} and  combining \eqref{eq:32}, \eqref{eq:33}, and \eqref{eq:34}, we
  find that
\begin{equation}\label{eqlimite}
\int_{B_1^+}t^{1-2s}\nabla w\cdot \nabla \Phi\,dz=
\int_{\partial^+B_1^+}t^{1-2s}f\Phi\,dS
\quad \text{for any $\Phi\in
H^1_{\Gamma^+_1}(B^+_1,t^{1-2s}dz)$}.
\end{equation}
 On the other hand, if we take $\Phi=w^{\lambda_{n_k}R_{\lambda_{n_k}}}$ in \eqref{eqnk}, we have that
\begin{align*}
  \int_{B^+_1}t^{1-2s}&A(\lambda_{n_k}R_{\lambda_{n_k}}y)\nabla
  w^{\lambda_{n_k}R_{\lambda_{n_k}}}(z)\cdot
\nabla
  w^{\lambda_{n_k}R_{\lambda_{n_k}}}(z)\,dz\\
  &=\kappa_s(\lambda_{n_k}R_{\lambda_{n_k}})^{2s}\int_{\Gamma^-_1}\tilde
     h(\lambda_{n_k}R_{\lambda_{n_k}}y)
    | \mathop{\rm Tr}w^{\lambda_{n_k}R_{\lambda_{n_k}}}(y)|^2\,dy\\
  &\qquad+
     \int_{\partial^+B_1^+}\bigl(t^{1-2s}A(\lambda_{n_k}R_{\lambda_{n_k}}z)\nabla
     w^{\lambda_{n_k}R_{\lambda_{n_k}}}(z)\cdot\nu\bigr)  w^{\lambda_{n_k}R_{\lambda_{n_k}}} (z)\,dS,
\end{align*}
hence, by \eqref{sviluppodellaA}, arguing as in \eqref{lastterm} and
\eqref{ho(1)} and using Lemma \ref{lemmgradlimit} and
\eqref{dernormconv}, we obtain that 
\begin{align}\label{eq:35}
  \lim_{k\to\infty}\int_{B^+_1} |\nabla
  w^{\lambda_{n_k}R_{\lambda_{n_k}}}|^2&=\lim_{k\to\infty} 
                                         \int_{B^+_1}t^{1-2s}A(\lambda_{n_k}R_{\lambda_{n_k}}y)\nabla
                                         w^{\lambda_{n_k}R_{\lambda_{n_k}}}(z)\cdot
                                         \nabla
                                         w^{\lambda_{n_k}R_{\lambda_{n_k}}}(z)\,dz
  \\
  \notag&=\lim_{k\to\infty}\int_{\partial ^+B^+_1}t^{1-2s}\frac{\partial
          w^{\lambda_{n_k}R_{\lambda_{n_k}}}}{\partial\nu}w^{\lambda_{n_k}R_{\lambda_{n_k}}}dS\\
  \notag&=\int_{\partial^+B^+_1}t^{1-2s} f w\,dS=\int_{B^+_1}t^{1-2s}|\nabla w|^2,
\end{align}
where we used the compactness of the trace operator
$H^1(B_1^+,t^{1-2s}dz)\hookrightarrow \hookrightarrow
  L^2(\partial^+B_1^+,t^{1-2s}dS)$
and \eqref{eqlimite} with $\Phi=w$.
The weak convergence   $w^{\lambda_{n_k} R_{\lambda_{n_k}}}\rightharpoonup w$ in
  $H^1(B_1^+,t^{1-2s}dz)$ together with \eqref{eq:35}
  imply~\eqref{strongconv}.

For every $k\in\mathbb{N}$ and $r\in (0,1]$, let us define 
\begin{multline*}
  E_k(r)=\frac{1}{r^{N-2s}}\biggl[\int_{B^+_r}t^{1-2s}A(\lambda_{n_k}R_{\lambda_{n_k}}y)\nabla
  w^{\lambda_{n_k}R_{\lambda_{n_k}}}\cdot\nabla
  w^{\lambda_{n_k}R_{\lambda_{n_k}}} dz\\
  -\kappa_s \lambda^{2s}_{n_k}R^{2s}_{\lambda_{n_k}}\int_{\Gamma^-_{
      r}}\tilde h(\lambda_{n_k}R_{\lambda_{n_k}}y)|\mathop{\rm
    Tr}w^{\lambda_{n_k}R_{\lambda_{n_k}}}|^2 dy\biggr]
\end{multline*}
and 
\begin{equation*}
  H_k(r)=\frac{1}{r^{N+1-2s}}\int_{\partial^+
    B^+_r}t^{1-2s}\mu(\lambda_{n_k}R_{\lambda_{n_k}}z)
|w^{\lambda_{n_k}R_{\lambda_{n_k}}}(z)|^2 dS.
\end{equation*}
We also define, for any $r\in (0,1]$,
\begin{equation}\label{Ew}
E_w(r)=\frac{1}{r^{N-2s}}\int_{B^+_r}t^{1-2s}|\nabla w(z)|^2dz 
\end{equation}
and 
\begin{equation}\label{Hw}
H_w(r)=\frac{1}{r^{N+1-2s}}\int_{\partial^+ B^+_r} t^{1-2s}w^2(z)\, dS.
\end{equation}
By scaling, one can  easily verify that 
\begin{equation}\label{Nk}
  \N_k(r):=\frac{E_k(r)}{H_k(r)}=\frac{E(\lambda_{n_k}R_{\lambda_{n_k}}r)}
  {H(\lambda_{n_k}R_{\lambda_{n_k}}r)}=\N(\lambda_{n_k}R_{\lambda_{n_k}}r)\quad\text{for all $r\in (0,1]$}.
\end{equation}
From \eqref{strongconv}, \eqref{sviluppodellaA}, and
  \eqref{ho(1)}, it follows that, for any fixed $r\in (0,1]$,
\begin{equation}\label{Ek->Ew}
E_k(r)\rightarrow E_w(r).
\end{equation}
On the other hand, by compactness of the trace operator and \eqref{sviluppomu}, we also have, for any fixed $r\in (0,1]$,
\begin{equation}\label{Hk->Hw}
H_k(r)\rightarrow H_w(r).
\end{equation}
In order to prove that $H_w$ is strictly positive, we
  argue by contradiction and assume that there exists  $r\in (0,1]$
  such that $H_w(r)=0$; then $r$ is a minimum point for $H_w$ and
  hence, arguing as in Lemma \ref{lemH}, we obtain that necessarily
  $0=H_w'(r)=2r^{2s-N-1}\int_{B^+_r}t^{1-2s}|\nabla w(z)|^2dz
$ and hence $w$ is constant in $B_r^+$. 
From Lemma \ref{l:hardy_boundary} we conclude
that $w\equiv 0$ in $B_r^+$, which implies that 
$w\equiv 0$ in $B_1^+$ from
classical unique continuation principles for second order elliptic
equations, thus contradicting \eqref{intw=1}.

 Hence $H_w(r)>0$ for all $r\in (0,1]$, thus the function 
\begin{equation*}
\mathcal N_w:(0,1]\to\R,\quad \mathcal N_w(r):=\frac{E_w(r)}{H_w(r)}
\end{equation*}
is well defined and one can easily prove that it belongs to
$W^{1,1}_{\mathrm{loc}}((0,1])$. From \eqref{Nk}, \eqref{Ek->Ew},
\eqref{Hk->Hw}, and Proposition \ref{Lemmalimitexists}, we deduce that
\begin{equation}\label{Nw=gamma}
\N_w(r)=\lim_{k\rightarrow \infty} \N(\lambda_{n_k}R_{\lambda_{n_k}}r)=\gamma
\end{equation}
for all $r\in (0,1]$. Therefore $\N_w$ is constant in $(0,1]$ and
hence 
\begin{equation}\label{eq:36}
\N'_w(r)=0\quad\text{for any $r\in
(0,1)$}.
\end{equation}
Recalling the equation satisfied by $w$, i.e. \eqref{eqlimite},
and arguing as in 
Lemma \ref{deriveN}
 with $A=\mathrm{Id}_N$ and
 $\tilde h\equiv 0$,  we can prove 
 that, for a.e. $r\in(0,1)$, 
\begin{equation}\label{eq:37}
  \N_w'(r)\geq
  \frac{2r\left[\left(\int_{\partial^+B_r^+}\!t^{1-2s}
        \left|\partial_\nu w\right|^2dS\right)\!\left(\int_{\partial^+B_r^+}\!t^{1-2s}w^2dS\right)\!
      -\!\left(\int_{\partial^+B_r^+}\!t^{1-2s}\partial_\nu w\,
        w\,dS\right)^{\!2}\right]}{\left(\int_{\partial^+B_r^+}t^{1-2s}w^2dS\right)^{\!2}}.
\end{equation}
Combining \eqref{eq:36} and \eqref{eq:37} with  Schwarz's inequality,
we obtain that, for a.e. $r\in(0,1)$, 
\begin{equation*}
  \left(\int_{\partial^+B_r^+}\!t^{1-2s}\left|\partial_\nu
      w\right|^2dS\right)\!
  \left(\int_{\partial^+B_r^+}\!t^{1-2s}w^2dS\right)\!
  -\!\left(\int_{\partial^+B_r^+}\!t^{1-2s}\partial_\nu w\,
    w\,dS\right)^{\!2}=0.
\end{equation*}
Hence,  for a.e. $r\in(0,1)$, $w$ and $\partial_\nu w$ have the same direction as vectors
in $L^2(\partial^+ B^+_r,t^{1-2s}dS)$, so that there exists a
function $\eta=\eta(r)$, defined a.e. in $(0,1)$, such
that $\partial_\nu w(r\theta)=\eta(r)w(r\theta)$ for a.e. $r\in
(0,1)$ and for all $\theta\in \mathbb{S}^N_+$.
It is easy to verify that $\eta(r)=\frac{H'_w(r)}{2H_w(r)}$ for a.e $r\in
(0,1)$, so that $\eta\in
L^1_{\mathrm{loc}}((0,1])$. After integration we obtain that 
\begin{equation}\label{w(r,theta)}
  w(r\theta)= e^{\int_1^r \eta(s)ds}w(\theta)= g(r)\psi(\theta),\quad r\in (0,1),\ \theta\in \mathbb{S}^N_+,
\end{equation}
where $ g(r)=e^{\int_1^r \eta(s)ds}$ and
$\psi=w\big|_{\mathbb{S}^N_+}$. 
We observe that \eqref{intw=1} implies that 
\begin{equation}\label{eq:38}
\|\psi\|_{L^2({\mathbb S}^{N}_+,\theta_{N+1}^{1-2s}dS)}=1.  
\end{equation}
From the fact that $w\in H^1_{\Gamma_{1}^+}(B_{1}^+,t^{1-2s}\,dz)$
it follows that $\psi\in\mathcal H_0$;
moreover, plugging \eqref{w(r,theta)} into \eqref{eqlimite} we obtain
that $\psi$ satisfies \eqref{defautoval} for some $\mu\in\R$, so that
$\psi$ is an eigenfunction of \eqref{eig}. Recalling \eqref{eq:28} and
letting $k_0\in\mathbb{N}$ be
such that $\mu=\mu_{k_0}=(k_0+s)(k_0+N-s)$, we can rewrite the equation
$-\mathrm{div}\left(t^{1-2s}\nabla w\right)=0$ in polar coordinates
exploiting \cite[Lemma 2.1]{FalFel}, thus obtaining, for all
$r\in~\!\!(0,1)$ and $\theta\in\mathbb{S}^N_+$, 
\begin{align*}
0&=\frac{1}{r^N}(r^{N+1-2s}
g')'\theta_{N+1}^{1-2s}\psi(\theta)+r^{-1-2s}
g(r)\mathrm{div}_{\mathbb{S}^N}(\theta_{N+1}^{1-2s}\nabla_{\mathbb{S}^N}\psi(\theta))\\
&=
\frac{1}{r^N}(r^{N+1-2s}
g')'\theta_{N+1}^{1-2s}\psi(\theta)-r^{-1-2s}
g(r)
\theta_{N+1}^{1-2s}\mu_{k_0}\psi(\theta).
\end{align*}
Then $ g(r)$ solves the equation
\begin{equation*}
-\frac{1}{r^N}(r^{N+1-2s} g')'+\mu_{k_0} r^{-1-2s} g(r)=0\quad\text{in
}(0,1)
\end{equation*}
i.e.
\begin{equation*}
- g''(r)-\frac{N+1-2s}{r} g'(r)+\frac{\mu_{k_0}}{r^2} g(r)=0\quad\text{in
}(0,1).
\end{equation*}
Hence $ g(r)$ is of the form 
\begin{equation*}
g(r)=c_1 r^{k_0+s}+ c_2 r^{s-N-k_0}
\end{equation*}
for some $c_1, c_2\in\mathbb{R}$.
Since
  $w\in H^1(B_1^+,t^{1-2s}\,dz)$ and the function
  $|z|^{-1}|z|^{s-N-k_0}\psi\bigl(\frac{z}{|z|}\bigr)\not\in
  L^2(B_1^+,t^{1-2s}\,dz)$,
  from Lemma~\ref{l:hardy_boundary} we deduce that necessarily $c_2=0$ and $ g(r)=c_1 r^{{k_0}+s}$. Moreover, from $ g(1)=1$, we
obtain that $c_1=1$ and then
\begin{equation}\label{c_2=0}
  w(r\theta)=r^{k_0+s}\psi(\theta), 
  \quad\text{for all $r\in (0,1)$ and $\theta\in \mathbb{S}^N_+$}.
\end{equation}
Let us now consider the sequence $\{w^{\lambda_{n_k}}\}$.  Up to a
further subsequence still denoted by $\{w^{\lambda_{n_k}}\}$,
we may suppose that $w^{\lambda_{n_k}}\rightharpoonup \overline{w}$
weakly in $ H^1(B_1^+,t^{1-2s}\,dz)$ for some $\overline{w}\in
H^1(B_1^+,t^{1-2s}\,dz)$ and that $R_{\lambda_{n_k}}\rightarrow
\overline{R}$ for some $\overline{R}\in [1,2]$. 

Strong convergence of $w^{\lambda_{n_k}R_{\lambda_{n_k}}}$ in
$H^1(B^+_1,t^{1-2s}dz)$ implies that, up to a subsequence, both
$w^{\lambda_{n_k}R_{\lambda_{n_k}}}$ and
$|\nabla w^{\lambda_{n_k}R_{\lambda_{n_k}}}|$ are a.e. dominated by a
$L^2(B^+_1,t^{1-2s}dz)$-function uniformly with respect to
$k$. Moreover, by \eqref{doubling}, up to a further subsequence, we may assume
that the limit
\begin{equation*}
\ell:= \lim_{k\rightarrow +\infty}\frac{H(\lambda_{n_k}R_{\lambda_{n_k}})}{H(\lambda_{n_k})}
\end{equation*}
exists and is finite, with $\ell>0$. Then, by Dominated Convergence Theorem, we have 
\begin{align*}
  \lim_{k\rightarrow +\infty}&\int_{B^+_1}t^{1-2s}w^{\lambda_{n_k}}(z)v(z)dz=
                               \lim_{k\rightarrow +\infty}R_{\lambda_{n_k}}^{N+2-2s} \int_{B^+_{1/R_{\lambda_{n_k}}}}t^{1-2s}w^{\lambda_{n_k}}
                               (R_{\lambda_{n_k}}z)v(R_{\lambda_{n_k}} z)dz\\
                             &=\lim_{k\rightarrow
                               +\infty}R_{\lambda_{n_k}}^{N+2-2s}
                               \sqrt{\frac{H(\lambda_{n_k}R_{\lambda_{n_k}})}
                               {H(\lambda_{n_k})}}\int_{B^+_1}t^{1-2s}\chi_{B^+_{1/R_{\lambda_{n_k}}}}\!\!\!\!\!(z)\,
                               w^{\lambda_{n_k}R_{\lambda_{n_k}}}(z)v(R_{\lambda_{n_k}}z)dz\\
                             &=\overline{R}^{N+2-2s}\sqrt{\ell}
                               \int_{B^+_1}t^{1-2s}\chi_{B^+_{1/\overline{R}}}(z)w(z)v(\overline{R}z)dz\\
                             &=\overline{R}^{N+2-2s}\sqrt{\ell}
                               \int_{B^+_{1/\overline{R}}}t^{1-2s}w(z)v(\overline{R}z)dz=\sqrt{\ell}\int_{B^+_1}t^{1-2s}w(z/\overline{R})v(z) dz
\end{align*}
for any $v\in C^\infty (\overline{B_1^+})$. By
density, the above convergence actually holds for all
$v\in L^2(B^+_1,t^{1-2s}dz)$. This proves that
$w^{\lambda_{n_k}}\rightharpoonup \sqrt{\ell}w(\cdot/\overline{R})$
weakly in
$L^2(B^+_1,t^{1-2s}dz)$. Since we know that $w^{\lambda_{n_k}}\rightharpoonup~\!\!\overline{w}$
weakly in $ H^1(B_1^+,t^{1-2s}\,dz)$, we conclude that
$\overline{w}=\sqrt{\ell}w(\cdot/\overline{R})$ and then $w^{\lambda_{n_k}}\rightharpoonup \sqrt{\ell}w(\cdot/\overline{R})$. Moreover
\begin{align*}
    \lim_{k\rightarrow +\infty} \int_{B^+_1}t^{1-2s}&|\nabla
    w^{\lambda_{n_k}}(z)|^2 dz
    =\lim_{k\rightarrow
    +\infty}R_{\lambda_{n_k}}^{N+2-2s}\int_{B^+_{1/R_{\lambda_{n_k}}}}
    t^{1-2s}|\nabla w^{\lambda_{n_k}}(R_{\lambda_{n_k}}z)|^2 dz\\
  &=\lim_{k\rightarrow +\infty}
    R_{\lambda_{n_k}}^{N-2s}
    \frac{H(\lambda_{n_k}R_{\lambda_{n_k}})}{H(\lambda_{n_k})}
    \int_{B^+_1}t^{1-2s}\chi_{B^+_{1/R_{\lambda_{n_k}}}}
    |\nabla w^{\lambda_{n_k}R_{\lambda_{n_k}}}(z)|^2 dz\\
  &=\overline{R}^{N-2s}\ell
    \int_{B^+_1}t^{1-2s}\chi_{B^+_{1/\overline{R}}}(z)|\nabla
    w(z)|^2 dz=\overline{R}^{N-2s}\ell
    \int_{B^+_{1/\overline{R}}}t^{1-2s}|\nabla
    w(z)|^2 dz\\
  &=\int_{B^+_1}t^{1-2s}\Big|\sqrt{\ell}\nabla \Big(w\Big(\frac{z}{\overline{R}}\Big) \Big) \Big|^2 dz.
\end{align*}
This shows that $w^{\lambda_{n_k}}\rightarrow
\overline{w}=\sqrt{\ell}w(\cdot/\overline{R})$ 
strongly in $H^1(B^+_1,t^{1-2s}dz)$. 

By \eqref{c_2=0} $w$ is homogeneous of
  degree $k_0+s$, hence $\overline w=\sqrt\ell
  \,\overline{R}^{-k_0-s}w$. Furthermore \eqref{intwlambda1} and the
  strong convergence $w^{\lambda_{n_k}}\rightarrow
\overline{w}$ in $L^2(\partial^+B_1^+,t^{1-2s}dS)$ imply that
\begin{equation*}
1=\int_{\partial^+B_1^+}t^{1-2s}\overline w^2\,dS=\ell\,\overline{R}^{-2k_0-2s}
\int_{\partial^+B_1^+}t^{1-2s}w^2\,dS=\ell\,\overline{R}^{-2k_0-2s}
\end{equation*}
in view of \eqref{intw=1}, thus implying that
$\overline{w}=w$.

It remains to prove part (i). By \eqref{c_2=0}, \eqref{eq:38} and
the fact that $\psi$ is an eigenfunction of \eqref{eig} with
associated eigenvalue $\mu_{k_0}=(k_0+s)(k_0+N-s)$,  we have that 
\begin{equation*}
\int_{B^+_r}t^{1-2s}|\nabla w(z)|^2
dz=\frac{r^{N+2k_0}}{N+2k_0}\big((k_0+s)^2+\mu_{k_0}\big)=
(k_0+s)r^{N+2k_0}
\end{equation*}
and
\begin{equation*}
\int_{\partial^+ B^+_r}t^{1-2s}w^2 dS= r^{N+1-2s}\int_{\mathbb{S}^N_+} \theta_{N+1}^{1-2s}w^2(r\theta)\, dS= r^{N+2k_0+1}.
\end{equation*}
Therefore, by \eqref{Ew}, \eqref{Hw} and \eqref{Nw=gamma}, it follows that
\begin{equation*}
\gamma=\N_w(r)=\frac{E_w(r)}{H_w(r)}= \frac{r\int_{B^+_r} t^{1-2s} |\nabla w(z)|^2 dz}{\int_{\partial ^+ B^+_r}t^{1-2s}w^2 dS}=k_0+s.
\end{equation*}
This completes the proof.
\end{proof}

To complete the blow-up analysis and detect the sharp asymptotic behaviour of $W$ at $0$, it remains to describe the behavior of $H(\lambda)$ as $\lambda\rightarrow 0^+$.

\begin{Lemma}\label{LemmalimHesiste}
  Let $\gamma=\lim_{r\to0}\N(r)$ be as in Proposition
  \ref{Lemmalimitexists}. Then the limit
  $\lim_{r\to0^+}r^{-2\gamma} H(r)$ exists and is finite.
\end{Lemma}
\begin{proof}
Thanks to \eqref{Hupper}, it is enough to show that the limit
exists. From \eqref{E=rH'/2} we deduce that, a.e. in $(0,R_0)$, 
\begin{align}\label{derivHgamma}
  \frac{d}{dr}\frac{H(r)}{r^{2\gamma}}&=\frac{H'(r)}{r^{2\gamma}}-2\gamma\frac{H(r)}{r^{2\gamma+1}}
                                        =\frac{2}{r^{2\gamma+1}}\left(E(r)+H(r)O(r)-\gamma H(r)\right)\\
  \nonumber&=\frac{2H(r)}{r^{2\gamma+1}}\left(\N(r)-\gamma+O(r)\right)
             =\frac{2H(r)}{r^{2\gamma+1}}\left(\int_0^r\N'(s)\ ds+O(r)\right)
\end{align}
as $r\to 0^+$.
Using the notation of Lemma \ref{Lemmalimitexists}, we
can write $\N'=\alpha_1+\alpha_2$ 
in $(0,\hat R)$, with
\begin{equation*}
  \alpha_1(r)=\N'(r)+C_2r^{-1+\bar\delta}\left(C_1+\frac{N-2s}{2}\right)
  \quad\mathrm{and}\quad\alpha_2(r)=-C_4r^{-1+\bar\delta},
\end{equation*}
where $\bar\delta\in(0,1]$ has been defined in \eqref{delta} and
$C_4=C_2 \left(C_1+\frac{N-2s}{2}\right)$.
Integrating \eqref{derivHgamma} between $(r,\hat R)$, we obtain that
\begin{align*}
 & \frac{H(\hat R)}{\hat R^{2\gamma}}-\frac{H(r)}{r^{2\gamma}}
=\int_r^{\hat R}\tfrac{2H(\rho)}{\rho^{2\gamma+1}}\left({\textstyle{\int_0^\rho\alpha_1(\tau)d\tau}}\right)d\rho
    +\int_r^{\hat R}\tfrac{2H(\rho)}{\rho^{2\gamma+1}}\left({\textstyle{\int_0^\rho\alpha_2(\tau)d\tau}}\right)d\rho+\int_r^{\hat R}\tfrac{H(\rho)}{\rho^{2\gamma}}O(1)d\rho\\
&=\int_r^{\hat R}\tfrac{2H(\rho)}{\rho^{2\gamma+1}}\left({\textstyle{\int_0^\rho\alpha_1(\tau)d\tau}}\right)d\rho
    -\int_r^{\hat R}\frac{H(\rho)}{\rho^{2\gamma}}\left(-\frac{2C_4}{\bar \delta} \rho^{-1+\bar\delta}+O(1)\right)d\rho.
\end{align*}
Since  $\alpha_1\geq0$ by \eqref{Nlimitata} and \eqref{C1}, we have that
$\lim_{r\to 0^+} \int_r^{\hat R}\frac{2H(\rho)}{\rho^{2\gamma+1}}\left(\int_0^\rho\alpha_1(\tau)d\tau\right)d\rho$
exists.  On the other hand, estimate \eqref{Hupper} ensures that $\rho\mapsto
\frac{H(\rho)}{\rho^{2\gamma}}\left(-\frac{2C_4}{\bar \delta}
  \rho^{-1+\bar\delta}+O(1)\right)\in L^1(0,\hat R)$, so that the
limit $\lim_{r\to 0^+} \int_r^{\hat
  R}\frac{H(\rho)}{\rho^{2\gamma}}\left(-\frac{2C_4}{\bar \delta}
  \rho^{-1+\bar\delta}+O(1)\right)d\rho$ exists and is finite. The
lemma is thereby proved.
\end{proof}
The next step is the proof
  that the limit $\lim_{r\to0^+}r^{-2\gamma} H(r)$ is actually
  strictly positive. To this aim, we first define the Fourier
coefficients associated with $W$, with respect to the
  orthonormal basis \eqref{eq:orthobasis} of
  $L^2(\mathbb S^N_+,\theta_{N+1}^{1-2s}dS)$, as
\begin{equation}\label{Fcoeff}
\varphi_{k,m}(\lambda)=\int_{\mathbb
  S^N_+}\theta_{N+1}^{1-2s}W(\lambda\theta)Y_{k,m}(\theta)dS,
\quad\lambda\in(0,R_1],\ k\in{\mathbb N},\ m=1,\dots,M_k.
\end{equation}
We
also define
\begin{align}\label{Upsilon}
  \Upsilon_{k,m}&(\lambda)=-\int_{B^+_\lambda}t^{1-2s}(A-\mathrm{Id}_{N+1})\nabla
                            W(z)\cdot
                            \frac{\nabla_{\mathbb{S}^N}Y_{k,m}(z/|z|)}{|z|}dz\\
\nonumber&+\kappa_s\int_{\Gamma^-_{\lambda}}\tilde{h}(y)\mathrm{Tr}W(y)\mathrm{Tr}Y_{k,m}\Big(\frac{y}{|y|}\Big)dy
+\int_{\partial^+ B^+_\lambda} t^{1-2s}(A-\mathrm{Id}_{N+1}) \nabla
  W\cdot \frac{z}{|z|}Y_{k,m}\Big(\frac{z}{|z|}\Big)dS,
\end{align}
for a.e. $\lambda\in(0,R_1]$, $k\in{\mathbb N}$ and $m\in\{1,2,...,M_k\}$.
\begin{Lemma}\label{Lemmafourier}
  Let $k_0$ be as in Proposition \ref{propH}. Then, for all
  $m\in\{1,2,\dots,M_{k_0}\}$ and $R\in (0,R_0]$, 
\begin{multline}\label{asintoticacoeff}
  \varphi_{k_0,m}(\lambda)=\lambda^{k_0+s}\biggl(R^{-k_0-s}\varphi_{k_0,m}(R)
  +\frac{(k_0+s)R^{-N-2k_0}}{N+2k_0}\int_0^{R}\rho^{k_0+s-1}\Upsilon_{k_0,m}(\rho)\,d\rho\\
  +\frac{N-s+k_0}{N+2k_0}\int_\lambda^{R}\rho^{-N-1+s-k_0}\Upsilon_{k_0,m}(\rho)\,d\rho\biggr)+
O(\lambda^{k_0+s+\bar\delta})
  \quad \text{as $\lambda\rightarrow 0^+$},
\end{multline}
where $\bar\delta$
is defined in \eqref{delta}.
\end{Lemma}
\begin{proof}
Let $k\in{\mathbb N}$ and $m\in \{1,2,...,M_k\}$. Testing
  \eqref{prob4} with $\phi=\omega(|z|)|z|^{-N-1+2s}Y_{k,m}(z/|z|)$ for
  any test function $\omega\in C^{\infty}_c(0,R_1)$ and using
  \eqref{defautoval}, we can easily verify that
    $\varphi_{k,m}$ solves the following second order differential equation 
\begin{equation}\label{differeq}
-\varphi_{k,m}''(\lambda)-\frac{N+1-2s}{\lambda}\varphi_{k,m}'(\lambda)+\frac{\mu_k}{\lambda^2}\varphi_{k,m}(\lambda)=\zeta_{k,m}(\lambda)\qquad\mathrm{in \ }(0,R_1)
\end{equation}
in a distributional sense, with $\mu_k$  as in \eqref{eq:28}, where the distribution
$\zeta_{k,m}\in\mathcal D'(0,R_1)$ is defined by 
\begin{multline*}
_{\mathcal D'(0,R_1)}\langle\zeta_{k,m},\omega\rangle_{\mathcal
   D(0,R_1)}=\kappa_s\int_0^{R_1}\frac{\omega(\lambda)}{\lambda^{2-2s}}\left(\int_{{\mathbb
   S}^{N-1}_-}\tilde h(\lambda\theta')\mathop{\rm Tr}W(\lambda\theta')Y_{k,m}(\theta',0)dS'\right)d\lambda\\
-\int_{B_{R_1}^+}t^{1-2s}(A-\mathrm{Id}_{N+1})\nabla W\cdot\nabla
   \left(\omega(|z|)|z|^{-N-1+2s}Y_{k,m}(z/|z|)
\right)\,dz
\end{multline*}
for any $\omega\in C^{\infty}_c(0,R_1)$, where ${\mathbb S}^{N-1}_-=
\{(\theta_1,\dots,\theta_N)\in
{\mathbb S}^{N-1}:\theta_{N}\leq 0\}$.
Letting $\Upsilon_{k,m}$ be as in \eqref{Upsilon}, by
  direct calculations we have that $\Upsilon_{k,m}\in 
  L^1(0,R_1)$ and 
\begin{equation}\label{eq:distrib}
\Upsilon_{k,m}'(\lambda)=\lambda^{N+1-2s}\zeta_{k,m}(\lambda)\qquad \mathrm{in \ } \mathcal D'(0, R_1).
\end{equation}
In view of \eqref{eq:distrib} and \eqref{eq:28}, we have that \eqref{differeq} is equivalent to 
\begin{equation*}
  -\left(\lambda^{N+1+2k}\left(\lambda^{-k-s}\varphi_{k,m}\right)'\right)'
  =\lambda^{k+s}\Upsilon'_{k,m}\quad \text{in } \mathcal D'(0, R_1).
\end{equation*}
Integrating the above
equation, we obtain that,
for every 
$R\in(0,R_1]$, $k\in{\mathbb N}$ and $m\in \{1,2,\dots,M_{k}\}$, 
  there exists a real number $c_{k,m}(R)$ (depending also on $R$)
such that
\begin{multline}\label{phiW11}
  \left(\lambda^{-k-s}\varphi_{k,m}(\lambda)\right)'=
  -\lambda^{-N-1+s-k}\Upsilon_{k,m}(\lambda)\\-(k+s)
  \lambda^{-N-1-2k}\left(c_{k,m}(R)+\int_\lambda^{R}\rho^{k+s-1}\Upsilon_{k,m}(\rho)\,d\rho\right),
\end{multline}
in the sense of distributions in $(0,R_1)$. From
\eqref{phiW11} we infer that
$\varphi_{k,m}\in W^{1,1}_\mathrm{loc}((0,R_1])$, thus a new integration
leads to
\begin{multline}\label{prima:espress:coeff}
  \varphi_{k,m}(\lambda)=\lambda^{k+s}\left(\frac{\varphi_{k,m}(R)}{R^{k+s}}-\frac{(k+s)c_{k,m}(R)
    }{(N+2k) R^{N+2k}}
    +\frac{N+k-s}{N+2k}\int_\lambda^{R}\rho^{-N-k+s-1}\Upsilon_{k,m}(\rho)\,d\rho\right)\\
  +\frac{(k+s)\lambda^{-N-k+s}}{N+2k}\left(c_{k,m}(R)+\int_\lambda^{R}\rho^{k+s-1}\Upsilon_{k,m}(\rho)\,d\rho\right)
\end{multline}
for all $\lambda\in(0,R_1]$. 

 From now on, we fix $k_0$ as in Proposition
\ref{propH}, $R_0$ as in \eqref{eq:R_0}, 
and $m\in \{1,2,\dots,M_{k_0}\}$.
We prove that
\begin{equation}\label{firststep}
\int_0^{R_0}\rho^{-N-k_0+s-1}|\Upsilon_{k_0,m}(\rho)|\,d\rho <+\infty.
\end{equation}
To this purpose, exploiting \eqref{sviluppodellaA} and using H\"{o}lder's inequality, one can estimate the first term in \eqref{Upsilon} for all $\rho\in (0,R_0)$ as follows
\begin{align}\label{stimaprimoterm}
\biggl|\int_{B^+_\rho}t^{1-2s}(A-\mathrm{I}_{N+1})&\nabla W(z)\cdot \frac{\nabla_{\mathbb{S}^N}Y_{k_0,m}(z/|z|)}{|z|}dz \biggr|\\
\notag&\leq \mathop{\rm const} \sqrt{\int_{B^+_\rho}t^{1-2s}|\nabla W|^2
  dz}\cdot 
\sqrt{\int_{B^+_\rho}t^{1-2s}|\nabla_{\mathbb{S}^N}Y_{k_0,m}(z/|z|)|^2dz}\\
\notag&=:\mathop{\rm const} I_1(\rho)\cdot I_2(\rho),
\end{align}
  where
\begin{align}\label{I1}
  I_1(\rho)=\sqrt{\rho^{N+2-2s}\int_{B^+_1}t^{1-2s}|\nabla W(\rho z)|^2
  dz}&=\rho^{\frac{N-2s}{2}}\sqrt{H(\rho)}
       \sqrt{\int_{B^+_1}t^{1-2s}|\nabla w^\rho(z)|^2dz}\\
  \notag&\leq \mathop{\rm const} \rho^{\frac{N-2s}{2}}\sqrt{H(\rho)},
\end{align}
as a consequence of Lemma \ref{Lemmawlamlimit}, and 
\begin{equation}\label{I2}
\begin{split}
I_2(\rho)=\sqrt{\int_0^\rho
  \tau^{N+1-2s}\bigg(\int_{\mathbb{S}^N_+}\theta_{N+1}^{1-2s}|\nabla_{\mathbb{S}^N}Y_{k_0,m}(\theta)|^2
  dS\bigg)\,d\tau}= \frac{\sqrt{\mu_{k_0}}}{\sqrt{N+2-2s}}\rho^{\frac{N+2-2s}{2}},
\end{split}
\end{equation}
due to \eqref{defautoval}. Combining \eqref{stimaprimoterm},
\eqref{I1}, \eqref{I2}, and \eqref{Hupper} we obtain that, for every $R\in(0,R_0]$,
\begin{multline}\label{stimapezzo1}
  \int_0^{R}\rho^{-N-1+s-k_0}\biggl|\int_{B^+_\rho}t^{1-2s}(A-\mathrm{I}_{N+1})\nabla
  W(z)\cdot \frac{\nabla_{\mathbb{S}^N}Y_{k_0,m}(z/|z|)}{|z|}dz
  \biggr|\, d\rho\\
  \leq \mathop{\rm const} \int_0^{R}\rho^{-s-k_0}\sqrt{H(\rho)}\,ds
\leq \mathop{\rm const} R.
\end{multline}
Moreover, H\"{o}lder's inequality implies that
\begin{equation}\label{secondopezzo}
  \biggl|\int_{\Gamma^-_{\lambda}}\tilde{h}(y)\mathop{\rm Tr}W(y)\mathop{\rm Tr}Y_{k_0,m}\big(\tfrac{y}{|y|}\big)\,dy
  \biggr| \leq 
  \sqrt{\int_{\Gamma^-_\lambda}|\tilde{h}||\mathop{\rm Tr}W|^2 dy}
\cdot\sqrt{\int_{\Gamma^-_\lambda}|\tilde{h}(y)||\mathop{\rm Tr}Y_{k_0,m}\big (\tfrac{y}{|y|}\big)|^2 dy}.
\end{equation}
From \eqref{useful} and homogeneity of the function
  $Y_{k_0,m}(y/|y|)$ it follows that, for all $\rho\in (0,R_0]$,
\begin{multline*}
  \sqrt{\int_{\Gamma^-_\rho}|\tilde{h}(y)||\mathop{\rm
      Tr}Y_{k_0,m}\big(\tfrac{y}{|y|}\big)|^2 dy} \leq \sqrt{\tilde
    c_{N,s,p}}\,\Vert \tilde h\,\Vert_{L^p(\Gamma_{R_1}^-)}^{1/2}
  \rho^{\overline\varepsilon/2}
  \left(\int_{\Gamma_\rho^-}|\mathop{\rm Tr}Y_{k_0,m}\big (\tfrac{y}{|y|}\big)|^{2^*(s)}\ dy\right)^{\frac{1}{2^*(s)}}\\
  = \sqrt{\tilde c_{N,s,p}}\,\Vert \tilde
  h\,\Vert_{L^p(\Gamma_{R_1}^-)}^{1/2}\rho^{\frac{\overline\varepsilon+N-2s}2}\left(\int_{\Gamma_1^-}|\mathop{\rm
      Tr}Y_{k_0,m}\big (\tfrac{y}{|y|}\big)|^{2^*(s)}\
    dy\right)^{\frac{1}{2^*(s)}}.
        \end{multline*}
Using \eqref{useful}, \eqref{coercvity}, and \eqref{Nlimitata}, we obtain
that, for all $\rho\in (0,R_0]$, 
\begin{align*}
\sqrt{\int_{\Gamma^-_\rho}|\tilde{h}||\mathop{\rm Tr}W|^2 dy} &\leq
\sqrt{
\frac{\tilde c_{N,s,p}}{\tilde C_{N,s}}\Vert
  \tilde h\,\Vert_{L^p(\Gamma_{R_1}^-)}\rho^{\overline\varepsilon
    +N-2s}H(\rho)\left(\mathcal N(\rho)+\frac{N-2s}{2}\right)}\\
\notag&\leq \mathop{\rm const}\rho^{\frac{N-2s+\overline{\varepsilon}}{2}}\sqrt{H(\rho)}.
\end{align*}
Putting the above estimates together and recalling \eqref{Hupper}, we conclude that, for every $R\in (0,R_0]$,
\begin{multline}\label{stimapezzo2}
  \int_0^{R}\rho^{-N-k_0+s-1}
  \biggl|\int_{\Gamma^-_{\rho}}\tilde{h}(y)\mathrm{Tr}W(y)
\mathrm{Tr}Y_{k_0,m}\big(\tfrac{y}{|y|}\big)\,dy
  \biggr| d\rho\\
  \leq \mathop{\rm const}\int_0^{R}
  \rho^{-1+\overline{\varepsilon}-k_0-s}\sqrt{H(\rho)}\, d\rho \leq
  \mathop{\rm const}R^{\overline{\varepsilon}}.
\end{multline}
 In order to estimate the last term, we observe that,
since
$\|Y_{k_0,m}\|_{L^2({\mathbb S}^{N}_+,\theta_{N+1}^{1-2s}dS)}=1$,
\begin{equation}\label{I3}
  \int_{B^+_\lambda}t^{1-2s}\big|Y_{k_0,m}\big(\tfrac{z}{|z|}\big)\big|^2
  dz=\int_0^\lambda \tau^{N+1-2s}
  \bigg(\int_{\mathbb{S}^N_+}\theta_{N+1}^{1-2s}|Y_{k_0,m}(\theta)|^2
  dS\bigg) \,d\tau=\frac{\lambda^{N+2-2s}}{N+2-2s}.
\end{equation}
Hence,
integrating by parts, we have that, for every $R\in (0,R_0]$,
\begin{align}\label{terzopezzo}
  \int_0^{R} \rho^{-N+s-1-k_0} &\biggl|\int_{\partial^+ B^+_\rho}
                                 t^{1-2s}(A-\mathrm{Id}_{N+1}) \nabla W\cdot
                                 \frac{z}{|z|}Y_{k_0,m}\big(\tfrac{z}{|z|}\big)dS\biggr| d \rho \\
  \notag&\leq \mathop{\rm const}\int_0^ {R}\rho^{-N+s-k_0}
          \left(\int_{\partial^+ B^+_\rho} t^{1-2s}
          |\nabla W||Y_{k_0,m}\big(\tfrac{z}{|z|}\big)|\, dS\right) d\rho\\
  \notag&=\mathop{\rm const}\bigg( R^{-N+s-k_0} \int_{B^+_{R}}
          t^{1-2s} 
          |\nabla W|\big|Y_{k_0,m}\big(\tfrac{z}{|z|}\big)\big|\,dz\\
  \notag&\quad+(N+k_0-s)\int_0^{R} \rho^{-N+s-1-k_0} \bigg(\int_{B^+_\rho}
          t^{1-2s} |\nabla W|
          \big|Y_{k_0,m}\big(\tfrac{z}{|z|}\big)\big|\,dz\bigg) d\rho\bigg)\\
  \notag&\leq \mathop{\rm const}\left( R^{1-s-k_0}\sqrt{H(R)}+\int_0^{R} \rho^{-s-k_0}\sqrt{H(\rho)}\,d\rho\right) \leq \mathop{\rm const}R,
\end{align}
thanks to \eqref{sviluppodellaA}, H\"{o}lder inequality, \eqref{I1},
\eqref{I3}, and \eqref{Hupper}. From \eqref{Upsilon},
\eqref{stimapezzo1}, \eqref{stimapezzo2}, and \eqref{terzopezzo} it
follows that, for every $R\in (0,R_0]$,
\begin{equation}\label{eq:40}
  \int_0^{R}\rho^{-N-k_0+s-1}|\Upsilon_{k_0,m}(\rho)|\,d\rho\leq 
  \mathop{\rm const}R^{\bar\delta}
\end{equation}
for some $\mathop{\rm const}>0$ independent of $R$.
 From \eqref{eq:40}, we derive immediately  \eqref{firststep}.

From \eqref{firststep} it follows that, for every $R\in (0,R_0]$,
\begin{multline}\label{Olambdasigma}
\lambda^{k_0+s}\left(\frac{\varphi_{k_0,m}(R)}{R^{k_0+s}}-\frac{(k_0+s)c_{k_0,m}(R)
    }{(N+2k_0) R^{N+2k_0}}
+\frac{N+k_0-s}{N+2k_0}\int_\lambda^{R}\rho^{-N-k_0+s-1}\Upsilon_{k_0,m}(\rho)\,d\rho\right)\\
=O(\lambda^{k_0+s})=o(\lambda^{-N-k_0+s})\quad \text{as $\lambda\rightarrow 0^+$}.
\end{multline}
Now we prove that, for every $R\in (0,R_0]$,
\begin{equation}\label{claim}
c_{k_0,m}(R)+\int_0^{R}\rho^{k_0+s-1}\Upsilon_{k_0,m}(\rho)\,d\rho=0.
\end{equation}
In order to do this, first we observe that 
\begin{equation}\label{claim2}
\int_0^{R_0} \rho^{k_0+s-1}|\Upsilon_{k_0,m}(\rho)| d\rho<+\infty,
\end{equation}
as a direct consequence of \eqref{firststep}, since
  $k_0+s-1>-N-k_0+s-1$. Suppose by contradiction that \eqref{claim}
does not hold true for some $R\in (0,R_0]$; then from
\eqref{prima:espress:coeff}, \eqref{Olambdasigma} and \eqref{claim2},
we should have that
\begin{equation*}
  \varphi_{k_0,m}(\lambda)\sim
  \frac{(k_0+s)\lambda^{-N-k_0+s}}{N+2k_0}\left(c_{k_0,m}(R)
    +\int_0^{R}\rho^{k_0+s-1}\Upsilon_{k_0,m}(\rho)\,d\rho\right)\quad\text{as $\lambda\rightarrow 0^+$},
\end{equation*}
and hence
\begin{equation*}
\int_0^{R_0}\lambda ^ {N-1-2s} |\varphi_{k_0,m}(\lambda) | ^2 d\lambda = +\infty.
\end{equation*}
On the other hand, by \eqref{Fcoeff}, we have that 
\begin{equation*}
\int_0^{R_0}\lambda ^ {N-1-2s} |\varphi_{k_0,m}(\lambda) | ^2 d\lambda\leq \int_0^{R_0}\lambda ^ {N-1-2s}\left(\int_{\mathbb{S}^N_+}\theta_{N+1}^{1-2s}|W(\lambda\theta)|^2 dS\right)d\lambda=\int_{B_{R_0}^+}t^{1-2s}\frac{W^2(z)}{|z|^2}dz<\infty,
\end{equation*}
as a consequence of Lemma \ref{l:hardy_boundary},
giving rise to a contradiction. Hence \eqref{claim} holds true. From
\eqref{claim} and \eqref{eq:40} we deduce that, for every $R\in (0,R_0]$,
\begin{align*}
\biggl|
\lambda^{-N-k_0+s}&\left(c_{k_0,m}(R)+\int_\lambda^{R}\rho^{k_0+s-1}\Upsilon_{k_0,m}(\rho)\,d\rho\right)\biggr|=
\lambda^{-N+s-k_0}\biggl|\int_0^\lambda \rho^{k_0+s-1}\Upsilon_{k_0,m}(\rho)\,d\rho\biggr|\\
 &\leq \lambda^{-N+s-k_0}\int_0^\lambda \rho^{N+2k_0}|\rho^{-N-1+s-k_0}\Upsilon_{k_0,m}(\rho)|\,d \rho\\
&\leq \lambda^{k_0+s}\int_0^{\lambda}
\rho^{-N-1+s-k_0}|\Upsilon_{k_0,m}(\rho)|\,d \rho
= O(\lambda^{k_0+s+\bar\delta})\quad \text{as $\lambda\rightarrow 0^+$}.
\end{align*}
Combining this last information with \eqref{claim} and
\eqref{prima:espress:coeff}, we finally obtain \eqref{asintoticacoeff}. 
\end{proof}

Using Lemma  \ref{Lemmafourier}, we now prove
that $\lim_{r\rightarrow 0^+}r^{-2\gamma}H(r)=\lim_{r\rightarrow 0^+}r^{-2(k_0+s)}H(r)>0$.
\begin{Lemma}\label{LemmalimH>0}
Let $\gamma=\lim_{r\to0}\N(r)$ be as in Proposition
  \ref{Lemmalimitexists}. 
  Then
\begin{equation*}
\lim_{r\rightarrow 0^+}r^{-2\gamma}H(r)>0.
\end{equation*}

\end{Lemma}
\begin{proof}
By \eqref{sviluppomu} and using the Parseval identity we have that
\begin{align}\label{Hlambda}
  H(\lambda)&= \int_{\mathbb{S}^N_+}
              \theta_{N+1}^{1-2s}\mu(\lambda\theta)|W(\lambda\theta)|^2dS\\
  \notag&=(1+O(\lambda))\int_{\mathbb{S}^N_+}
          \theta_{N+1}^{1-2s}|W(\lambda\theta)|^2dS=(1+O(\lambda))
          \sum ^\infty_{k=0}\sum_{m=1}^{M_k}|\varphi_{k,m}(\lambda)|^2.
\end{align}
Let $k_0\in\mathbb N$ be as in Proposition \ref{propH}, thus $\gamma=k_0+s$. 
We argue by contradiction, assuming that 
\begin{equation}\label{assurdo}
  \lim_{\lambda\rightarrow 0^+}\lambda^{-2\gamma}H(\lambda)=0.
\end{equation}
Hence from \eqref{Hlambda} it follows that
$\lim_{\lambda\rightarrow
  0^+}\lambda^{-(k_0+s)}\varphi_{k_0,m}(\lambda)=0$  for any $m\in\{1,2,\dots,M_{k_0}\}$.
This and Lemma \ref{Lemmafourier} lead to 
\begin{multline}\label{lim=0}
  R^{-k_0-s}\varphi_{k_0,m}(R)
  +\frac{(k_0+s)R^{-N-2k_0}}{N+2k_0}\int_0^{R}\rho^{k_0+s-1}\Upsilon_{k_0,m}(\rho)\,d\rho\\
  +\frac{N-s+k_0}{N+2k_0}\int_0^{R}\rho^{-N-1+s-k_0}\Upsilon_{k_0,m}(\rho)\,d\rho=0,
\end{multline}
for all $m\in\{1,2,\dots,M_{k_0}\}$ and for every $R\in (0,R_0]$. 
From \eqref{lim=0}, \eqref{asintoticacoeff}, and \eqref{eq:40} it follows that 
\begin{equation*}
  \varphi_{k_0,m}(\lambda)=-\lambda^{k_0+s}\frac{N-s+k_0}{N+2k_0}
  \int_0^\lambda\rho^{-N-1+s-k_0}\Upsilon_{k_0,m}(\rho)\,d\rho+
  O(\lambda^{k_0+s+\bar\delta})=O(\lambda^{k_0+s+\bar\delta})
\end{equation*}
as $\lambda\rightarrow 0^+$ for all $m\in\{1,2,\dots,M_{k_0}\}$.
Hence
\begin{equation}\label{eq:39}
\sqrt{H(\lambda)}\,(w^\lambda,\psi)_{L^2({\mathbb S}^{N}_+,\theta_{N+1}^{1-2s}dS)}=O(\lambda^{k_0+s+\bar\delta})\quad\text{as
}\lambda\to0^+
\end{equation}
for every $\psi\in \mathop{\rm span}\{
Y_{k_0,m}:m=1,\dots, M_{k_0}\}$.
 From Lemma \ref{l:k1k2}-(ii), $\sqrt{H(\lambda)}\geq
\sqrt{k_2(\bar\delta)} \lambda^{k_0+s+\frac{\bar\delta}2}$ for
$\lambda$ small, so that \eqref{eq:39} yields
\begin{equation}\label{eq:43}
(w^\lambda,\psi)_{L^2({\mathbb S}^{N}_+,\theta_{N+1}^{1-2s}dS)}=O(\lambda^{\bar\delta/2})\quad\text{as
}\lambda\to0^+
\end{equation}
for every $\psi\in \mathop{\rm span}\{
Y_{k_0,m}:m=1,\dots, M_{k_0}\}$.  On the other hand, by
Proposition \ref{propH} and 
continuity of the trace map from $H^1(B_1^+,t^{1-2s}dz)$ into
$L^2(\partial^+B_1^+,t^{1-2s}dS)$,
 for any sequence $\lambda_n\to0^+$, there exist a
  subsequence $\{\lambda_{n_k}\}$ and $\psi_0\in \mathop{\rm span}\{
Y_{k_0,m}:m=1,\dots, M_{k_0}\}$ such that
  \begin{equation}\label{eq:44}
    \|\psi_0\|_{L^2({\mathbb
        S}^{N}_+,\theta_{N+1}^{1-2s}dS)}=1\quad\text{and}
    \quad w^{\lambda_{n_k}}\to\psi_0 \quad\text{in }L^2({\mathbb S}^{N}_+,\theta_{N+1}^{1-2s}dS).
\end{equation}
From (\ref{eq:43}) and (\ref{eq:44}) we deduce that
\[
0=\lim_{k\to\infty}(w^{\lambda_{n_k}},\psi_0)_{L^2({\mathbb S}^{N}_+,\theta_{N+1}^{1-2s}dS)}
=\|\psi_0\|_{L^2({\mathbb S}^{N}_+,\theta_{N+1}^{1-2s}dS)}^2=1,
\]
thus reaching a contradiction.
\end{proof}
\begin{Theorem}\label{t:4.9}
Let $k_0\in\mathbb N$ be as in Proposition \ref{propH}. Let
  $M_{k_0}\in\mathbb N\setminus\{0\}$ be the multiplicity of the
  eigenvalue $\mu_{k_0}=(k_0+s)(k_0+N-s)$ and let
  $\{Y_{k_0,m}\}_{m=1,...,M_{k_0}}$ be a
  $L^2(\mathbb S^N_+,\theta_{N+1}^{1-2s}dS)$-orthonormal basis of the
  eigenspace of \eqref{eig} associated to $\mu_{k_0}$. Then, for every
  $m\in \{1,2,\dots,M_{k_0}\}$, there exists $\beta_m\in{\mathbb R}$
  such that $(\beta_1,\beta_2,\dots,\beta_{M_{k_0}})\neq(0,0,\dots,0)$,
\begin{equation*}
\frac{W(\lambda z)}{\lambda^{k_0+s}}\rightarrow |z|^{k_0+s}\sum_{m=1}^{M_{k_0}}\beta_m Y_{k_0,m}(z/|z|)\quad \text{in $H^1(B_1^+,t^{1-2s}dz)$ as $\lambda\rightarrow 0^+$},
\end{equation*}
and
\begin{multline}\label{eq:coeff-beta}
  \beta_m=R^{-(k_0+s)}\varphi_{k_0,m}(R)+\frac{(k_0+s)
    R^{-N-2k_0}}{N+2k_0}\int_0^{R}\rho^{k_0+s-1}
  \Upsilon_{k_0,m}(\rho)\,d\rho\\
  +\frac{N-s+k_0}{N+2k_0}\int_0^{R}\rho^{-N-1+s-k_0}\Upsilon_{k_0,m}(\rho)\,d\rho\quad
\text{for all }R\in(0,R_0],
\end{multline}
with $\varphi_{k_0,m} $ and $\Upsilon_{k_0,m}$ given by \eqref{Fcoeff}
and \eqref{Upsilon} respectively. 
\end{Theorem}
\begin{proof}
  If we consider any sequence of strictly positive real numbers
  $\lambda_n\rightarrow 0^+$, then from Proposition \ref{propH} and Lemmas
  \ref{LemmalimHesiste} and \ref{LemmalimH>0}, we deduce that there
  exist a subsequence $\{\lambda_{n_k}\}$ and real numbers
  $\beta_1,\beta_2, \dots\beta_{M_{k_0}}$ not all equal to $0$ such that
\begin{equation}\label{thm1}
  \frac{W(\lambda_{n_k} z)}{\lambda_{n_k}^{k_0+s}}\rightarrow
  |z|^{k_0+s}
  \sum_{m=1}^{M_{k_0}}\beta_m Y_{k_0,m}(z/|z|)\quad \text{in
    $H^1(B_1^+,t^{1-2s}dz)$ as $k\rightarrow \infty$}.
\end{equation}
We claim that the coefficients $\beta_m$ depend neither on the
sequence $\{\lambda_n\}$, nor on its subsequence $\{\lambda_{n_k}\}$. To
this aim, we observe that \eqref{Fcoeff}, \eqref{thm1}, and the
continuity of the trace map from $H^1(B_1^+,t^{1-2s}dz)$ into $L^2(\partial^+B_1^+,t^{1-2s}dS)$ imply that,
for all $m\in \{1,2,\dots,M_{k_0}\}$,
\begin{equation*}
\begin{split}
\lim_{k\rightarrow \infty} \lambda_{n_k}^{-(k_0+s)}\varphi_{k_0,m}(\lambda_{n_k})&=\lim_{k\rightarrow +\infty} \int_{\mathbb{S}^N_+}\theta_{N+1}^{1-2s}\lambda_{n_k}^{-\gamma} W(\lambda_{n_k}\theta)Y_{k_0,m}(\theta)dS\\
&= \sum_{i=1}^{M_{k_0}}\beta_i\int_{\mathbb{S}^N_+}\theta_{N+1}^{1-2s}Y_{k_0,i}(\theta)Y_{k_0,m}(\theta)dS=\beta_m,
\end{split}
\end{equation*}
for all $m\in\{1,2,\dots,M_{k_0}\}$. At the same time, after fixing $R<R_0$, by \eqref{asintoticacoeff} we have that
\begin{equation*}
\begin{split}
\lim_{k\rightarrow \infty} \lambda_{n_k}^{-(k_0+s)}\varphi_{k_0,m}(\lambda_{n_k})=&R^{-(k_0+s)}\varphi_{k_0,m}(R)+\frac{(k_0+s)R^{-N-2k_0}}{N+2k_0}\int_0^{R}\rho^{k_0+s-1}\Upsilon_{k_0,m}(\rho)\,d\rho\\
&+\frac{N-s+k_0}{N+2k_0}\int_0^{R}\rho^{-N-1+s-k_0}\Upsilon_{k_0,m}(\rho)\,d\rho,
\end{split}
\end{equation*}
hence, by uniqueness of the limit, we can deduce that, for all $m\in\{1,2,\dots,M_{k_0}\}$,
\begin{equation*}
\begin{split}
  \beta_m=&
  R^{-(k_0+s)}\varphi_{k_0,m}(R)+\frac{(k_0+s)R^{-N-2k_0}}{N+2k_0}\int_0^{R}\rho^{k_0+s-1}\Upsilon_{k_0,m}(\rho)\,d\rho\\
  &+\frac{N-s+k_0}{N+2k_0}\int_0^{R}\rho^{-N-1+s-k_0}\Upsilon_{k_0,m}(\rho)\,d\rho.
\end{split}
\end{equation*}
This is enough to conclude that the coefficients $\beta_m$ depend
neither on the sequence $\{\lambda_n\}$, nor on its subsequence
$\{\lambda_{n_k}\}$.  Urysohn's Subsequence Principle allows us
to conclude that the convergence in \eqref{thm1} holds as
$\lambda\to0^+$, thus completing the proof.
\end{proof}

We are now in position to prove Theorem \ref{t:asymp-U}.
  \begin{proof}[Proof of Theorem \ref{t:asymp-U}]
Up to a translation, we can assume that $x_0=0$.    If $U$ is as in the assumptions of Theorem \ref{t:asymp-U}, then,
    letting $F$ as in Section \ref{sect2.1}, 
 $W=U\circ F\in H^1_{\Gamma_{R_1}^+}(B_{R_1}^+,t^{1-2s}\,dz)$ is a
nontrivial weak solution of \eqref{prob4}.
We notice that the nontriviality of $U$ in any neighbourhood of $0$,
and consequently  of $W$ in $B_{R_1}^+$, can
be easily deduced from nontriviality of $U$ in $\R^{N+1}_+$ and 
classical unique continuation principles for second order elliptic
equations with Lipschitz coefficients \cite{Garlin}.

 Then, by Proposition  \ref{propH}
 and Theorem \ref{t:4.9}, there exist $k_0\in{\mathbb N}$  and  an
 eigenfunction $Y$ of problem  \eqref{eig} associated to the eigenvalue
 $\mu_{k_0}=(k_0+s)(k_0+N-s)$  such that 
\begin{equation}\label{eq:42}
  \frac{W(\lambda z)}{\lambda^{k_0+s}}\rightarrow
  |z|^{k_0+s}Y(z/|z|)
  \quad \text{in $H^1(B_1^+,t^{1-2s}dz)$ as $\lambda\rightarrow 0^+$}.
\end{equation}
We observe that 
\begin{equation}\label{eq:41}
 \frac{U(\lambda z)}{\lambda^{k_0+s}}= \frac{W(\lambda
   G_\lambda(z))}{\lambda^{k_0+s}},\quad 
 \nabla\left(\frac{U(\lambda \cdot)}{\lambda^{k_0+s}}\right)(z)=\nabla\left( \frac{W(\lambda
\cdot)}{\lambda^{k_0+s}}\right)(   G_\lambda(z))\,J_{G_{\lambda}}(z),
\end{equation}
where 
\begin{equation*}
G_\lambda(z)=\frac{1}{\lambda}F^{-1}(\lambda z).
\end{equation*}
From \eqref{eq:11} we have that 
\begin{equation}\label{eq:4}
  G_\lambda(z)=z+O(\lambda)\quad\text{and}\quad 
J_{G_{\lambda}}(z)=\mathrm{Id}_{N+1}+O(\lambda)
\end{equation}
as $\lambda\to0^+$ uniformly with respect to $z\in B_1^+$.
From \eqref{eq:4} one can easily deduce that, if $f_\lambda\to f$ in
$L^2(B^+_1,t^{1-2s}dz)$, then  $f_\lambda\circ G_\lambda\to f$ in
$L^2(B^+_1,t^{1-2s}dz)$. In view of \eqref{eq:42} and \eqref{eq:41},
this yields the conclusion.
 \end{proof}

As a direct consequence of Theorem \ref{t:asymp-U} and of the
 equivalent formulation of problem \eqref{prob1} given in
 \eqref{prob2}, we obtain Theorem \ref{t:asymp-u}
 \begin{proof}[Proof of Theorem \ref{t:asymp-u}]
If $u\in \mathcal D^{s,2}(\R^N)$, $u\not\equiv0$,  is a nontrivial
weak solution to \eqref{prob1}, then its extension $U=\mathcal H(u) \in
\mathcal{D}^{1,2}(\R^{N+1}_+,t^{1-2s}\,dz)$ weakly solves
\eqref{prob2} in the weak sense specified in \eqref{eq:45}, see
\cite{CafSil1} and Section \ref{sec:intr-main-results}.    Then the
conclusion follows from Theorem \ref{t:asymp-U} applied to $U$ and the continuity of
the trace map from $H^1(B_1^+,t^{1-2s}dz)$ into $H^s(B_1')$, see e.g.
\cite[Proposition 2.1]{JLX}.
 \end{proof}

\appendix

\section{Some boundary regularity results at edges of cylinders}\label{sec:some-bound-regul}
Let us consider the following local problem: $\Omega\subset\R^N$ is a
$C^{1,1}$ domain, $x_0\in\partial\Omega$, $R,T>0$ and $U$ is  a weak solution to
\begin{equation}\label{corner}
\begin{cases}
\mathrm{div}\left(t^{1-2s}\nabla U\right)=0 &\text{in }C_{R,T}(x_0),\\
U=0 & \text{in }D_{R,T}(x_0),\\
\lim_{t\to0}t^{1-2s}\partial_tU=0 &\text{in }\sigma_{R,T}(x_0),
\end{cases}
\end{equation}
where we denoted 
\begin{gather*}
C_{R,T}(x_0) :=(B'_R(x_0)\cap\Omega)\times(0,T), \quad D_{R,T}(x_0)
:=(B'_R(x_0)\cap\partial\Omega)\times(0,T),\\
 \sigma_{R,T}(x_0)
=(B'_R(x_0)\cap\Omega)\times\{0\};
\end{gather*}
i.e. $U$ belongs to the space  $\mathcal H$ defined as the closure of
the set 
\begin{equation*}
\{v\in C^\infty(\overline {C_{R,T}(x_0)}):v=0\text{ in a
  neighbourhood of $D_{R,T}(x_0)$}\}
\end{equation*}
in
$H^1(C_{R,T}(x_0),t^{1-2s}\,dz)$, and
\begin{equation*}
\int_{C_{R,T}(x_0)} t^{1-2s}\nabla U\cdot\nabla\Phi\ dz=0
\quad\text{for all }\Phi\in
C^\infty_c(C_{R,T}(x_0)\cup \sigma_{R,T}(x_0)).
\end{equation*}
The following regularity result holds true.
\begin{Lemma}\label{regularitycorner}
Let
$\alpha\in(0,1)$, $\beta\in(0,1)\cap(0,2-2s)$, $r<R$, and $\tau<T$.
 Then there exists a positive constant $C$ such
 that, 
for every weak solution $U$ to \eqref{corner},
\begin{equation*}
\|U\|_{C^{1,\alpha}(C_{r,\tau}(x_0))}+\|t^{1-2s}\partial_tU\|_{C^{0,\beta}(
C_{r,\tau}(x_0))}\leq C\|U\|_{L^2(C_{R,T}(x_0),t^{1-2s}dz)}.
\end{equation*}
\end{Lemma}
\begin{proof}
Denoting the total variable $z=(x,t)\in\R^N\times(0,+\infty)$, with
$x=(x',x_N)\in \R^{N-1}\times \R$, let us consider $g\in
C^{1,1}(\R^{N-1})$ such that $B'_R(x_0)\cap\Omega=\{x=(x',x_N)\in
B'_R(x_0): x_N<g(x')\}$. Without loss of generality we can assume that
$x_0=0$, $g(0)=0$ and $\nabla g(0)=0$.
Starting from this function $g$, we can argue as in Section
\ref{sect2.1} and construct a function $F$ as in \eqref{F}, which
turns out to be a diffeomorphism in a neighbourhood of $0$. Hence
there exist positive constants $r_0<R$ and $\tau_0<T$ such that 
the composition $W=U\circ F$ weakly solves the following straightened problem
\begin{equation*}
\begin{cases}
\mathrm{div}\left(t^{1-2s}A\nabla W\right)=0 &\text{in }\Gamma_{r_0}^-\times(0,\tau_0),\\
W=0 &\text{in } (B'_{r_0}\cap\{y_N=0\})\times(0,\tau_0),\\
\lim_{t\to0}t^{1-2s}A\nabla W\cdot\nu=0 &\text{in }\Gamma_{r_0}^-,
\end{cases}
\end{equation*}
with $A=A(y)$ being as in \eqref{A}; in particular the matrix $A(y)$ does not depend on the vertical variable $t$, is
symmetric, uniformly elliptic, and possesses $C^{0,1}$ coefficients.

Let us  consider the odd reflection of $W$ 
(which we still denote as $W$)
through the hyperplane $\{y_N=0\}$ in $B'_{r_0}\times (0,\tau_0)$, i.e. we
set 
$W(y',y_N,t)=-W(y',-y_N,t)$ for $y_N<0$;  it is easy to verify that $W$ 
weakly satisfies
\begin{equation*}
\begin{cases}
\mathrm{div}\left(t^{1-2s}\widetilde A\nabla W\right)=0 &\text{in }B_{r_0}'\times(0,\tau_0),\\
\lim_{t\to0}t^{1-2s}\widetilde A\nabla W\cdot\nu=0 &\text{in }B_{r_0}',
\end{cases}
\end{equation*}
	where
	\begin{equation*}
	\widetilde{A}(y)=\widetilde{A}(y',y_N):=\begin{cases}
	A(y',y_N), &\text{if }y_N\leq 0, \\
	SA(y',-y_N) S, &\text{if }y_N>0,
	\end{cases}
	\end{equation*}
	with
	\begin{equation*}
	S:=
\left( \renewcommand{\arraystretch}{1.5}
\begin{array}{c|c|c}
\mathrm{Id}_{N-1}&{\mathbf 0}&{\mathbf 0}\\\hline
{\mathbf 0}^T&-1&0\\\hline
{\mathbf 0}^T&0&1
\end{array}\right),
\end{equation*}
We observe that no discontinuities
appear in the coefficients of the matrix $\widetilde A$ since,
denoting as $(a_{ij})$ the entries of the matrix $A$,
$a_{i,N}(y',0,t)=0$ for all $i<N$ thanks to \eqref{eq:1} and \eqref{eq:6}.
Then the matrix  $\widetilde A$ has Lipschitz continuous coefficients.
Let us  then consider the even reflection of $W$ 
(which we still denote as $W$)
through the hyperplane $\{t=0\}$ in $B'_{r_0}\times (-\tau_0,\tau_0)$, i.e. we
set 
$W(y',y_N,t)=W(y',y_N,-t)$ for $t<0$; due to the homogeneous Neumann
type boundary condition satisfied by $W$ on $B_{r_0}'$ and the fact that
the matrix $A$ is independent of $t$, we obtain that
such even reflection through  $\{t=0\}$ weakly solves 
\begin{equation*}
\mathrm{div}\left(|t|^{1-2s}\widetilde A\nabla W\right)=0 \quad\text{in }B_{r_0}'\times(-\tau_0,\tau_0).
\end{equation*}
From \cite[Lemma 7.1]{SirTerVit1} it follows that 
$V=|t|^{1-2s}\partial_tW\in H^1_{\rm
  loc}(B_{r_0}'\times(-\tau_0,\tau_0),|t|^{2s-1}\,dz)$ is a weak solution to
\begin{equation*}\label{corner3}
\mathrm{div}\left(|t|^{2s-1}\widetilde A\nabla V\right)=0 \quad\text{in } B_{r_0}'\times(-\tau_0,\tau_0) 
\end{equation*}
which is odd with respect to $\{t=0\}$,
i.e. $V(y',y_N,-t)=-V(y',y_N,-t)$.

From \cite[Theorem 1.2]{SirTerVit1} it follows that, for all $r\in(0,r_0)$ and
$\tau\in(0,\tau_0)$, $W\in  C^{1,\alpha}(B_{r}'\times(-\tau,\tau))$
and $\|W\|_{ C^{1,\alpha}(B_{r}'\times(-\tau,\tau))}\leq
{\rm const}\|W\|_{L^2(B_{r_0}'\times(-\tau_0,\tau_0),
|t|^{1-2s}\,dz)}$ for some ${\rm const}>0$ (independent of $W$). 
Furthermore, \cite{FabKenSer} ensures that $V$ is
locally H\"older continous. More precisely, \cite[Proposition 2.10]{SirTerVit2} yields that the
function 
$\Phi(x,t)=\frac{V(x,t)}{t|t|^{1-2s}}$, which  is even in the variable $t$,
belongs to the weighted Sobolev space   $H^1_{\rm
  loc}(B_{r_0}'\times(-\tau_0,\tau_0),|t|^{3-2s}\,dz)$, and weakly solves 
\begin{equation*}
\mathrm{div}\left(|t|^{3-2s}\widetilde A\nabla \Phi\right)=0 \quad\text{in } B_{r_0}'\times(-\tau_0,\tau_0),
\end{equation*}
thanks to the fact that the matrix $\widetilde A$ is independent of $t$.

From \cite[Theorem 1.2]{SirTerVit1} we have that
$\Phi\in  C^{0,\gamma}(B_{r}'\times(-\tau,\tau))$ for all
$\gamma\in(0,1)$, $r\in(0,r_0)$ and
$\tau\in(0,\tau_0)$, 
and 
\begin{equation*}
\|\Phi\|_{ C^{0,\gamma}(B_{r}'\times(-\tau,\tau))}=
\left\|\frac{V}{t|t|^{1-2s}}\right\|_{ C^{0,\gamma}(B_{r}'\times(-\tau,\tau))}
\leq
{\rm const\,}\|V\|_{L^2(B_{r_0}'\times(-\tau_0,\tau_0),
|t|^{2s-1}\,dz)}
\end{equation*} for some ${\rm const}>0$ (independent of $V$). Therefore $V\in
C^{0,\delta}(B_{r}'\times(-\tau,\tau))$ with
$\delta=\min\{2-2s,\gamma\}$ and 
$\|V\|_{ C^{0,\delta}(B_{r}'\times(-\tau,\tau))}\leq
{\rm const\,}\|V\|_{L^2(B_{r_0}'\times(-\tau_0,\tau_0),
|t|^{2s-1}\,dz)}$.

The conclusion follows by recalling that $U=W\circ F^{-1}$ with
$F^{-1}$ being of class $C^{1,1}$ and taking into account the
particular form of the matrix in \eqref{eq:JacF}.
\end{proof}

\section{Homogeneity degrees and
  eigenvalues of the spherical problem}\label{sec:color-eigenv-probl}
  
In this appendix, we derive an  explicit formula for the eigenvalues
of problem \eqref{eig}, which follows  from a complete
classification of  possible homogeneity degrees of homogeneous weak
solutions to the problem 
\begin{equation}\label{eq:27}
\begin{cases}
-\mathrm{div}\left(t^{1-2s}\nabla \Psi\right)=0 &\text{in } \R^{N+1}_+,\\
\lim_{t\to0^+}\left(
t^{1-2s}\nabla \Psi\cdot\nu\right)=0 &\text{in }\Gamma^-,\\
\Psi=0 &\text{in }\Gamma^+,
\end{cases}
\end{equation}
where $\Gamma^-:=\{(y',y_N,0)\in
\R^N\times\{0\}:y_N< 0\}$ and $\Gamma^+:=\{(y',y_N,0)\in
\R^N\times\{0\}:y_N\geq 0\}$. 

\begin{Proposition}\label{gamma=k+s}
Let $\Psi\in \cap_{r>0}H^1_{\Gamma_{r}^+}(B_r^+,t^{1-2s}\,dz)$ be a
weak solution to \eqref{eq:27}, i.e. 
\begin{equation*}
\int_{\R^{N+1}_+}t^{1-2s}\nabla\Psi\cdot\nabla\Phi\,dz=0,\quad\text{for
  all $\Phi\in C^\infty_c(\overline{\R^{N+1}_+}\setminus \Gamma^+)$}.
\end{equation*}
If, for some $\gamma\geq0$, $\Psi(z)=|z|^\gamma\Psi(\frac{z}{|z|})$,
then there exists $j\in\mathbb N$ such that $\gamma=j+s$.
\end{Proposition}
The proof of  Proposition \ref{gamma=k+s} requires  a polynomial
Liouville type theorem for even solutions to degenerate equations with
a weight which is possibly out of the $A_2$-Muckenhoupt class. To this
aim, Lemma \ref{Liouville} below provides a generalization of Lemma
2.7 in \cite{CafSalSil}.
For all $a\in(-1,+\infty)$ and $r>0$, we define  
$H^1(B_r,|t|^{a}\,dz)$ as the completion of 
$C^\infty(\overline{B_r})$ with respect to the norm
$\sqrt{\int_{B_r}|t|^{a}\left(|\Psi|^2+|\nabla
\Psi|^2\right)\,dz}$ and  $H^{1,a}_{\rm loc}(\R^{N+1})$ as 
\begin{equation*}
H^{1,a}_{\rm loc}(\R^{N+1})=\{\Psi\in L^2_{\rm
  loc}(\R^{N+1},|t|^{a}\,dz): \Psi\in H^1(B_r,|t|^{a}\,dz)\text{ for
  all }r>0\}.
\end{equation*}
We also define 
\begin{equation*}
H^{1,a}_{\rm loc}(\overline{\R^{N+1}_+})=\{\Psi\in L^2_{\rm
  loc}(\overline{\R^{N+1}_+},t^{a}\,dz): \Psi\in H^1(B_r^+,t^{a}\,dz)\text{ for
  all }r>0\}.
\end{equation*}

\begin{Lemma}\label{Liouville}
Let $a\in(-1,+\infty)$ and $v\in H^{1,a}_{\rm loc}(\R^{N+1})$ be a weak solution to
\begin{equation}\label{evenliouville}
\mathrm{div}(|t|^a\nabla v)=0 \quad\text{in }\R^{N+1}
\end{equation}
which is even in t, i.e.
\begin{equation*}
v(x,-t)=v(x,t)\quad\text{a.e. in }\R^{N+1}.
\end{equation*}
If there exist $k\in\mathbb N$ and $c>0$ such that
\begin{equation*}
|v(z)|\leq c(1+|z|^k)\qquad\mathrm{for \ all \ }z\in\R^{N+1},
\end{equation*}
then $v$ is a polynomial.
\end{Lemma}
\begin{proof}
  Let $a>-1$ and $v\in H^{1,a}_{\rm loc}(\R^{N+1})$ be a weak solution
  to \eqref{evenliouville} even in $t$. For $\alpha\in(0,1)$ and
  $k\in\mathbb N$, let $D^{\beta_k}_xv$ be a partial derivative in the
  variables $x=(x_1,...,x_N)$ of order $k=|\beta_k|$, with
  $\beta_k\in\mathbb N^N$ multiindex. Then, there exists a positive
  constant $C$ depending only on $N,\alpha,a,k$ such that
\begin{equation}\label{sup}
\sup_{B_{r/2}}|D^{\beta_k}_xv|\leq \frac{C}{r^k}\sup_{B_{r}}|v|
\end{equation}
and
\begin{equation}\label{seminorm}
[D^{\beta_k}_xv]_{C^{0,\alpha}(B_{r/2})}\leq \frac{C}{r^{k+\alpha}}\sup_{B_{r}}|v|,
\end{equation}
where 
$[w]_{C^{0,\alpha}(\Lambda)}:=\sup_{z,z'\in\Lambda}|z-z'|^{-\alpha}|w(z)-w(z')|$.
In order to prove the previous inequalities we apply some local regularity estimates for even solutions contained in \cite{SirTerVit1}. If $k=0$, then the inequalities follow by scaling
\begin{equation*}
\|v\|_{C^{0,\alpha}(B_{1/2})}\leq C\|v\|_{L^\infty(B_{1})}
\end{equation*}
proved in \cite[Theorem 1.2 part $i)$]{SirTerVit1}. If $k\geq1$, we remark that any
partial derivation in variables $x_i$ for $i\in\{1,...,N\}$ commutes
with the operator $\mathrm{div}(|t|^a\nabla \cdot)$ and
$D^{\beta_k}_xv$ are actually even solutions to the same equation, 
(see \cite[Section 7]{SirTerVit1} for
details). Hence, inequalities \eqref{sup} and \eqref{seminorm} follow
by scaling and iterating the estimate
\begin{equation*}
\|v\|_{C^{1,\alpha}(B_{1/2})}\leq C\|v\|_{L^\infty(B_{1})}
\end{equation*}
proved in \cite[Theorem 1.2 part $ii)$]{SirTerVit1}. Indeed, fixed a multiindex
$\beta_k$, we can choose 
\begin{equation*}
r_k=1/2<r_{k-1}<...<r_0=1,
\end{equation*}
then
\begin{align*}
\|D^{\beta_k}_xv\|_{C^{0,\alpha}(B_{1/2})}&\leq C_{k-1}\sup_{B_{r_{k-1}}}|D^{\beta_{k-1}}_xv|\leq C_{k-1}C_{k-2}\sup_{B_{r_{k-2}}}|D^{\beta_{k-2}}_xv|\\
&\leq ...\leq\left( \prod_{i=0}^{k-1} C_{i}\right)\sup_{B_{1}}|v|.
\end{align*}
Once we have \eqref{sup} and \eqref{seminorm}, we can proceed exactly
as in proof of \cite[Lemma 2.7]{CafSalSil}. We have only to remark
that for any $a\in(-1,+\infty)$, given an even solution to \eqref{evenliouville}
$v$, then $\partial^2_{tt}v+\frac{a}{t}\partial_tv=-\Delta_xv$ is also an even
  solution to \eqref{sup}.
\end{proof}
Now we are able to prove Proposition \ref{gamma=k+s}.
\begin{proof}[Proof of Proposition \ref{gamma=k+s}]
Let $\Psi\in H^{1,1-2s}_{\rm loc}(\overline{\R^{N+1}_+})$ be a
weak solution to \eqref{eq:27}, such that 
\begin{equation*}
\Psi(z)=|z|^\gamma\Psi\bigg(\frac{z}{|z|}\bigg)\quad\text{in }\R^{N+1}_+,
\end{equation*}
 for some $\gamma\geq0$.  The homogeneity condition trivially
implies a polynomial global bound on the growth of $\Psi$. The same
bound is inherited by the trace $\phi=\mathop{\rm Tr}\Psi$ on
$\R^{N}=\partial \R^{N+1}_+$, which is also
$\gamma$-homogeneous. Moreover, $\phi\in C^{\infty}(\Gamma^-)$ by
\cite[Theorem 1.1]{SirTerVit1} and $\phi\in C^{0}(\R^N)$ by
\cite[Proposition 5.3]{Nia}. With these
premises, we can define the extension $V$ of $\phi$ in the sense of
\cite[Lemma 3.3]{AbaRos}. Actually, we introduce a minor change in the
definition of the extension given in \cite{AbaRos}; that is, for every $R>0$ we define
\begin{equation}\label{estensioneregolare}
\phi_R=\phi\eta_R
\end{equation}
(instead of $\phi_R=\phi\chi_{B'_R}$), where
$\eta_R\in C^{\infty}_c(B'_{2R})$ is a radially decreasing cut-off
function with $|\eta_R|\leq1$ and $\eta_R\equiv1$ in $B'_R$.  We
remark that the adjusted family of functions $\phi_R$ convoluted with the
usual Poisson kernel of the upper half-space converge in a
suitable way to the same extension $V$ obtained by Abatangelo and
Ros-Oton in \cite{AbaRos}. Moreover, defining the extension starting from
\eqref{estensioneregolare}, we can easily ensure that
$V\in H^{1,1-2s}_{\rm loc}(\overline{\R^{N+1}_+})$ and that it is weak
solution to \eqref{eq:27}. Nevertheless, also $V$ inherits from $\phi$
an at most 
polynomial growth. Let us consider
$W=V-\Psi\in H^{1,1-2s}_{\rm loc}(\overline{\R^{N+1}_+})$, which
 weakly solves
\begin{equation*}
\begin{cases}
\mathop{\rm div}(t^{1-2s}\nabla W)=0 &\mathrm{in \ }\R^{N+1}_+,\\
\mathop{\rm Tr}W=0 &\mathrm{on \ }\R^{N}=\partial \R^{N+1}_+.
\end{cases}
\end{equation*}
Then, denoting as $\widetilde W$ the odd reflection of $W$ through
   $\R^{N}=\partial \R^{N+1}_+$,  by \cite[Proposition 2.10]{SirTerVit2} 
\begin{equation*}
   \overline W=\frac{\widetilde W}{t|t|^{2s-1}}\in
   H^{1,1+2s}_{\rm loc}(\R^{N+1})
 \end{equation*}
   is an even entire weak solution to
  \eqref{evenliouville} with $a=1+2s\in(1,3)$.
We 
have that $\overline W$ satisfies the assumptions of Lemma
\ref{Liouville}, 
being a polynomial bound on its growth ensured by the polynomial bounds of $\Psi$ and
$V$. From  Lemma
\ref{Liouville} we can
 promptly conclude that $\overline W$ is a polynomial.  We also have
 that 
\begin{equation*}
t^{1-2s}\partial_tV=t^{1-2s}\partial_t\Psi+t^{1-2s}\partial_t(t^{2s}\overline W)=t^{1-2s}\partial_t\Psi+P_k
\end{equation*}
for some polynomial $P_k$ of degree $k\in{\mathbb N}$. Hence, passing
to the trace of the weighted derivative above, by \cite[Lemma
3.3]{AbaRos} it follows that
\begin{equation*}
(-\Delta)^s\phi\mathop{=}\limits^{k+1}0\quad\text{in }\Gamma^-
\end{equation*}
and $\phi=0$ in $\Gamma^+$, where the above identity is meant in the
sense of the notion of  ``fractional Laplacian modulus polynomials of
degree at most $k$''  given in \cite[Definition 3.1]{AbaRos}, see also
\cite{DSV}. Hence, by \cite[Theorem 3.10]{AbaRos}, we have that
\begin{equation*}
\phi(x)=p(x)(x_N)_-^s,
\end{equation*}
for some polynomial $p$. By
homogeneity of $\phi$, this implies that necessary there exists
$j\in{\mathbb N}$ such that $\gamma=j+s$.
\end{proof}

We are now going to derive from Proposition \ref{gamma=k+s} the explicit formula \eqref{eq:28} for the
eigenvalues of problem \eqref{eig}. 
We first observe that, if $\mu$ is an eigenvalue of
\eqref{eig} with an associated eigenfunction $\psi$, then
 the function  
$\Psi(\rho\theta)=\rho^\sigma\psi(\theta)$
 with $\sigma=
-\frac{N-2s}{2}+\sqrt{\left(\frac{N-2s}{2}\right)^2+\mu}$ belongs to
$H^{1,1-2s}_{\rm loc}(\overline{\R^{N+1}_+})$  and is a
weak solution to \eqref{eq:27}. From Proposition  \ref{gamma=k+s} we
then deduce that there exists $j\in{\mathbb N}$ such that $\sigma=j+s$
and hence 
\begin{equation*}
\mu=(j+s)(j+N-s).
\end{equation*}
Viceversa, we prove now that all
numbers of the form $\mu=(j+s)(j+N-s)$ with
$j\in\mathbb{N}$  
are eigenvalues of \eqref{eig}. For any fixed $j\in\mathbb{N}$, we
consider the function $\Psi$ defined, in cylindrical coordinates, as  
\begin{equation*}
\Psi(x',r\cos \tau,r\sin \tau)=r^{s+j}\left|\sin\bigg(\frac
  \tau2\bigg)\right|^{2s}
{_2}F_1\bigg(-j,j+1;1-s;\frac{1+\cos\tau}2\bigg),\ r\geq 0,\  \tau\in [0,2\pi],
\end{equation*}
where ${_2}F_1$ is the hypergeometric function.
From \cite{RosSer3} we have that $\Psi\in H^{1,1-2s}_{\rm loc}(\overline{\R^{N+1}_+})$ is a
weak solution to \eqref{eq:27}. Furthermore $\Psi$ is homogeneous of
degree $s+j$ and therefore the function $\psi:=\Psi\big|_{{\mathbb
    S}^{N}_+}$ belongs to $\mathcal H_0$, $\psi\not\equiv0$, and 
\begin{equation*}
\Psi(\rho\theta)=\rho^{s+j}\psi(\theta),\quad \rho\geq0,\ \theta\in
{\mathbb S}^{N}_+.
\end{equation*}
Plugging the above characterization of $\Psi$ into  \eqref{eq:27}, we obtain that 
\begin{equation*}
\rho^{j-1-s}\Big((j+s)(j+N-s) \theta_{N+1}^{1-2s}\psi(\theta)+\mathrm{div}_{\mathbb S^N}\left(\theta_{N+1}^{1-2s}\nabla_{\mathbb S^N}\psi\right)\Big)=0,\quad \rho>0,\ \theta\in {\mathbb S}^{N}_+,
\end{equation*}
so that $(j+s)(j+N-s)$ is an eigenvalue of \eqref{eig}. 

We then conclude that the set of all eigenvalues of problem
\eqref{eig} is 
$\left\{(j+s)(j+N-s):\, j\in \mathbb{N}\right\}$.

\section*{Acknowledgments}
We thank Giorgio Tortone for kindly providing us with the reference \cite{AbaRos}.

\end{document}